%% file: main.tex
\documentclass[reqno, 11pt]{amsart}

\input{Preamble}

\input{specificPreamble}

\usepackage{svg-extract}

\newcommand{\exampleNoStronglyLinkedTD}{Example~2}
\newcommand{\exampleNoLeanTD}{Example~1}
\newcommand{\exampleNoUpwardsDisjointAdhesionSets}{Example~5.2}
\newcommand{\exampleNoTDEffDistAllEnds}{Lemma~3.2}
\newcommand{\constructionExample}{Construction~3.1}

\definecolor{cMaroon}{HTML}{93152a}
\newcommand{\defn}[1]{{\color{cMaroon}{\emph{#1}}}}

\title{Linked tree-decompositions into finite parts}
\author{Sandra Albrechtsen}

\author{Raphael~W.\ Jacobs}

\author{Paul Knappe}

\author{Max Pitz}

\address{Universität Hamburg, Department of Mathematics, Bundesstraße 55 (Geomatikum), 20146 Hamburg, Germany}
\email{\{sandra.albrechtsen, raphael.jacobs, paul.knappe, max.pitz\}@uni-hamburg.de}


\begin{document}

\begin{abstract}
We prove that every graph which admits a \td\ into finite parts has a rooted \td\ into finite parts that is linked, tight and componental. 

As an application, we obtain that every graph without half-grid minor has a lean \td\ into finite parts, strengthening the corresponding result by K{\v r}{\'i}{\v z} and Thomas for graphs of finitely bounded tree-width. In particular, it follows that every graph without half-grid minor has a \td\ which efficiently distinguishes all ends and critical vertex sets, strengthening results by Carmesin and by Elm and Kurkofka for this graph class.

As a second application of our main result, it follows that every graph which admits a \td\ into finite parts has a \td\ into finite parts that displays all the ends of $G$ and their combined degrees, resolving a question of Halin from 1977. This latter \td\ yields short, unified proofs of the characterisations due to Robertson, Seymour and Thomas of graphs without half-grid minor, and of graphs without binary tree subdivision.
\end{abstract}

\maketitle

\section{Introduction}\label{sec:Introduction}

\subsection{The main result}
Our point of departure is K{\v r}{\'i}{\v z} and Thomas's result on linked tree-decompositions, which forms a cornerstone both in Robertson and Seymour's work \cite{GMIV} on well-quasi-ordering finite graphs, and in Thomas's result \cite{thomas1989wqo} that the class of infinite graphs of tree-width~$< k$ is well-quasi-ordered under the minor relation for all~$k \in \N$. 

\begin{theorem}[{Thomas 1990 \cite{LeanTreeDecompThomas}, K{\v r}{\'i}{\v z} and Thomas 1991  \cite{kriz1991mengerlikepropertytreewidth}}]
\label{thm_intro_krizthomas}
    Every (finite or infinite) graph  of tree-width  $< k$ has a linked rooted tree-decomposition\footnote{There is also an unrooted version of this theorem 
    but this is not needed for the well-quasi-ordering applications.} of width $< k$.
\end{theorem}

To make this result precise, recall that a \defn{(rooted) \td} $(T, \cV)$ of a possibly infinite graph $G$ is given by a (rooted) \defn{decomposition tree} $T$ whose nodes $t$ are assigned \defn{bags} $V_t \subseteq V(G)$ or \defn{parts}~$G[V_t]$ of the underlying graph $G$ such that $\cV = (V_t)_{t \in T}$ covers $G$ in a way that reflects the separation properties of $T$: 
Similarly as the deletion of an edge $e=st$ from $T$ separates it into components $T_s \ni s$ and $T_t \ni t$, the corresponding sets $A^e_s = \bigcup_{x \in T_s} V_x$ and $A^e_t = \bigcup_{x \in T_t} V_x$ in the underlying graph $G$ are separated by the \defn{adhesion set} $V_e = V_{s} \cap V_{t}$. A graph $G$ has 
\begin{itemize}
    \item \defn{tree-width $< k$} if it admits a (rooted) \td\ into parts of size $\leq k$,
    \item \defn{finitely bounded tree-width} if it has tree-width $< k$ for some $k \in \N$, and
    \item \defn{finite tree-width} if it admits a (rooted) \td\ into finite parts (of possibly unbounded size).\footnote{Note that for a graph to have finite tree-width we do not require that the \td\ into finite parts also satisfies that $\liminf_{e \in E(R)} V_e$ is finite for all rays $R$ in $T$, as it is sometimes \cites{bibel,robertson1995excluding} done.}
\end{itemize}
Given a tree $T$ rooted at a node $r$, its \defn{tree-order} is given by $x \leq y$ for $x, y \in V(T) \cup E(T)$ if~$x$ lies on the (unique) $\subseteq$-minimal path \defn{$rTy$} from $r$ to $y$. 
For an edge $e = st$ of $T$ with $s < t$, the \defn{part above $e$} is $G \up e = G[A^e_t]$ and the \defn{part strictly above $e$} is $G \strictup e = G \up e - V_e$.
A rooted \td\ $(T, \cV)$ of a graph $G$ is
\begin{itemize}
    \item \defn{linked} if for every two comparable nodes $s < t$ of $T$ there are $\min\{|V_e| \colon e \in E(sTt)\}$ pairwise disjoint $V_s$--$V_{t}$ paths in $G$ \cite{thomas1989wqo},
    \item \defn{tight} if for every edge $e$ of $T$ some component $C$ of $G \strictup e$ satisfies $N(C) = V_e$, and
    \item \defn{componental} if for every edge $e$ of $T$ the part $G \strictup e$ strictly above $e$ is connected. 
\end{itemize} 
Our main result extends \cref{thm_intro_krizthomas} to graphs of (possibly unbounded) finite tree-width: 

\begin{mainresult}\label{main:LinkedTightCompTreeDecompnew}
    Every graph of finite tree-width admits a rooted \td\ into finite parts that is linked, tight, and componental.
\end{mainresult}

Note that achieving just a subset of the properties in \cref{main:LinkedTightCompTreeDecompnew} may be significantly easier. Recall that a \defn{normal spanning tree} of a graph $G$ is a rooted spanning tree $T$ such that the endvertices of every edge of $G$ are comparable in its tree-order. It is well known that the connected graphs of finite tree-width are precisely the graphs with normal spanning trees (see \cref{thm:FiniteTWyieldsNST} below). 
Given a normal spanning tree $T$ with root $r$, by assigning to each of its nodes~$t$ the bag $V_t := V(rTt)$ we obtain a rooted \td\ $(T,\cV_\textnormal{NT})$ into finite parts that is componental and linked, the latter albeit for the trivial reason that~$s < t \in T$ implies~$V_s \subseteq V_t$. However, this \td\ clearly fails to be tight.
Following an observation by Diestel \cite{diestel1994depth}, one can restore tightness by taking as bags the subsets $V'_t \subseteq V_t$ that consist only of those vertices in $V_t$ that send a $G$-edge to a vertex above $t$ in $T$. The new \td~$(T,\cV'_\textnormal{NT})$ is then tight and componental, but in general no longer linked. 

Having all three properties simultaneously is more challenging to achieve. 
In what follows, we hope to convince the reader of the  usefulness of \cref{main:LinkedTightCompTreeDecompnew} by demonstrating the surprisingly powerful interplay of the properties of being linked, tight, and componental.
We do so by presenting several applications in the following sections.

In \cite{ExamplesLinkedTDInfGraphs} we provide examples showing that \cref{main:LinkedTightCompTreeDecompnew} appears to lie at the frontier of what is still true. For example, \emph{linked} cannot be strengthened to its `unrooted' version that requires the \emph{linked} property between every pair of nodes and not just comparable ones \cite{ExamplesLinkedTDInfGraphs}*{\exampleNoStronglyLinkedTD}.

\subsection{Displaying end structure and excluding infinite minors}

As our first application, we show that every \td\ as in  \cref{main:LinkedTightCompTreeDecompnew}  displays the combinatorial and topological end structure of the underlying graph as follows, resolving a question by Halin from 1977 \cite{halin1977systeme}*{\S6}. (Halin's original work yields proofs of \cref{main:LinkedTightCompTreeDecompnew,main:LinkedTightCompTreeDecompnew2} for locally finite connected graphs with at most two ends.) 

\begin{mainresult}
\label{main:LinkedTightCompTreeDecompnew2}
Every graph $G$ of finite tree-width admits a rooted \td\ into finite parts that homeomorphically displays all the ends of $G$, their dominating vertices, and their combined degrees.
\end{mainresult}

Recall that an \defn{end} $\eps$ of a graph $G$ is an equivalence class of rays in $G$ where two rays are equivalent if for every finite set $X$ they have a tail in the same component of~$G-X$. We refer to this component of $G-X$ as $C_G(X, \eps)$. Write $\Omega(G)$ for the \defn{set of all ends} of $G$.
A vertex $v$ of $G$ \defn{dominates} an end $\eps$ of $G$ if it lies in $C_G(X, \eps)$ for every finite set $X$ of vertices other than $v$.
We denote the set of all vertices of $G$ which dominate an end $\eps$ by $\Dom(\eps)$.
The \defn{degree} $\deg(\eps)$ of an end $\eps$ of $G$ is the supremum over all cardinals $\kappa$ such that there exists a set of $\kappa$ pairwise disjoint rays in $\eps$, and its \defn{combined degree} is~$\Delta(\eps) := \deg(\eps) +|\Dom(\eps)|$.

In a rooted \td\ $(T,\cV)$ of a graph $G$ into finite parts, every end $\eps$ of $G$ \defn{gives rise} to a single rooted ray $R_\eps$ in $T$ that starts at the root and then always continues upwards along the unique edge $e \in T$ with $C_G(V_t, \eps) \subseteq G \strictup e$. 
This yields a map~$\phi \colon \Omega(G) \to \Omega(T),\; \eps \mapsto R_\eps$. 
A rooted \td~$(T, \cV)$
\begin{itemize}
\item  \defn{displays} the ends of $G$ if $\phi$ is a bijection  \cite{carmesin2019displayingtopends},
\item \defn{displays the ends homeomorphically} if $\phi$ is a homeomorphism \cites{koloschin2023end},
\item \defn{displays all dominating vertices} if $\liminf_{e \in E(R_\eps)} V_e = \Dom(\eps)$\footnote{The set-theoretic $\liminf_{n \in \N} A_n$ consists of all points that are contained in all but finitely many $A_n$. For a ray $R= v_0e_0v_1e_1v_1 \dots$ in $T$, one gets $\liminf_{e \in E(R_\eps)} V_e = \bigcup_{n \in \N} \bigcap_{i \geq n} V_{e_i}$.} for all $\eps \in \Omega(G)$, and
\item \defn{displays all combined degrees} if $\liminf_{e \in E(R_\eps)} |V_e| = \Delta(\eps)$ for all $\eps \in \Omega(G)$. 
\end{itemize}
Roughly, every componental, rooted \td\ homeomorphically displays all the ends \cite{koloschin2023end}*{Lemma~3.1}, every tight such \td\ displays all dominating vertices (\cref{lem:TreeDecompDisplayingDominatingVertices}), and every such linked tight \td\ displays all combined degrees (\cref{lem:linkedandliminfareonlydominatingverticesimpliesdisplayenddegrees}).

Diestel's \td\ $(T,\cV'_\textnormal{NT})$ mentioned above was not linked, but still displays all the ends and their dominating vertices. This observation allowed Diestel \cite{diestel1994depth} to obtain a short proof of Robertson, Seymour and Thomas's characterisation that graphs without subdivided infinite cliques are precisely the graphs that admit a  \td\ $(T,\cV)$ into finite parts such that for every ray $R$ of $T$, the set $\liminf_{e \in E(R)}{V_e}$ is finite.
Extending this idea, we now use \cref{main:LinkedTightCompTreeDecompnew2} to provide short, unified proofs for two other related results by Robertson, Seymour and Thomas.

\begin{corollary}[Robertson, Seymour \& Thomas 1995 \cite{robertson1995excluding}*{(2.6)}] \label{cor:HalfGridMinor}
    A graph contains no half-grid minor if and only if it admits a \td\ $(T,\cV)$ into finite parts such that for every ray~$R$ of~$T$ we have $\liminf_{e \in E(R)}{|V_e|} < \infty$. 
\end{corollary}

\begin{proof}
    We prove here only the hard implication\footnote{\label{note:Intro} See \cites{robertson1995excluding,seymour1993binarytree} for details about the `easy' implications.}.  So assume that $G$ is a graph without half-grid minor. By a result of Halin \cites{halin78,pitz2020unified}, every connected graph without a subdivided $K^{\aleph_0}$ has a normal spanning tree. Applying this to every component of $G$ we see that $G$ has finite tree-width. So we may apply \cref{main:LinkedTightCompTreeDecompnew2}.
    
    By Halin's grid theorem \cites{halin65,kurkofka2022strengthening}, every end of infinite degree contains a half-grid minor. And by a routine exercise, every end dominated by infinitely many vertices contains a subdivided $K^{\aleph_0}$. So every end of $G$ has finite combined degree. 
    Let $(T,\cV)$ be the \td\ into finite parts from \cref{main:LinkedTightCompTreeDecompnew2}. Then $\liminf_{e \in R}{|V_e|}$ equals the combined degree of the end of~$G$ giving rise to $R$, and hence is finite, as desired.
\end{proof}

\begin{corollary}[Seymour \& Thomas 1993 \cite{seymour1993binarytree}*{(1.5)}] \label{cor:T2Subdivision}
    A graph contains no subdivision of $T_2$ if and only if it admits a \td\ $(T, \cV)$ into finite parts such that for every ray $R$ in $T$ we have $\liminf_{e \in E(R)}{|V_e|} < \infty$ 
    and $T$ contains no subdivision of the binary tree $T_2$.
\end{corollary}

\begin{proof}
    Again, we only prove the hard implication. 
    Assume that $G$ is a graph without a subdivision of $T_2$. In particular, $G$ contains no half-grid minor. 
    As above, $G$ has finite tree-width and every end of $G$ has finite combined degree.
    Let $(T,\cV)$ be the \td\ into finite parts from \cref{main:LinkedTightCompTreeDecompnew2}. 
    Then $\liminf_{e \in E(R)}{|V_e|}$ equals the combined degree of the end of~$G$ giving rise to $R$, and hence is finite.

    We complete the proof by showing that $T$ contains no subdivided binary tree. 
    Recall Jung's characterisation that a graph $G$ contains no end-injective\footnote{A rooted subtree $T$ of a graph $G$ is \defn{end-injective} if distinct rooted rays in $T$ belong to distinct ends of $G$.} subdivided $T_2$ if and only if the end space of $G$ is scattered\footnote{The precise definition of `scattered' is not relevant here, it is only important that `scattered' is a property of topological spaces and hence preserved under homeomorphisms.} 
    \cite{jung1969wurzelbaume}*{\S3}. 
    Hence, if $G$ contains no subdivided binary tree, then its end space is scattered. Since $(T,\cV)$ displays the ends of $G$ homeomorphically, we conclude that the end space of $T$ is scattered, so $T$ contains no end-injective subdivided $T_2$, by the converse of Jung's result.
    As all subtrees of trees are end-injective, this yields the desired result.
\end{proof}

Note that \cref{main:LinkedTightCompTreeDecompnew} shows that one can always require the witnessing  \td s for \cref{cor:HalfGridMinor,cor:T2Subdivision} to be linked, tight, and componental.

\subsection{Lean \td s} 
\label{sec_leanIntro}

We already mentioned that there exist even stronger versions of the K{\v r}{\'i}{\v z}-Thomas \cref{thm_intro_krizthomas}. See the articles by Bellenbaum and Diestel \cite{bellenbaum2002two} and by Erde~\cite{erde2018unified} for a modern treatment of the finite case.

\begin{theorem}[{Thomas 1990 \cite{LeanTreeDecompThomas}, K{\v r}{\'i}{\v z} and Thomas 1991  \cite{kriz1991mengerlikepropertytreewidth}}]
\label{thm_intro_krizthomas2}
    Every (finite or infinite) graph  of tree-width  $< k$ has a lean \td\ of width $< k$.
\end{theorem}

Here, a \td\ of a graph $G$ is \defn{lean} if for every two (not necessarily distinct) nodes~$t_1,t_2 \in T$ and vertex sets $Z_1 \subseteq V_{t_1}$ and $Z_2 \subseteq V_{t_2}$ with $|Z_1| = |Z_2| =: \ell \in \N$, either $G$ contains~$\ell$ pairwise disjoint $Z_1$--$Z_2$ paths or there exists an edge $e \in t_1Tt_2$ with $|V_e| < \ell$.
There are two ways in which `lean' is stronger than `linked': First, every lean  \td~$(T,\cV)$ satisfies that for \emph{every} two nodes $s  \neq t \in T$ there are~$\min\{|V_e| \colon e \in E(sTt)\}$ pairwise disjoint $V_s$--$V_{t}$ paths in $G$. In a way, this is the unrooted version of the property `linked' as introduced above which only required this property when $s$ and $t$ are comparable.
The other difference between `linked' and `lean' is that a bag of a lean \td\ is only large if it is highly connected, which follows by considering the case $t_1 = t_2$ in the definition of `lean'. 

Now it is natural to ask whether \cref{thm_intro_krizthomas2} extends to graphs of finite tree-width. Given that the \td\ $(T,\cV_\textnormal{NT})$ from above is linked but not necessarily lean, the following question appears to be non-trivial:

\begin{center}
    \mbox{\textit{Does every graph of finite tree-width admit a lean \td\ into finite parts?}}
\end{center}

Unfortunately, the answer to this question is in the negative. 
In \cite{ExamplesLinkedTDInfGraphs}*{\exampleNoLeanTD} we construct a locally finite and planar graph which does not admit any lean \td.
In particular, not even graphs without $K^{\aleph_0}$ minor admit lean \td s.

Yet, we can extend \cref{thm_intro_krizthomas2} in the optimal way to graphs without half-grid minor by post-processing the \td\ from \cref{main:LinkedTightCompTreeDecompnew} using the finite case of \cref{thm_intro_krizthomas2}.

\begin{mainresult} \label{main:LeanTD}
    Every graph~$G$ without half-grid minor admits a lean \td\ into finite parts.
    Moreover, if the tree-width of $G$ is finitely bounded, then the lean \td\ can be chosen to have width $\tw(G)$.
\end{mainresult}

\noindent The first part of \cref{main:LeanTD} is a new result, while the `moreover'-part reobtains the infinite case in \cref{thm_intro_krizthomas2}. 
We recall that Robertson, Seymour and Thomas's strongest version \cite{robertson1995excluding}*{(12.11)} of \cref{cor:HalfGridMinor} yields witnessing \td s which are additionally linked (even in its unrooted version).
An immediate application of our detailed version of \cref{main:LeanTD} (see \cref{thm:LeanTDTechnical} in \cref{sec:LeanTreeDecs}) strengthens their result by showing that the witnessing \td\ in \cref{cor:HalfGridMinor} may even be chosen to be lean.

\subsection{Displaying all infinities}

A crucial tool for the proof of \cref{main:LinkedTightCompTreeDecompnew} are `critical vertex sets', the second kind of infinities besides ends:
A set $X$ of vertices of $G$ is \defn{critical} if there are infinitely many \defn{tight} components of $G-X$,
that is components $C$ of $G-X$ with $N_G(C) = X$. 
Polat \cite{PolatEME1} already noted that an infinite graph always contains either a ray/end or a critical vertex set.

As an intermediate step in the proof of \cref{main:LinkedTightCompTreeDecompnew} we obtain a similar \td\ which `displays all the infinities' of the underlying graph: 
A rooted \td~$(T, \cV)$
\begin{itemize}
\item \defn{displays the critical vertex sets} if the map $t \mapsto V_t$ restricted to the infinite-degree nodes of~$T$ whose $V_t$ is finite is a bijection to the critical vertex sets. In this context, we denote for every critical vertex set $X$ the unique infinite-degree node of $T$ whose bag is $X$ by $t_X$. 
\item \defn{displays the tight components of every critical vertex set cofinitely} if it displays the critical vertex sets such that for every critical vertex set $X$ every $G \strictup e$ with $e =t_Xt \in T$ with $t_X <_T t$ is a tight component of $G-X$ and cofinitely many tight components of $G-X$ are some such $G \strictup e$, then $(T,\cV)$.
\item \defn{displays the infinities} if it displays the ends homeomorphically, their combined degree, their dominating vertices, the critical vertex sets and their tight components cofinitely.
\end{itemize}
A \td\ that cofinitely displays the tight components of every critical vertex set can no longer in general be componental; 
but we still ensure that it is \defn{cofinally componental}, i.e.\,along every rooted ray of $T$ there are infinitely many edges $e$ such that $G \strictup e$ is connected.

\begin{mainresult}\label{maincor:TreeDecompDisplayingInfs}
    Every graph of finite tree-width admits a linked, tight, cofinally componental, rooted \td\ into finite parts which displays the infinities.
\end{mainresult}

Further, a \td\ $(T, \cV)$ \defn{efficiently distinguishes all the ends and critical vertex sets} if for any pair of ends and/or critical vertex sets one of the adhesion sets of $(T, \cV)$ is a size-wise minimal separator for them.
Carmesin \cite{carmesin2019displayingtopends}*{Theorem~5.12} showed that for every graph there is a nested set of separations which efficiently distinguishes all its ends.
Elm and Kurkofka 
\cite{infinitetangles}*{Theorem~1} extended this result by showing that there always exists a nested set of separations
which efficiently distinguishes all ends and all critical vertex sets. If one aims not only for a nested set of separations but for a \td, then the best current result is that every locally finite graph without half-grid minor admits even a \td\ which distinguishes all its ends
efficiently \cite{jacobs2023efficiently}*{Theorem~1}.
In \cite{ExamplesLinkedTDInfGraphs}*{\constructionExample} we present a planar graph witnessing that this result cannot be extended to graphs without $K^{\aleph_0}$ minor \cite{ExamplesLinkedTDInfGraphs}*{\exampleNoTDEffDistAllEnds}, even if they are locally finite.
Yet we extend it from locally finite graphs to arbitrary infinite graphs without half-grid minor:

\begin{maincorollary} \label{main:ToTIntroVersion}
    Every graph~$G$ without half-grid minor admits a \td\ of finite adhesion which distinguishes all its ends and critical vertex sets efficiently. 
\end{maincorollary}

\noindent We obtain \cref{main:ToTIntroVersion} as an application of \cref{main:LeanTD}: Our proof of \cref{main:LeanTD} yields that the obtained \td\ also displays the infinities of $G$ and hence in particular distinguishes them. Since the \td\ is lean, one can conclude that it even does so efficiently, and thus is as desired for \cref{main:ToTIntroVersion}. 
\medskip

An equivalent way of stating \cref{main:ToTIntroVersion} is that every graph $G$ without half-grid minor admits a \td\ of finite adhesion which efficiently distinguishes all its `combinatorially distinguishable infinite tangles'\footnote{See \cref{subsec:ToT} for definitions.}. This fits into a series of results \cites{infinitesplinter,carmesin2022canonical,infinitetangles,jacobs2023efficiently} extending the `tree-of-tangles' theorem from Robertson and Seymour \cite{GMX}*{(10.3)} to infinite graphs.  

Additionally, the first author \cite{SATangleTreeDualityInfGraphs}*{Theorem~3} applies \cref{maincor:TreeDecompDisplayingInfs} to obtain a `tangle-tree duality' theorem for infinite graphs\footnote{Every graph with no principal $k$-tangle not induced by an end admits a \td\ witnessing this.}, which extends Robertson and Seymour's \cite{GMX}*{(4.3) \& (5.1)} other fundamental theorem about tangles to infinite graphs.

\subsection{An open problem}

Thomas famously conjectured that the class of countable graphs is well-quasi-ordered under the minor relation \cite{thomas1989wqo}*{(10.3)}. In light of the observation that all known counterexamples to well-quasi-orderings of infinite graphs \cites{thomas1988counter,komjath1995note,diestel2001normal,pitz2023note} do not have finite tree-width by the characterisation of finite tree-width graphs (i.e.\ graphs with normal spanning trees, \cref{thm:FiniteTWyieldsNST} below) in \cite{pitz2021proof}, it might be interesting to also consider the following, even stronger conjecture:

\begin{conjecture}
    The  graphs of finite tree-width are well-quasi-ordered under the minor relation.
\end{conjecture}

We remark that already the simplest case of the conjecture -- namely for the class of infinite graphs where all components are finite -- is open.

\subsection{How this paper is organised}

In \cref{sec:Preliminaries} we give a short introduction into \td s, ends and critical vertex sets.
We show in \cref{subsec:DisplayingEnds} that the \td\ from \cref{main:LinkedTightCompTreeDecompnew} already satisfies \cref{main:LinkedTightCompTreeDecompnew2}.
\cref{sec:MainResultCounterexamples} consists of three parts. First, we give in \cref{subsec:ATwoStepApproach} a brief overview over the proof of \cref{main:LinkedTightCompTreeDecompnew}, which includes the statement of \cref{thm:critVtxIntro,thm:RaylessThmIntro}, our two main ingredients to the proof of \cref{main:LinkedTightCompTreeDecompnew}. In \cref{subsec:detailedversions}, we state detailed versions of \cref{main:LinkedTightCompTreeDecompnew,maincor:TreeDecompDisplayingInfs} which assert further properties of the respective \td s. In \cref{subsec:ProofOfMainResult}, we then reduce \cref{main:LinkedTightCompTreeDecompnew,maincor:TreeDecompDisplayingInfs} to \cref{thm:critVtxIntro,thm:RaylessThmIntro}.
We prove \cref{thm:critVtxIntro} in \cref{sec:TreeDecompAlongCrit}.
In \cref{sec:liftingspathsandraysfromtorso}, we show some basic behaviour of paths and rays in torsos to prove \cref{thm:RaylessThmIntro} in \cref{sec:DecomposeIntoRaylessParts}.
Finally, in \cref{sec:LeanTreeDecs}, we prove \cref{main:LeanTD} and \cref{main:ToTIntroVersion}.
\arXivOrNot{Furthermore, we compare in \cref{subsec:Counterexamples} our \cref{thm:critVtxIntro} to Elm and Kurkofka's \cite{infinitetangles}*{Theorem~2}.}{}

\section{Preliminaries} \label{sec:Preliminaries}

In this section, we gather all concepts needed for the remainder of the paper. We repeat here concepts already defined in the introduction, in the hope that it may be convenient for future reference to have all definitions gathered in one place.

Throughout the paper, graphs may be infinite.
The notation and definitions follow \cite{bibel} unless otherwise specified; in particular, $\N = \{0,1,2, \dots\}$, and we may speak of a vertex $v \in G$ (rather than $v \in V(G)$), an edge $e \in G$, and so on. 
\medskip

In this paper, a tree $T$ often contains a special vertex $\rt(T)$, its \defn{root}.
A rooted tree~$T$ has a natural partial order~\defn{$\le_T$} on its vertices and edges, which depends on~$r = \rt(T)$:
for two~$x, y \in V(T) \cup E(T)$, we write~$xTy$ for the unique $\subseteq$-minimal path in $T$ which contains~$x$ and~$y$, and we then set~$t \leq_T t'$ if $t \in rTt'$.
In a rooted tree, a \defn{leaf} is any maximal node in the tree-order.
The \defn{down-closure} $\lceil x \rceil $ and \defn{up-closure} $\lfloor x  \rfloor$ of $x$ in $T$ are $\{y \in V(T) \mid y \leq x\}$ and $\{y \in V(T) \mid y \geq x\}$, respectively.
We write $\mathring{\lceil x \rceil}$ and $\mathring{\lfloor x  \rfloor}$ as shorthand for $\lceil x \rceil \setminus x$ and $\lfloor x  \rfloor \setminus x$.

A rooted tree $T$ in a graph $G$ is \defn{normal} if the endvertices of every $T$-path in $G$ are $\leq_T$-comparable.
\medskip

An \defn{(oriented) separation} of a graph~$G$ is a tuple~$(A, B)$ of vertex sets~$A, B$ of~$G$, its \defn{sides}, such that~$A \cup B = V(G)$ and there are no edges in~$G$ joining~$A \setminus B$ and~$B \setminus A$.
Its \defn{separator} is~$A \cap B$, and its \defn{order} is the size of its separator.
Note that if~$(A, B)$ is a separation of~$G$, then so is~$(B, A)$.

A \defn{star} of separations is a set~$\sigma$ of separations of~$G$ such that~$A \subseteq D$ and~$B \subseteq C$ for every two distinct~$(A, B), (C, D) \in \sigma$.
The \defn{interior} of a star~$\sigma$ is~$\interior(\sigma) := \bigcap_{(A,B) \in \sigma} B$.

For a star~$\sigma$ of separations of~$G$, we let~\defn{$\torsostar(\sigma)$} be the graph arising from~$G[\interior(\sigma)]$ by making each separator of a separation in~$\sigma$ complete. We call these added edges that lie in $\torsostar(\sigma)$ but not in $G$ \defn{torso edges}.
We refer to~$\torsostar(\sigma)$ as the \defn{torso at}~$\sigma$.

\subsection{Tree-decompositions}

A \defn{\td} of a graph~$G$ is a pair~$(T, \cV)$ that consists of a tree~$T$ and a family~$\cV = (V_t)_{t \in T}$ of vertex sets of~$G$ indexed by the nodes of~$T$ and satisfies the following two conditions:
\begin{enumerate}[label=(T\arabic*)]
    \item \label{prop:TD1} $G = \bigcup_{t \in T} G[V_t]$,
    \item \label{prop:TD2} for every vertex $v \in G$, the subgraph of $T$ induced by $\{t \in T \mid v \in V_t\}$ is connected.
\end{enumerate}
The sets $V_t$ are the \defn{bags} of the \td, the induced subgraphs $G[V_t]$ on the bags are its \defn{parts}, and $T$ is its \defn{decomposition tree}.
Whenever a \td\ is introduced as $(T, \cV)$ in this paper, we tacitly assume that $\cV = (V_t)_{t \in T}$. A \td\ $(T, \cV)$ is \defn{rooted} if its decomposition tree $T$ is rooted.

If $(T, \cV)$ is a \td\ of a graph $G$, then a tree $T'$ obtained from $T$ by edge-contractions \defn{induces} the \td\ $(T', \cV')$ of $G$ whose bags are $V'_t = \bigcup_{s \in t} V_s$ for every~$t \in T'$, where we denote the vertex set of $T'$ as the set of branch sets, that is the $\subseteq$-maximal subtrees of $T$ consisting of contracted edges.
Whenever $t \in T'$ is a subtree of $T$ on a single vertex~$s$, we may reference to $t$ by $s$, as well.
If $T$ is a rooted tree, then $\rt(T')$ is the node of $T'$ containing $\rt(T)$.

In a \td\ $(T, \cV)$ of $G$ every (oriented) edge~$\ve = (t_0, t_1)$ of the decomposition tree~$T$ induces a separation of~$G$ as follows:
Write~$T_0$ for the component of~$T - e$ containing~$t_0$ and~$T_1$ for the one containing~$t_1$.
Then~$(\bigcup_{t \in T_0} V_t, \bigcup_{t \in T_1} V_t)$ is a separation of~$G$ \cite{bibel}*{Lemma 12.3.1}. We say that~$(\bigcup_{t \in T_0} V_t, \bigcup_{t \in T_1} V_t)$ is \defn{induced} by~$\ve$, or more generally by~$(T, \cV)$.
Its separator is~$V_e := V_{t_0} \cap V_{t_1}$, the \defn{adhesion set} of~$(T, \cV)$ corresponding to~$e$, where~$e$ is the undirected edge of~$T$ underlying~$\ve$.

For every node~$t \in T$, we write~\defn{$\sigma_t$} for the set of separations of~$G$ induced by the oriented edges~$\ve = (s, t)$ for $s \in N_T(t)$.
It is easy to see that such $\sigma_t$ are stars of separations and that their interior is precisely $V_t$.
We refer to~$\torsostar(\sigma_t)$ as the \defn{torso} of~$(T, \cV)$ \defn{at}~$t$.

The \defn{leaf separations} of a \td~$(T, \cV)$ are those separations of~$G$ that are induced by the oriented edges~$(s, t)$ of~$T$ where~$t$ is a leaf of~$T$.
We refer to this separation also as leaf separation \defn{at} $t$.

Let $(T, \cV)$ be a rooted \td.
Given an edge~$e = t_0 t_1$ of~$T$ with~$t_0 <_T t_1$, we abbreviate the sides of its induced separation by~$G \down e := G \left[ \bigcup_{t \in T_0} V_t \right]$ and~$G \up e := G \left[ \bigcup_{t \in T_1} V_t \right]$.
Further, we write~$G \strictdown e := G \down e - V_e$ and~$G \strictup e := G \up e - V_e$.
For a node $t \in T$, we set $G \down t := G \down e$, $G \strictdown t := G \strictdown e$, $G \up t := G \up e$ and $G \strictup t := G \strictup e$ where $e = st$ is the unique edge with $s <_T t$.
It is easy to see that $G \up x \supseteq G \up y$, $G \strictup x \supseteq G \strictup y$, $G \down x \subseteq G \down y$ and $G \strictdown x \subseteq G \strictdown y$ for every two nodes or edges $x, y \in T$ with $x \leq_T y$, and also~$\bigcap_{e \in R} G\strictup e = \emptyset$ for every rooted ray~$R$ in~$T$.

A rooted \td\ $(T, \cV)$ is \defn{componental} if $G \strictup e $ is connected for every edge $e \in T$, and it is \defn{cofinally componental} if $G \strictup e$ is connected for cofinally many edges $e$ of every $\subseteq$-maximal $\leq_T$-chain\footnote{These are the $\subseteq$-maximal rooted paths or rays.} in $T$. It is \defn{tight} if, for every edge $e \in T$, there is a component $C$ of $G \strictup e$ with $N_G(C) = V_e$, and if additionally all components $C$ of $G \strictup e$ satisfy $N_G(C) = V_e$, then $(T, \cV)$ is \defn{fully tight}.
Every tight, componental, rooted \td\ is fully tight.
A rooted \td\ $(T,\cV)$ is \defn{linked} if for every two edges $e \leq_T e'$ of $T$, there is an edge $f \in T$ with $e \leq_T f \leq_T e'$ and a family $\{P_v \mid v \in V_f\}$ of pairwise disjoint $V_{e}$--$V_{e'}$~paths, or equivalently $(G \down e)$ -- $(G \up e')$~paths, in~$G$ such that $v \in P_v$.
In particular, the size of the family of pairwise disjoint paths equals the size of~$V_f$.
Given a set~$X$ of vertices of~$G$, the rooted \td~$(T, \cV)$ of~$G$ is~\defn{$X$-linked} if~$X \subseteq V_{\rt(T)}$ and if for every edge $e \in T$ there exists an edge~$f \leq_T e$ and a family $\{P_v \mid v \in V_f\}$ of pairwise disjoint $X$--$V_{e}$~paths in~$G$ such that~$v \in P_v$.

\begin{lemma}\label{lem:BagsEqualXAreConnectedInT}
    Let $(T, \cV)$ be a rooted \td\ of a graph $G$ and $X \subseteq V(G)$.
    Further let $e,f \in E(T)$, $t \in V(T)$ with $t \leq_T e$ and $f$ is the unique edge in $tTe$ incident with $t$.
    If $V_e = X \subseteq V_t$, $G \strictup e$ is non-empty and $G \strictup f$ is connected, 
    then $V_s = V_t$ for every node $s \in tTe$ other than $t$.
\end{lemma}

\begin{proof}
    Suppose for a contradiction that there exists a node $s \in tTe$ other than $t$ with $V_s \neq X$.
    Let $s$ be a $\leq_T$-minimal such node.
    By \cref{prop:TD2}, $X \subseteq V_s$.
    Thus, there exists some $v \in V_s \setminus X$.
    Then $v \in G \strictup f \setminus G \strictup e$, since $V_e = X$.
    Moreover, $G \strictup e \subseteq G \strictup f$, and $G \strictup f$ is connected by assumption.
    So there is a $v$--$G \strictup e$ path in $G \strictup f$, because $G \strictup f$ is non-empty; in particular, this path avoids $X$, as it avoids $V_f \supseteq X$. So it also avoids $N_G(G \strictup e) \subseteq X$, which is a contradiction as $v \notin G \strictup e$.
\end{proof}

If all adhesion sets of a \td~$(T, \cV)$ are finite, then~$(T, \cV)$ has \defn{finite adhesion}.
A graph $G$ has \defn{tree-width} less than $k$ if it admits a \td\ whose parts all have size at most $k$.
If there exists such a minimal $k \in \N$ we denote it by $\tw(G)$, and say it has \defn{finitely bounded tree-width}.
We say that a graph has \defn{finite tree-width} if it admits a \td\ into finite parts.

The following result shows that the finite tree-width graphs are essentially the graphs with normal spanning trees; the latter have been heavily investigated and are by now well-understood. In particular, we have Jung's \emph{Normal Spanning Tree Criterion} \cite{jung1969wurzelbaume}*{Satz~6'}, that a connected graph $G$ admits a normal spanning tree if its vertex set is a countable union of dispersed sets; here a set of vertices $U$ in $G$ is \defn{dispersed} if for every end $\eps$ of $G$ there is some finite set $X \subseteq V(G)$ such that $U$ is disjoint from $C_G(X,\eps)$ (see the next subsection for a background on ends).

\begin{theorem} \label{thm:FiniteTWyieldsNST}
A graph has finite tree-width if and only if each of its components admits a normal spanning tree.
\end{theorem}

\begin{proof}
    Assume first that each component $C$ of a given graph $G$ admits a normal spanning tree~$T_C$. This induces a \td\ $(T_C,\cV_C)$ of $C$ into finite parts, given by assigning each vertex $t \in T_C$ its down-closure $\lceil t \rceil_{T_C}$ as bag. Consider the tree $T = \{r\} \sqcup \bigsqcup_C T_C$ where the new root $r$ is adjacent to each $\rt(T_C)$. If we assign to $r$ the empty bag and keep all other bags, then we get a tree-decomposition of $G$ into finite parts as desired.

    Conversely, let $(T, \cV)$ be a \td\ of a given graph $G$ into finite parts. We may assume that $G$ is connected.
    Since the parts of $(T, \cV)$ are finite, every end $\eps$ of $G$ gives rise to a rooted ray $R^\eps = v_0,e_0,v_1,e_1,v_2, \dots$ in $T$.
    We claim that the union $U_i$ of the bags assigned to nodes of the decomposition on a fixed level $i$ is dispersed: Indeed, for every end~$\eps$ of~$G$, the finite adhesion set $V_{e_i}$ separates $U_i$ from $C_G(V_{e_i},\eps)$.
    As~$V(G) = \bigcup_{i \in \N} U_i$, we are done by Jung's normal spanning tree criterion. 
\end{proof}

\subsection{Ends and tree-decompositions}\label{subsec:endandtd}

An \defn{end} of a graph~$G$ is an equivalence class of rays in~$G$, where two rays are equivalent if they are joined by infinitely many disjoint paths in~$G$ or, equivalently, if for every finite set~$X \subseteq V(G)$ both rays have tails in the same component of~$G - X$.
The set of all ends of a graph~$G$ is denoted by~\defn{$\Omega(G)$}.

For every finite set~$X \subseteq V(G)$ and every end~$\eps$ of~$G$, there is a unique component of~$G - X$ which contains a tail of some, or equivalently every, ray in~$\eps$; we denote this component by~$C_G(X, \eps)$ and say that~$\eps$ \defn{lives} in~$C_G(X, \eps)$.
We denote by $\Omega_G(X, \eps)$ all end of $G$ which live in $C_G(C, \eps)$.
Now the collection of all the $\Omega_G(X, \eps)$ with finite $X \subseteq V(G)$ and $\eps \in \Omega(G)$ form a basis for a topology of $\Omega(G)$.

A vertex $v$ of $G$ \defn{dominates} an end $\eps$ of $G$ if $v \in C_G(X,\eps)$ for every finite $X \subseteq V(G - v)$.
We write~\defn{$\Dom(\eps)$} for the set of all vertices of $G$ which dominate the end $\eps$. 
For a connected subgraph~$C$ of $G$ with finite $N_G(C)$, we write $\Dom(C)$ for all vertices which dominate some end of $G$ that lives in $C$.
Note that $\Dom(C)$ is a subset of $V(C) \cup N_G(C)$.

The \defn{degree} of an end~$\eps$ of~$G$ is defined as
\begin{equation*}
    \deg(\eps) := \sup \{ |\cR| \mid \cR \text{ is a family of disjoint rays in } \eps \}.
\end{equation*}
Halin \cite{halin65}*{Satz~1} showed that this supremum is always attained.
Together with the number~$\dom(\eps) := |\Dom(\eps)|$ of vertices dominating~$\eps$, this sums up to the \defn{combined degree}~$\Delta_G(\eps) := \deg(\eps) + \dom(\eps)$ of~$\eps$.
Every end $\eps$ of a graph with normal spanning tree has countable combined degree: By \cite{bibel}*{Lemma~8.2.3} $\deg(\eps)$ is countable, and by \cite{bibel}*{Lemma~1.5.5}, every vertex in $\Dom(\eps)$ lies on the unique normal ray in $\eps$, so~$\dom(\eps)$ is countable, too.

Every end~$\eta$ of a rooted tree~$T$ contains precisely one \defn{rooted ray}~$R$ in~$T$, i.e.\ a ray in~$T$ that starts in~$\rt(T)$; we will frequently use this one-to-one correspondence between the ends of~$T$ and the rooted rays in~$T$.
Let $(T, \cV)$ be a rooted \td~$(T, \cV)$ of a graph $G$ with finite adhesion, an end $\eps$ of $G$ \defn{lives in} an end $\eta$ of $T$ if some, or equivalently every, ray in $\epsilon$ has a tail in $G \strictup e$ for every edge $e$ of the unique rooted ray~$R$ in~$\eta$; with the above relation between the rooted rays in~$T$ and the ends of~$T$ in mind, we also say that~$\eps$ \defn{gives rise to} $R$ and $R$ \defn{arises from}~$\eps$.
We remark that every end of $G$ gives rise to at most one rooted ray in $T$, since $(T, \cV)$ has finite adhesion.
If every end of $G$ gives rise to some rooted ray of $T$, we encode this correspondence between the ends of~$G$ and~$T$ by a map~$\varphi \colon \Omega(G) \to \Omega(T)$.
In particular, this is the case if every torso of $(T, \cV)$ is rayless.
If~$\varphi$ is a bijection between the ends of~$G$ and the ends of~$T$, then we say that the rooted \td~$(T, \cV)$ \defn{displays all ends} of~$G$.
If $\varphi$ is a homeomorphism from~$\Omega(G)$ to $\Omega(T)$, then $(T, \cV)$ \defn{homeomorphically displays all ends} of $G$.

For a sequence~$(V_i)_{i \in \N}$ of sets, we set~\defn{$\liminf_{i \in \N} V_i$} $:= \bigcap_{i \in \N} \bigcup_{j \ge i} V_i$.
If the sequence of sets~$V_{e_i}$ is indexed by the edges of a ray~$R = v_0 e_0 v_1 e_1 \dots$ in a graph~$G$, then we also write~\defn{$\liminf_{e \in R} V_e$} for~$\liminf_{i \in \N} V_{e_i}$. 
If~$(T, \cV)$ displays the ends of $G$ and additionally has the property that the unique rooted ray~$R$ in~$T$ which arises from the end~$\eps$ of~$G$ satisfies~$\liminf_{e \in R} |V_e| = \Delta_G(\eps)$, then the \td~$(T, \cV)$ of~$G$ \defn{displays the combined degrees of every end}.
If a rooted \td\ $(T, \cV)$ of $G$ displays all ends of $G$ and additionally $\liminf_{e \in R} V_e = \Dom(\eps)$ for every rooted ray $R$ of $T$ and its arising end $\eps$ of $G$, then $(T, \cV)$ \defn{displays the dominating vertices of every end}.

Let $X$ and $Y$ be two sets of vertices in a graph $G$.
Then~$X$ is \defn{linked to}~$Y$ in~$G$ if there is a family $\{R_x \mid x \in X\}$ of pairwise disjoint~$X$--$Y$~paths in~$G$ such that $x \in R_x$.
Let $\eps$ be an end of~$G$.
An \defn{$X$--$\eps$ ray} in $G$ is a ray which starts in $X$ and is contained in $\eps$.
An \defn{$X$--$\eps$ path} in $G$ is an $X$--$\Dom(\eps)$ path in $G$.
A finite set $S \subseteq V(G)$ is an \defn{$X$--$\eps$ separator} in $G$ if $C_G(S,\eps)$ does not meet~$X$.  
The set $X$ is \defn{linked to} an end $\eps$ of~$G$ if there is a family $\{R_x \mid x \in X\}$ of pairwise disjoint $X$--$\eps$ paths and rays such that $x \in R_x$. 
Further, a rooted \td\ $(T,\cV)$ is \defn{end-linked} if for every edge $e$ of $T$ there exists some end $\eps$ of~$G$ which lives in $G \up e$ and to which~$V_e$ is linked.
The following lemma is clear.

\begin{lemma}\label{lem:endlinkedimpliestight}
    Every end-linked, rooted \td\ $(T, \cV)$ is tight. \qed
\end{lemma}

\subsection{Critical vertex sets and tree-decompositions}\label{subsec:critvertexsets}

Given a set $X$ of vertices of a graph~$G$, a component~$C$ of $G-X$ is \defn{tight at $X$ in $G$} if $N_G(C) = X$. By slight abuse of notation, we will refer to such $C$ as \defn{tight} components of $G-X$.
We write~$\cC_X := \cC(G-X)$ for the set of components of $G - X$ and $\breve{\cC}_X := \breve{\cC}(G-X) \subseteq \cC_X$ for the set of all tight components $C$ of $G - X$.
A \defn{critical vertex set} $X$ of $G$ is a finite set $X \subseteq V(G)$ such that the set~$\breve{\cC}_X$ is infinite.
We denote by \defn{$\crit(G)$} the set of all critical vertex sets of $G$.
A graph is \defn{tough} if it has no critical vertex set, or equivalently, if deleting finitely many vertices never leaves infinitely many components.

As two vertices in a critical vertex set can always be joined by a path which avoids any given finite set of other vertices, a greedy argument yields the following.
\begin{lemma}\label{lemma:LinkingPathsAlongCritVertexSet}
    Assume that for a critical vertex set $X$ of a graph $G$ we have two path families $\cP, \cQ$ of $k \in \N$ disjoint $Y$--$X$ paths and of $k$ disjoint $X$--$Z$ paths, respectively, for some $Y, Z \subseteq V(G)$, such that the paths in $\cP \cup \cQ$ are disjoint outside of $X$.
    Then there exists a family of $k$ disjoint $Y$--$Z$ paths in $G$. \qed
\end{lemma}

The following theorem was first proved by Polat \cite{PolatEME1}*{Theorems 3.3 \& 3.8}; we here present a short proof using normal spanning trees.

\begin{theorem} \label{prop:RaylessToughGraphsAreFinite}
    Every tough, rayless graph is finite. 
\end{theorem}

\begin{proof}
    Let $G$ be a tough, rayless graph. Since $G$ is rayless, $V(G)$ is trivially dispersed, so every component $C$ of $G$ admits a normal spanning tree $T_C$ by Jung's normal spanning tree criterion.  By normality, for every node $t \in T_C$ all its successors $s$ are contained in distinct components $\lfloor s \rfloor$ of $C - \lceil t \rceil$. Hence, since $G$ is tough and $\lceil t \rceil$ is finite, $T_C$ has no vertices of infinite degree. As every locally finite, rayless tree, such as the $T_C$, is finite \cite{bibel}*{Proposition~8.2.1}, the components~$C$ of $G$ are finite as well.
    Since $G$ is tough, it has only finitely many components. Hence, $G$ is finite.
\end{proof}

A tree-decomposition~$(T, \cV)$ of a graph $G$ \defn{displays the critical vertex sets} if the map $t \mapsto V_t$ restricted to the infinite-degree nodes of $T$ whose $V_t$ is finite is a bijection to the critical vertex sets of $G$.
For a critical vertex set $X$, we denote by $t_X$ the unique infinite-degree node $t$ with $V_t = X$.
If a rooted \td\ $(T,\cV)$ not only displays the critical vertex sets but also for every critical vertex set $X$ cofinitely many tight components of $G-X$ are $G \strictup e$ for some $e =t_Xt \in T$ with $t_X <_T t$ and every such $G \strictup e$ is a tight component of $G-X$, then $(T,\cV)$ \defn{displays the tight components of every critical vertex set cofinitely}.
We remark that if such a $(T, \cV)$ is tight, then for every finite proper subset $Y$ of any (possibly infinite) bag $V_t$ there are at most finitely many edges $e = ts \in T$ with $t <_T s$ such that $V_e = Y$.
A rooted \td\ $(T, \cV)$ of $G$ into finite parts \defn{displays the infinities} of $G$, if it displays the ends of $G$ homeomorphically, their combined degrees, their dominating vertices, the critical vertex sets and their tight components cofinitely.

\subsection{Critical vertex sets of torsos}

The following lemma, which we will use in the proofs of \cref{maincor:TreeDecompDisplayingInfs,main:LeanTD}, says that the critical vertex sets of a torso are closely related to the critical vertex sets of the underlying graph; in particular, torsos in tough graphs are tough.

\begin{lemma} 
\label{lem:TechnicalCriticalVertexSetsOfTorsos}
    Let $\sigma$ be a star of finite-order separations of a graph $G$ such that for cofinitely many separations $(A,B) \in \sigma$ the side $A$ contains a tight component of $G- (A \cap B)$. 
    Then $\crit(\torsostar(\sigma)) \subseteq \crit(G)$. 
    Moreover, 
    \begin{enumerate}
        \item \label{itm:CriticalVertexSetsOfTorsos:1} 
        If $X \in \crit(\torsostar(\sigma))$, then there are infinitely many tight components of $G-X$ which meet $\torsostar(\sigma)$. 
        \item \label{itm:CriticalVertexSetsOfTorsos:2}
        If $X \subseteq \interior(\sigma)$, then the set $V(C) \cap \interior(\sigma)$ induces a tight component of $\torsostar(\sigma)-X$ for cofinitely many $C \in \breve{\cC}(G-X)$ which meet $\interior(\sigma)$. 
    \end{enumerate}
\end{lemma}

\begin{proof}
    We remark that \cref{itm:CriticalVertexSetsOfTorsos:1} immediately yields that $\crit(\torsostar(\sigma)) \subseteq \crit(G)$.
    Let $U$ be the union of the finite separators of those finitely many $(A,B) \in \sigma$ whose side $A$ does not contain a tight component of $G-(A \cap B)$; in particular, $U$ is finite.

    \cref{itm:CriticalVertexSetsOfTorsos:1}: 
    Since $U$ is finite, only finitely many components of $\torsostar(\sigma) - X$ meet $U$.
    For every component~$C'$ of $\torsostar(\sigma) - X$ which avoids $U$, the subgraph 
    $$C := C' \cup  \bigcup \{G[A] \colon (A,B) \in \sigma, \; V(C') \cap A \neq \emptyset\}$$
    is connected by the definition of $U$. Moreover, since the separators $A \cap B$ of separations $(A,B) \in \sigma$ are complete in $\torsostar(\sigma)$, the component~$C'$ contains the entire $(A \cap B)\setminus X$ as soon as it meets $A \cap B$. Hence, $C$ is a component of $G-X$, and by definition it contains no other components of $\torsostar(\sigma)-X$ than $C'$. 
    It thus suffices to show for all infinitely many components $C'$ of $\torsostar(\sigma) - X$ that avoid $U$ that the component $C$ of $G-X$ which contains $C'$ is tight.
    For this, it suffices to prove that whenever there is a torso edge from $u' \in C'$ to $v \in X$, then there also is some edge from $C$ to $v$ in $G$.
    By the definition of torso, there is a separation $(A,B) \in \sigma$ with $u', v \in A \cap B$.
    The side $A$ of $(A,B)$ contains a tight component $K$ of $G -(A \cap B)$ as $C'$ avoids $U$; in particular, $K \subseteq C$ by the definition of $C$, and $K$ sends an edge to $v$ in $G$. 

    \cref{itm:CriticalVertexSetsOfTorsos:2}: 
    Let $X \subseteq \interior(\sigma)$.
    It suffices to show that $C' := C \cap \interior(\sigma)$ is a tight component of $\torsostar(\sigma)-X$ for every tight component $C$ of $G - X$ which meets $\interior(\sigma)$ but avoids the finite set $U$. 
    The definition of torso immediately yields that $C'$ induces a connected subgraph of $\torsostar(\sigma) - X$ with $N_{\torsostar(\sigma)}(C') \supseteq X$ because $X \subseteq \interior(\sigma)$.
    It remains to show that $N_{\torsostar(\sigma)}(C') \subseteq X$. 
    For this it suffices to prove that whenever there is a torso edge from $u' \in C'$ to $v \notin C'$ the vertex $v$ is already in $X$.
    Let $(A,B) \in \sigma$ such that $u', v \in A \cap B$.
    As $C$ avoids $U$, the component $C$, or equivalently $C'$, only meets separators of separations $(A,B) \in \sigma$ whose side $A$ contains a tight component of $G- (A \cap B)$.
    Thus, we have $(A \cap B) \setminus X \subseteq A \setminus X \subseteq C$ as soon as $C$ meets $A \cap B$. 
    Since $v \in (A \cap B)$ but not in $C$, we thus have $v \in X$, as desired.
\end{proof}

\section{Displaying ends} \label{subsec:DisplayingEnds}

In this short section we show that \cref{main:LinkedTightCompTreeDecompnew2} follows from \cref{main:LinkedTightCompTreeDecompnew}, that is, we show that a linked, tight, componental, rooted \td\ into finite parts homeomorphically displays all ends, their dominating vertices and their combined degrees. 

The proof is divided into three lemmas.
The first follows immediately from \cite{koloschin2023end}*{Lemma~3.1} applied to the \td\ induced by contracting all edges which violate componental.

\begin{lemma}\label{lem:TreeDecompDisplayingEnds}
    Let~$(T, \cV)$ be a cofinally componental\footnote{Our definition of \emph{componental} agrees with the definition of \emph{upwards connected} from \cite{koloschin2023end}.}, rooted \td\ of a graph~$G$ which has finite adhesion.
    Then every rooted ray $R$ of $T$ arises from precisely one end of $G$.
    Moreover, if all torsos of $(T,\cV)$ are rayless, then~$(T, \cV)$ displays all ends of~$G$ homeomorphically. \qed
\end{lemma}

\begin{lemma} \label{lem:TreeDecompDisplayingDominatingVertices}
    Let~$(T, \cV)$ be a tight, componental, rooted \td\ of a graph~$G$ which has finite adhesion.
    Then $\liminf_{e \in R} V_{e} = \Dom(\eps)$ for every rooted $R$ of $T$ and the unique end $\eps$ of $G$ which gives rise to $R$.  
\end{lemma}

\begin{proof}
    Since $(T, \cV)$ has finite adhesion, $\liminf_{e \in R} V_{e}$ contains $\Dom(\eps)$.
    Conversely, let~$v$ be a vertex in~$\liminf_{e \in R} V_{e}$.
    Let $Q \in \eps$ be arbitrary.
    We aim to construct an infinite $v$--$Q$ fan in $G$, which then shows that $v \in \Dom(\eps)$.
    Note that we may recursively find infinitely many pairwise internally disjoint~$v$--$Q$~paths in~$G$ 
    if for each finite set~$X \subseteq V(G) \setminus \{ v \}$ there is a~$v$--$Q$~path in~$G$ avoiding~$X$.
    As $(T, \cV)$ is a \td, there is an edge~$e \in R$ for which~$G \strictup e$ avoids any given finite set~$X \subseteq V(G)$.
    Since~$\eps$ gives rise to~$R$, the ray $Q$ has a tail in $G \strictup e$.
    Thus, we find the desired~$v$--$Q$~path in the connected subgraph $G \up e$, as $(T, \cV)$ is tight and componental.
\end{proof}

\begin{lemma} \label{lem:linkedandliminfareonlydominatingverticesimpliesdisplayenddegrees}
    Let $(T, \cV)$ be a linked, rooted \td\ of a graph~$G$ which has finite adhesion. Suppose that an end $\eps$ of $G$ gives rise to a ray $R$ in $T$ which arises from no other end of $G$ and that $\liminf_{e \in R} V_e = \Dom(\eps)$.
	Then $\liminf_{e \in R} |V_e| = \Delta(\eps)$.

    In particular, if $(T, \cV)$ displays all ends of $G$ and their dominating vertices, then $(T, \cV)$ also displays the combined degree of each end of $G$.
\end{lemma}

\begin{proof}
    Let~$R = v_0 e_0 v_1 e_1 \dots$ be the unique rooted ray in~$T$ which arises from the end~$\eps$.
    If~$\dom(\eps)$ is infinite, then~$(T, \cV)$ displays the combined degree of~$\eps$ by assumption.
    Thus, we may assume~$\dom(\eps)$ to be finite.
   Moving to a tail of~$R$, we may assume that~$\Dom(\eps) \subseteq V_e$ for every~$e \in R$.
    
    From the sequence~$(V_{e_i})_{i \in \N}$ of adhesion sets along~$R$, we extract a subsequence by letting~$i_0 \in \N$ with~$|V_{e_{i_0}}| = \min_{e \in R} |V_e|$ and recursively choosing~$i_{n+1} \in \N$ with~$i_{n+1} > i_{n}$ and~$|V_{e_{i_{n+1}}}| = \min_{j > i_{n}} |V_{e_j}|$ for $n \in \N$.
    Write $S_n := V_{e_{i_n}}$.
    Then~$(|S_n|)_{n \in \N}$ is non-decreasing, satisfies~$\liminf_{e \in R} |V_e| = \liminf_{n \in \N} |S_n|$ and $\liminf_{n \in \N} S_{n} = \liminf_{e \in R} V_{e} = \Dom(\eps)$ by definition.
    
    Since $\liminf_{e \in R} S_{n} = \Dom(\eps)$, we may assume that $S_n \cap S_m \subseteq \Dom(\eps)$ for all~$m \ge n$ by passing to a further subsequence.
    We also have $C_G(S_m, \eps) \subseteq C_G(S_n, \eps)$ since~$\eps$ gives rise to~$R$.
    Moreover, $\bigcap_{n \in \N} C_G(S_n, \eps) = \emptyset$, since~$\eps$ gives rise to~$R$ and~$(T, \cV)$ is a \td.
    This implies by~\cite{enddefiningsequences}*{Corollary~5.7} that~$\liminf_{e \in R} |V_e| = \liminf_{n \in \N} |S_n| \ge \Delta_G(\eps)$.

    For~$\liminf_{e \in R} |V_e| = \liminf_{n \in \N} |S_n| \le \Delta_G(\eps)$, we use that~$(T, \cV)$ is linked:
    By the construction of the sequence~$(S_n)_{n \in \N}$, the linkedness yields~$|S_n|$ disjoint~$S_n$--$S_{n+1}$~paths in~$G$ for every~$n \in \N$.
    The resulting rays in~$G$ are disjoint and contained in~$\eps$ since~$\eps$ is the unique end of~$G$ that gives rise to $R$, and the resulting trivial paths are by assumption precisely the ones given by~$\Dom(\eps)$.
    Moreover, there are~$\liminf_{n \in \N} |S_n|$ many such disjoint rays in $\eps$ and vertices dominating $\eps$, which implies that~$\liminf_{n \in \N} |S_n| \le \Delta_G(\eps)$, as desired.
\end{proof}

\begin{proof}[Proof of \cref{main:LinkedTightCompTreeDecompnew2} given \cref{main:LinkedTightCompTreeDecompnew}]
    By \cref{lem:TreeDecompDisplayingEnds,lem:TreeDecompDisplayingDominatingVertices,lem:linkedandliminfareonlydominatingverticesimpliesdisplayenddegrees} the \td\ from \cref{main:LinkedTightCompTreeDecompnew} is as desired.
\end{proof}

\section{A high-level proof of the main result} \label{sec:MainResultCounterexamples}

\subsection{A two-step approach to \texorpdfstring{\cref{main:LinkedTightCompTreeDecompnew}}{Theorem 1}} \label{subsec:ATwoStepApproach}

By \cref{prop:RaylessToughGraphsAreFinite}, every tough, rayless graph is finite. Hence, to arrive at a \td\ into finite parts, we first construct a \td~$(T, \cV)$ whose torsos are tough. Next, for each torso of $(T,\cV)$ corresponding to a node~$t \in T$ we construct another \td\ $(T^t, \cV^t)$ with rayless torsos. By refining $(T, \cV)$ with all the~$(T^t, \cV^t)$, we obtain a \td\ $(T', \cV')$ into tough and rayless parts, which then must be finite by \cref{prop:RaylessToughGraphsAreFinite}.

But how to guarantee that the resulting \td\ is linked? \cref{lemma:LinkingPathsAlongCritVertexSet}  ensures that if we begin in the above two-step approach with a \td\ $(T, \cV)$ whose adhesion sets are all critical vertex sets, and then refine by \td s $(T^t, \cV^t)$ that are each linked when considered individually, then the arising combined \td\ is again linked.\footnote{This is not completely true: In order to ensure that the arising \td\ is linked we need a `refinement version' of \cref{thm:RaylessThmIntro} (see \cref{thm:RaylessThmTechnical}) that ensures that the adhesion sets of $(T, \cV)$ corresponding to edges incident with a given node $t \in T$ appear as adhesion sets in $(T^t, \cV^t)$. But this is only a technical issue which does not add much complexity to the proof of \cref{thm:RaylessThmIntro}.}

A little more formally, the first step in our two-step approach is the following theorem:

\begin{mainresult} \label{thm:critVtxIntro}
    Every graph of finite tree-width admits a tight, cofinally componental, rooted \td\ whose adhesion sets are critical vertex sets, whose torsos are tough and which displays the critical vertex sets and their tight components cofinitely.
\end{mainresult}

The second step in our two-step approach is formally given by the following theorem:

\begin{mainresult} \label{thm:RaylessThmIntro}
    Every graph of finite tree-width admits a linked, tight, componental, rooted \td\ of finite adhesion whose torsos are rayless.
\end{mainresult}

Our aim in the remainder of this section is to demonstrate formally how \cref{thm:critVtxIntro} and \cref{thm:RaylessThmIntro} can be combined to yield a proof of \cref{main:LinkedTightCompTreeDecompnew}. 
The subsequent \cref{sec:TreeDecompAlongCrit,sec:DecomposeIntoRaylessParts} are then concerned with proving \cref{thm:critVtxIntro,thm:RaylessThmIntro} respectively.

\subsection{Detailed versions of the main theorems}\label{subsec:detailedversions}
In fact, we will prove a version of \cref{main:LinkedTightCompTreeDecompnew} not as stated in the introduction, but with three additional properties that are important for technical reasons, but which we believe are also of interest in their own right.

\begin{customthm}{Theorem~1'}[Detailed version of \cref{main:LinkedTightCompTreeDecompnew}] \label{thm:MainTechnical}
    Every graph~$G$ of finite tree-width admits a rooted \td~$(T, \cV)$ into finite parts that is linked, tight, and componental.
    In particular, $(T, \cV)$ displays all the ends of $G$ homeomorphically, their combined degrees and their dominating vertices.
    Moreover, we may assume that
    \begin{enumerate}[label=\rm{(L\arabic*)}]
        \item \label{itm:MainTechnical:EndCritLinked} for every~$e \in E(T)$, the adhesion set~$V_e$ is either linked to a critical vertex set of~$G$ that is included in~$G \up e$ or linked to an end of~$G$ that lives in~$G \up e$, and
        \item \label{itm:MainTechnical:IncDisj} for every $e <_T e' \in E(T)$ with $|V_e| \leq |V_{e'}|$, each vertex of $V_e \cap V_{e'}$ either dominates some end of~$G$ that lives in~$G \up e'$, or is contained in a critical vertex set of $G$ that is included in~$G\up e'$.
        \item \label{itm:MainTechnical:DistinctBags} the bags of $(T, \cV)$ are pairwise distinct.
    \end{enumerate}
\end{customthm}

Before we continue, let us quickly comment on these additional properties: 
Property \cref{itm:MainTechnical:EndCritLinked} is a minimality condition: Since every end of $G$ lives in an end of $(T, \cV)$, there will be infinitely many edges of $T$ whose adhesion set is linked to that end. Moreover, since all parts of $(T, \cV)$ are finite, one can easily see that every critical vertex set will appear as some adhesion set of $(T, \cV)$ (cf.~\cref{lem:TightCompToughTDNearlyDisplaysCrit}). So \cref{itm:MainTechnical:EndCritLinked} says that we did not decompose $G$ `more than necessary'. In particular, every bag at a leaf of $(T, \cV)$ will be of the form $X \cup V(C)$ for a critical vertex set $X$ of $G$ and a finite tight component $C$ of $G - X$.

Let us now turn to property \cref{itm:MainTechnical:IncDisj}.
Recall that Halin \cite{halin1975chainlike}*{Theorem~2} showed that every locally finite connected graph has a linked ray-decomposition into finite parts with disjoint adhesion sets. In light of this, \cref{itm:MainTechnical:IncDisj} describes how close we can come to having `disjoint adhesion sets' in the general case; \cite{ExamplesLinkedTDInfGraphs}*{\exampleNoUpwardsDisjointAdhesionSets} shows that even for locally finite graphs, it may be impossible to get a tree-decomposition with disjoint adhesion sets, so the condition `with $|V_e| \leq |V_{e'}|$'  in \cref{itm:MainTechnical:IncDisj} is indeed necessary.

Now as already indicated in the introduction, we first show \cref{maincor:TreeDecompDisplayingInfs} and then derive \cref{main:LinkedTightCompTreeDecompnew} from it. 
In fact, we derive the detailed version \cref{thm:MainTechnical} from the following detailed version of \cref{maincor:TreeDecompDisplayingInfs}.

\begin{customcor}{Theorem~4'}[Detailed version of \cref{maincor:TreeDecompDisplayingInfs}] \label{maincor:TreeDecompDisplayingInfsTechnical}
    Every graph $G$ of finite tree-width admits a fully tight, cofinally componental, linked, rooted \td\ into finite parts which displays the infinities of $G$ and which satisfies \cref{itm:MainTechnical:EndCritLinked} and \cref{itm:MainTechnical:IncDisj} from \cref{thm:MainTechnical}.
    Moreover,
    \begin{enumerate}[label=\rm{(I\arabic*)}]
        \item \label{itm:TreeDecompDisplayingInfsTechnical:CofinComp} if $G\strictup e$ is disconnected for $e = st \in E(T)$ with $s <_T t$, then $V_s \supseteq V_t \in \crit(G)$ and $\deg(t) = \infty$, and 
        \item \label{itm:TreeDecompDisplayingInfsTechnical:DistinctBags}  
        if $V_t = V_e$ for some node $t \in T$ and the unique edge $e = st \in T$ with $s <_T t$, then $\deg(t) = \infty$ and $V_t \in \crit(G)$.
    \end{enumerate}
\end{customcor}

We remark that whenever a \td\ $(T,\cV)$ displays the critical vertex sets and satisfies \cref{itm:TreeDecompDisplayingInfsTechnical:CofinComp}, then it is automatically cofinally componental, because if two successive comparable edges $rs, st \in T$ violate componental, then \cref{itm:TreeDecompDisplayingInfsTechnical:CofinComp} yields that $\deg(t), \deg(s) = \infty$ and $V_s \supseteq V_t$, which implies that $V_s \supsetneq V_t$, since $(T, \cV)$ displays the critical vertex sets and hence $V_s \neq V_t$. 
Moreover, every \td\ as in \cref{maincor:TreeDecompDisplayingInfsTechnical} already `nearly' satisfies \cref{itm:MainTechnical:DistinctBags}, that is, has `almost' distinct bags.

\begin{lemma} \label{lem:I1AndI2YieldDistinctBags}
    Let $(T, \cV)$ be a tight, rooted \td\ of a graph $G$ which displays the critical vertex sets of $G$ and satisfies~\cref{itm:TreeDecompDisplayingInfsTechnical:CofinComp} and~\cref{itm:TreeDecompDisplayingInfsTechnical:DistinctBags}. Then the bags of $(T, \cV)$ are pairwise distinct, unless they are critical vertex sets, which may appear as at most two bags associated with adjacent nodes of $T$. 
\end{lemma}

\begin{proof}
Suppose there are distinct nodes $t, s \in T$ such that $V_t = V_s$. Then $tTs$ contains from at least one of $t,s$, say from $t$, its unique down-edge $e$. Then \cref{prop:TD2} ensures that $V_e = V_t$, so by~\cref{itm:TreeDecompDisplayingInfsTechnical:DistinctBags} we have $V_t \in \crit(G)$ and $\deg(t) = \infty$. In particular, since $(T, \cV)$ displays the critical vertex sets of $G$, we have $\deg(s) \neq \infty$ and hence $s <_T t$, again by \cref{itm:TreeDecompDisplayingInfsTechnical:DistinctBags}. 
Now if $V_{s'} = V_s$ for the unique neighbour of $s$ in $sTt$, then $\deg(s') = \infty$ by \cref{itm:TreeDecompDisplayingInfsTechnical:DistinctBags}, so $s' = t$ because $(T, \cV)$ displays the critical vertex sets of~$G$. Hence, we may assume that $V_{s'} \neq V_s$, so $V_{s'} \supsetneq V_s$ by \cref{prop:TD2}. Then \cref{itm:TreeDecompDisplayingInfsTechnical:CofinComp} implies that $G \strictup f$ is connected. But then \cref{lem:BagsEqualXAreConnectedInT} implies that $V_{s'} = V_t = V_s$, a contradiction.
\end{proof}

\begin{proof}[Proof of \cref{thm:MainTechnical} given \cref{maincor:TreeDecompDisplayingInfsTechnical}]
    Let $(T', \cV')$ be the \td\ obtained from \cref{maincor:TreeDecompDisplayingInfsTechnical}.
    Let $(T, \cV)$ be the \td\ induced by contracting every edge $e$ of $T'$ whose $G \strictup e$ is disconnected.
    It is immediate from the construction that $(T, \cV)$ is componental.
    Since~$(T',\cV')$ is fully tight and satisfies \cref{itm:MainTechnical:EndCritLinked} and \cref{itm:MainTechnical:IncDisj} from \cref{thm:MainTechnical}, so does every \td\ induced by edge-contractions from $(T',\cV')$.
    We note that whenever an edge~$e = st \in E(T')$ with $s <_{T'} t$ has been contracted then no other edge $f$ incident with $t$ has been contracted and for $V_f = V_t = V_e$ for all such edges $f \in T'$ incident with $t$,
    because $(T',\cV')$ displays the critical vertex sets and their tight components cofinitely and it also satisfies \cref{itm:TreeDecompDisplayingInfsTechnical:CofinComp}.
    Hence, such as $(T',\cV')$ does, the induced \td\ $(T, \cV)$ still is linked, displays the ends of $G$ homeomorphically, their combined degrees and their dominating vertices, and also its parts are finite.
    
    Moreover, $(T, \cV)$ satisfies \cref{itm:MainTechnical:DistinctBags}. Indeed, by \cref{lem:I1AndI2YieldDistinctBags}, we only need to check that every critical vertex set appears as at most one bag of $(T, \cV)$. By \cref{lem:I1AndI2YieldDistinctBags}, every $X \in \crit(G)$ can appear as at most two bags of $(T', \cV')$, which then need to be adjacent. 
    So assume $V_s = V_t = X$ with~$s$ being the unique down-neighbour of $t$. Then by \cref{itm:TreeDecompDisplayingInfsTechnical:DistinctBags}, $\deg(t) = \infty$. 
    Since $(T', \cV')$ displays the critical vertex sets and their tight components cofinitely, it follows that $G\strictup(st)$ is disconnected. Thus, we have contracted the edge $st$ in the construction of $(T, \cV)$. Hence, also every critical vertex set of $G$ appears as at most one bag of $(T, \cV)$, so its bags are paiwise distinct.
\end{proof}

In order to prove \cref{maincor:TreeDecompDisplayingInfsTechnical}, we still follow the promised two-step approach, but need the following detailed versions of \cref{thm:critVtxIntro,thm:RaylessThmIntro}.

\begin{customthm}{Theorem~6'}[Detailed version of \cref{thm:critVtxIntro}] \label{thm:critVtxTechnical}
    Let $G$ be a graph of finite tree-width.
    Then $G$ admits a fully tight, cofinally componental, rooted \td~$(T,\cV)$ whose adhesion sets are critical vertex sets, whose torsos are tough and which displays the critical vertex sets and their tight components cofinitely.

    Moreover, it satisfies \cref{itm:TreeDecompDisplayingInfsTechnical:CofinComp} and \cref{itm:TreeDecompDisplayingInfsTechnical:DistinctBags} from \cref{maincor:TreeDecompDisplayingInfsTechnical}.
\end{customthm}

In order to state the detailed version of \cref{thm:RaylessThmIntro}, we need one more definition. 
A separation~$(A,B)$ of a graph $G$ is \defn{left-tight} if some components $C$ of $G[A \setminus B]$ satisfies $N_G(C) = A \cap B$.
Moreover, a separation $(A,B)$ of a graph $G$ is \defn{left-fully-tight} if all components $C$ of $G[A \setminus B]$ satisfy $N_G(C) = A \cap B$.

\begin{customthm}{Theorem~7'}[Detailed version of \cref{thm:RaylessThmIntro}] \label{thm:RaylessThmTechnical}
    Let~$G$ be a graph, and let~$\sigma$ be a star of left-well-linked left-fully-tight finite-order separations of~$G$ such that $\torsostar(\sigma)$ has finite tree-width. 
    Further, let~$X\subseteq \interior(\sigma)$ be a finite set of vertices of~$G$. 
    Then~$G$ admits a linked, $X$-linked, fully tight,  rooted \td\ $(T, \cV)$ of finite adhesion 
    such that
    \begin{enumerate}[label=\rm{(R\arabic*)}]
        \item \label{itm:RaylessThmTechnical:rayless:Copy} its torsos at non-leaves are rayless and its leaf separations are precisely~$\{(B, A) \mid (A, B) \in \sigma\}$,

        \item \label{itm:RaylessThmTechnical:EndAdhesionLinked:Copy} for all edges $e$ of $T$, the adhesion set $V_e$ is either linked to an end living in $G \up e$ or linked to a set $A \cap B \subseteq G \up e$ with $(A,B) \in \sigma$,
        \item \label{itm:RaylessThmTechnical:IncDis:Copy} for every $e <_T e' \in E(T)$ with $|V_e| \leq |V_{e'}|$, either each vertex of $V_e \cap V_{e'}$ dominates some end of~$G$ that lives in~$G \up e'$, or $V_e \cap V_{e'}$ is contained in $A \cap B \subseteq G \up e'$ for some $(A,B) \in \sigma$, 
        \item \label{itm:RaylessThmTechnical:DistinctBags:Copy}
        $V_s \supsetneq V_e \subsetneq V_t$ for all edges $e = st \in T$ with $s <_T t$ and $s \neq r := \rt(T)$. Moreover, if $X \subsetneq\interior(\sigma) $, $G-X$ is connected and $N_G(G-X) = X$, then $X \subsetneq V_r$ and also $V_r \supsetneq V_e \subsetneq V_t$ for all edges $e = rt \in T$.
    \end{enumerate}
\end{customthm}

\noindent Note that \cref{itm:RaylessThmTechnical:DistinctBags:Copy} implies that the bags of $(T, \cV)$ are pairwise distinct.

\subsection{Proof of the main result} \label{subsec:ProofOfMainResult}

 According to our two-step approach, we prove \cref{maincor:TreeDecompDisplayingInfsTechnical} by applying \cref{thm:RaylessThmTechnical} to the torsos of the \td\ given by \cref{thm:critVtxTechnical}. For this, we need to ensure that all torsos again have finite tree-width: 

\begin{lemma}\label{lem:NSTsOfTorsos}
    Let $G$ be a graph of finite tree-width, and let $G'$ be obtained from $G$ by adding an edge between $u, v \in V(G)$ whenever there are infinitely many internally disjoint $u$--$v$ paths in $G$. 
    Then every \td\ of~$G$ of finite adhesion is also a \td\ of~$G'$.

    In particular, if $G$ has finite tree-width, then the torso at a star of separations of~$G$ whose separators are critical vertex sets in $G$ has finite tree-width.
\end{lemma}

\begin{proof}
    Assume that $(T, \cV)$ is a \td\ of~$G$ of finite adhesion, and consider any two vertices~$u, v \in V(G)$ with $uv \in E(G') \setminus E(G)$.
    If there exists a bag~$V_t$ containing both~$u$ and~$v$, then the edge~$uv$ in $G'$ cannot violate that~$(T, \cV)$ is a \td\ of $G'$.
    To find such a bag~$V_t$, recall that there are infinitely many internally disjoint $u$--$v$ paths in $G$.
    In particular, no finite set of vertices other than $u$ and $v$ separates $u$ and $v$ in $G$.
    Since $(T, \cV)$ has finite adhesion, $u$ and $v$ must be contained in some bag~$V_t$ of~$(T, \cV)$, as desired.

    For the `in particular'-part assume that $G$ has finite tree-width and that $\sigma$ is a star of separations of $G$ whose separators are critical in $G$. Let $(T, \cV)$ be a \td\ of $G$ into finite parts.
    Then by the first part and because critical vertex sets are infinitely connected, $(T, \cV')$ with $V'_t := V_t \cap \interior(\sigma)$ is a \td\ of the torso at $\sigma$ in $G$, as desired.
\end{proof}

\begin{proof}[Proof of~\cref{maincor:TreeDecompDisplayingInfsTechnical} given \cref{thm:critVtxTechnical} and \cref{thm:RaylessThmTechnical}]
    Let $(T^1, \cV^1)$ be the rooted \td\ from \cref{thm:critVtxTechnical} whose adhesion sets are critical vertex sets.
    In particular, $(T^1, \cV^1)$ displays the critical vertex sets of $G$ and their tight components cofinitely. Moreover, its torsos are tough and it satisfies~\cref{itm:TreeDecompDisplayingInfsTechnical:CofinComp} from \cref{maincor:TreeDecompDisplayingInfsTechnical}.
    Let~$t$ be a node of~$T^1$. 
    We describe how we refine the torso at~$t$ in~$T^1$ using \cref{thm:RaylessThmTechnical}:

    If $t$ is the root of $T^1$, then we set $G^t := G$, $X_t := \emptyset$ and $\sigma'_t = \sigma_t$.
    Else let $s \in T^1$ be the (unique) predecessor of $t$ and let~$G^t$ be the graph obtained from $G\up st$ by adding all edges between vertices of $V^1_{st}$.
    Set $X_t := V^1_{st}$ and $\sigma'_t := \{(A,B \setminus V(G \strictdown st)) \mid (A_{st}, B_{st}) \neq (A,B) \in \sigma_t\}$, where $(A_{st}, B_{st})$ is the separation induced by $(s,t)$.

    First, we assume that $t$ is a node of $T^1$ with $V^1_t \in \crit(G)$; in particular, all infinite-degree nodes whose corresponding bag is finite are such $t$, since $(T^1, \cV^1)$ displays the critical vertex sets.
    Then we set $(T^t, \cV^t)$ to be the \td\ of $G^t$ whose decomposition tree is a star rooted in its centre $t$ and with bag $V^t_t := V^1_t$ while its leaf separations are precisely $((B,A) \mid (A,B) \in \sigma'_t)$. Note that all adhesion sets of this \td\ $(T^t,\cV^t)$ are $V^1_t \in \crit(G)$.
    Thus, $(T^t, \cV^t)$ is a rooted \td\ as in the conclusion of \cref{thm:RaylessThmTechnical}.
    
    Secondly, we now assume that $t$ is a node of $T^1$ with $V^1_t \notin \crit(G)$.
    By construction, the torso of $\sigma'_t$ in $G^t$ is equal to the torso of $\sigma_t$ in $G$; so by \cref{lem:NSTsOfTorsos}, this torso has finite tree-width.
    We claim that $G^t$, $X_t$ and~$\sigma'_t$ are as required to apply~\cref{thm:RaylessThmTechnical}. 
    For this it suffices that all separations in $\sigma'_t$ are left-well-linked.
    Let $(A,B) \in \sigma'_t$. 
    Then $X =: A \cap B$ is some critical vertex set of $G$, as it is an adhesion set of $(T^1, \cV^1)$.
    Since $(T^1, \cV^1)$ displays the critical vertex sets, there is a unique infinite-degree node $t_X \in T^1$ with $V^1_{t_X} = X$.
    If infinitely many tight component of~$G - X$ are contained in $A$, then $(A,B)$ is left-well-linked.
    Thus, it now suffices to show that~$t <_{T^1} t_X$, as $(T^1, \cV^1)$ cofinitely displays also the tight components of the critical vertex sets.
    Because~$V_t \notin \crit(G)$, $t \neq t_X$; hence, we suppose towards a contradiction that~$t >_{T^1} t_X$.
    Now~$G[A \setminus B]$ is in particular non-empty, as $(T^1, \cV^1)$ is fully tight.
    Also $G \strictup f$ is a (tight) component of $G-X$ where $f$ is the unique edge on $t_XT^1t$ incident with $t$, since $(T^1, \cV^1)$ cofinitely displays the tight components of the the critical vertex sets. 
    But then \cref{lem:BagsEqualXAreConnectedInT} yields $V_t = X \in \crit(G)$ which contradicts the assumptions on $t$.
    All in all, we may now apply \cref{thm:RaylessThmTechnical} to obtain the rooted \td~$(T^t, \cV^t)$ of~$G^t$.

    We remark that all these rooted \td s $(T^t, \cV^t)$ in particular contain~$X_t$ in its root part and have precisely~$((B, A) \mid (A, B) \in \sigma'_t)$ as its leaf separations.
    To build the desired \td\ of~$G$, we first stick all these \td s $(T^t, \cV^t)$ together along~$(T^1, \cV^1)$:
    Formally, the tree~$T^2$ arises from the disjoint union of the trees~$T^t$ by identifying a leaf~$u$ of~$T^t$ with the root of~$T^s$ if $G\strictup u$ (with respect to $T^t$) is equal to $G^s - X_s$; 
    we say that the edge of~$T^t$ (and hence of $T^2$) incident with the leaf $u$ \defn{belongs to~$T^1$} and that it \defn{corresponds to~$ts$}.
    All edges of~$T^t$ that do not belong to~$T^1$ are said to \defn{belong to~$T^t$}.
    We remark that every edge belongs to precisely one of the $T^t$ or $T^1$.
    We set the root of~$T^2$ to be~$\rt(T^{\rt(T^1)})$. 
    For each node~$s \in T^2$, we then set~$V^2_s$ to be~$V^t_s$, where~$t$ is the (unique) node of~$T^1$ such that~$s$ is either a non-leaf of~$T^t$ or the unique node of~$T^t$. 
    We say that $s$ \defn{belongs to $T^t$}.
    Now we claim that~$(T^2, \cV^2)$ has all the desired properties.

    Let us first note that the construction of $T^2$ immediately ensures that
    \begin{itemize}
        \item if $e \in T^2$ belongs to $T^t$ for some $t \in T^1$, then $G \strictup e = G^t \strictup e$ and $V^2_e = V^t_e$, and
        \item if $e \in T^2$ belongs to $T^1$, then $G \strictup e = G \strictup f$ and $V^2_e = V^1_f$ for the edge $f \in T^1$ which $e$ corresponds to.
    \end{itemize}
    Thus, $(T^2, \cV^2)$ is fully tight, since $(T^1, \cV^1)$ and all the $(T^t, \cV^t)$ are fully tight.
    The \td\ $(T^2, \cV^2)$ satisfies \cref{itm:TreeDecompDisplayingInfsTechnical:CofinComp} from \cref{maincor:TreeDecompDisplayingInfsTechnical}, since $(T^1, \cV^1)$ and all the $(T^t, \cV^t)$ satisfy \cref{itm:TreeDecompDisplayingInfsTechnical:CofinComp}.
    It also satisfies \cref{itm:TreeDecompDisplayingInfsTechnical:DistinctBags}. Indeed, by \cref{thm:critVtxTechnical}, $(T^1, \cV^1)$ satisfies~\cref{itm:TreeDecompDisplayingInfsTechnical:DistinctBags}.
    Moreover, by \cref{itm:RaylessThmTechnical:DistinctBags:Copy} from \cref{thm:RaylessThmTechnical}, the $(T^t, \cV^t)$ have the property that $V^t_x \supsetneq V^t_e \subsetneq V^t_y$ for all edges $e = xy \in T^t$ with~$x <_{T^t} y$ and also $V^1_e \subsetneq V^s_{\rt(T^s)}$ where $e = st \in T^1$ with $s <_{T^1} t$. 
    We remark that we used here that every node $t$ and its unique edge $e=st \in T^1$ with $s <_{T^1} t$ to which we applied \cref{thm:RaylessThmTechnical} satisfies $X_t = V^1_{st} \subsetneq V^1_t = \interior(\sigma) = \interior(\sigma_t')$ by \cref{itm:TreeDecompDisplayingInfsTechnical:DistinctBags} from \cref{thm:critVtxTechnical:copy} and also $G^t-X_t$ is connected with $N_G^t(G^t-X) = X$ by \cref{itm:TreeDecompDisplayingInfsTechnical:CofinComp} and since $(T^1, \cV^1)$ is fully tight.
    It follows that $V^2_x \supsetneq V^2_e \subsetneq V^2_y$ for all edges $e = xy$ with $x <_{T^2} y$ except those where $V^2_x = V^1_x = V^1_y = V^2_y \in \crit(G)$. In particular, $(T^2, \cV^2)$ satisfies~\cref{itm:TreeDecompDisplayingInfsTechnical:DistinctBags}.

    Let us now show that all parts of $(T^2, \cV^2)$ are finite. By~\cref{prop:RaylessToughGraphsAreFinite}, it suffices to show that the torso at every $s \in T^2$ is rayless and tough.
    Let $t$ be the node of $T^1$ to whose $T^t$ the node $s$ belongs.
    It is immediate from the construction of $(T^2, \cV^2)$ that the torso $G^2_s$ at $s \in T^2$ in $(T^2, \cV^2)$ is equal to the torso $G^t_s$ at $s$ in $(T^t, \cV^t)$; in particular, these torsos are rayless by \cref{itm:RaylessThmTechnical:rayless:Copy}.
    Suppose for a contradiction that the torso $G^2_s$ at $s \in T^2$ in $(T^2, \cV^2)$ contains a critical vertex set $X$.
    By \cref{lem:TechnicalCriticalVertexSetsOfTorsos}~\cref{itm:CriticalVertexSetsOfTorsos:1}, infinitely many tight components of $G-X$ meet the torso $G^2_s$; in particular, they meet the torso $G^1_t$ at $t \in T^1$ in $(T^1, \cV^1)$.
    Now by \cref{lem:TechnicalCriticalVertexSetsOfTorsos}~\cref{itm:CriticalVertexSetsOfTorsos:2}, cofinitely many of these tight components of $G-X$ restrict to tight components of the torso $G^1_t$ at $t$ in $(T^1, \cV^1)$.
    Thus, $X$ is a critical vertex set of the tough torso $G^1_t$ which is a contradiction.

    Since the adhesion sets of $(T^1, \cV^1)$ are critical vertex sets of $G$, \cref{itm:MainTechnical:EndCritLinked} and \cref{itm:MainTechnical:IncDisj} follow immediately from \cref{itm:RaylessThmTechnical:EndAdhesionLinked:Copy} and \cref{itm:RaylessThmTechnical:IncDis:Copy}, respectively.

    It remains to show that~$(T^2, \cV^2)$ is linked.
    So let~$e \le_{T^2} e'$ be given.
    Suppose first that there is a node~$t$ of~$T^1$ such that all edges in~$eT^2e'$ belong either to~$T^t$ or correspond to edges of~$T^1$ incident with~$t$.
    The \td\ $(T^t, \cV^t)$ obtained from \cref{thm:RaylessThmTechnical} is linked and~$X_t$-linked, where~$X_t = V^1_{st}$ for the predecessor $s$ of $t$ in $T^1$.
    Thus, there exists an edge~$f \in eT^te'$ with~$e \le_{T^t} f \le_{T^t} e'$ and a family~$\cP$ of~$k := |V^t_f|$ disjoint~$V^t_e$--$V^t_{e'}$~paths in the auxiliary graph~$G^t$ such that~$v \in P_v$.
    As $G^t \subseteq G$, these paths are also paths in $G$.
    Since the adhesion sets of $T^t$ and the tree-order in~$T^t$ directly transfer to~$T^2$ by construction, this completes the first case.

    To conclude the proof that $(T^2, \cV^2)$ is linked, let~$k$ be the minimum size of an adhesion set~$V_g$ among all edges~$g \in eT^2e'$.
    Further, let~$f_1, \dots, f_\ell$ be the edges on the path~$eT^2e'$ that belong to~$T^1$ ordered by $\leq_{T^2}$.
    To find~$k$ disjoint~$V^2_e$--$V^2_{e'}$~paths in~$G$, we apply the above argument to each subpath~$f_i T^2 f_{i+1}$ with~$i \in \{1, \dots, n-1\}$.
    By the choice of~$k$, we get a family~$\cP_i$ of~$k$ disjoint~$V^2_{f_i}$--$V^2_{f_{i+1}}$~paths in~$G$ for every $i \in \{1, \dots ,n-1\}$.
    As $(T^2, \cV^2)$ is a \td, we have that for every~$g_1 \le_{T^2} g_2 \le_{T^2} g_3$, $V_{g_2}$ separates~$V^2_{g_1}$ and~$V^2_{g_3}$. Therefore, $P_i \in \cP_i$ and $P_j \in \cP_j$ are internally disjoint for $i \neq j$.
    We remark that~$V^2_{f_{i+1}}$ is an adhesion set of~$(T^1, \cV^1)$ and thus critical in~$G$.
    Hence, \cref{lemma:LinkingPathsAlongCritVertexSet} yields the desired path family.

    In particular, $(T^2, \cV^2)$ displays all the ends of $G$ homeomorphically, their dominating vertices and their combined degrees by \cref{lem:TreeDecompDisplayingEnds,lem:TreeDecompDisplayingDominatingVertices,lem:linkedandliminfareonlydominatingverticesimpliesdisplayenddegrees}.
    It is immediate from the construction that~$(T^2, \cV^2)$ still displays the critical vertices and their tight components cofinitely.
\end{proof}

\section{Tree-decomposition along critical vertex sets} \label{sec:TreeDecompAlongCrit}

In this section we prove 
\cref{thm:critVtxTechnical}, which we restate here for convenience.

\begin{customthm}{\cref{thm:critVtxTechnical}}[Detailed version of \cref{thm:critVtxIntro}] \label{thm:critVtxTechnical:copy}
    Let $G$ be a graph of finite tree-width.
    Then $G$ admits a fully tight, cofinally componental, rooted \td~$(T,\cV)$ whose adhesion sets are critical vertex sets, whose torsos are tough and which displays the critical vertex sets and their tight components cofinitely.

    Moreover,
    \begin{enumerate}[label=\rm{(I\arabic*)}]
        \item \label{itm:critVtxTechnical:copy:CofinComp} if $G\strictup e$ is disconnected for $e = st \in E(T)$ with $s <_T t$, then $V_s \supseteq V_t \in \crit(G)$ and $\deg(t) = \infty$, and 
        \item \label{itm:critVtxTechnical:copy:DistinctBags}  
        if $V_t = V_e$ for some node $t \in T$ and the unique edge $e = st \in T$ with $s <_T t$, then $\deg(t) = \infty$ and $V_t \in \crit(G)$.
    \end{enumerate}
\end{customthm}

\noindent Recall that these \cref{itm:critVtxTechnical:copy:CofinComp} and \cref{itm:critVtxTechnical:copy:DistinctBags} are the same properties as \cref{itm:TreeDecompDisplayingInfsTechnical:CofinComp} and \cref{itm:TreeDecompDisplayingInfsTechnical:DistinctBags} in \cref{maincor:TreeDecompDisplayingInfsTechnical}.

This \td\ is not difficult to construct: Start from the \td~$(T,\cV'_\textnormal{NT})$ described in the introduction. By contracting all edges of the decomposition tree whose corresponding adhesion sets are not critical vertex sets of $G$, one obtains a \td\ of $G$ that satisfies all properties required for the \td\ in \cref{thm:critVtxTechnical:copy} except that it might not display the critical vertex sets. We then describe how one can turn this \td\ into one that additionally displays all critical vertex sets and their tight components cofinitely.

We first collect two lemmas that describe how the critical vertex sets of a graph interact with a tight, componental, rooted \td. 

\begin{lemma}\label{lem:TightCompToughTDNearlyDisplaysCrit}
    Let $(T, \cV)$ be a rooted \td\ of a graph $G$ of finite adhesion. 
    Then every critical vertex set of $G$ is contained in some bag of $(T, \cV)$.
    
    Moreover, 
    if $(T,\cV)$ is tight, componental and its torsos are tough, then for every $X \in \crit(G)$ cofinitely many tight components of $G-X$ are of the form $G \strictup e$ for an edge $e=t_Xt \in T$ with~$t_X <_T t$ where $t_X$ is the (unique) $\leq_T$-minimal node of $T$ whose corresponding bag contains~$X$.
\end{lemma}

\begin{proof}
    Since critical vertex sets are infinitely connected and $(T, \cV)$ has finite adhesion, every critical vertex set is contained in some bag of $(T, \cV)$.
    Thus, the nodes whose corresponding bags contain a fixed critical vertex set $X$ form a subtree of $T$ by \cref{prop:TD2} which thus has a unique $\leq_T$-minimal node $t_X$. Let $e_0 <_T t_X$  be the unique edge of $T$ incident with $t_X$.
    Since $(T, \cV)$ is componental, every tight component of $G-X$ which does not meet $G \down e_0$ is of the form $G \strictup e$ for an edge $e=t_Xt \in T$ with $t_X <_T t$.
    By the choice of $t_X$, every tight component that meets $G \strictdown e_0$ also meets $V_{e_0} \subseteq V_{t_X}$.
    \cref{lem:TechnicalCriticalVertexSetsOfTorsos}~\cref{itm:CriticalVertexSetsOfTorsos:2} yields that only finitely many tight components of $G-X$ meet $V_{t_X}$, as $(T,\cV)$ is tight and its torsos are tough.
    Thus, cofinitely many tight components of~$G-X$ are of the desired form.
\end{proof}

The following lemma is kind of a converse of \cref{lem:TightCompToughTDNearlyDisplaysCrit}:

\begin{lemma} \label{lem:DisplayingCritYieldsToughTorsos}
    Let $(T, \cV)$ be a tight, rooted \td\ of a graph $G$ such that for every $X \in \crit(G)$ there is a node $t_X \in T$ such that cofinitely many tight components of $G-X$ are of the form $G\strictup e$ for an edge $e \in t_Xt \in T$ with $t_X <_T t$. Then all torsos of $(T, \cV)$ are tough.
\end{lemma}

\begin{proof}
    Let $t \in T$ be arbitrary, and suppose for a contradiction that $\torsostar(\sigma_t)$ is not tough, that is, there is some $X \in \crit(\torsostar(\sigma_t))$. Since $(T, \cV)$ is tight by assumption, we may apply \cref{lem:TechnicalCriticalVertexSetsOfTorsos}~\ref{itm:CriticalVertexSetsOfTorsos:1} to the star~$\sigma_t$, which yields that $X$ is also a critical vertex set of $G$; moreover, infinitely many tight components of $G-X$ meet $V_t$. In particular, we have $t \neq t_X$ by assumption on $t_X$. So let $e \in T$ be the edge incident with $t_X$ on the unique $t$--$t_X$ path in $T$, and let $(A,B) \in \sigma_{t_X}$ be the separation induced by the orientation of $e$ towards~$t_X$; in particular, $V_t \subseteq A$. Then by the assumption on $t_X$ at most finitely many tight components of $G-X$ are contained in $A$, or equivalently meet $A$, which is a contradiction as $V_t \subseteq A$.
\end{proof}

We now use the previous two lemmas to show that every graph of finite tree-width admits a \td\ which satisfies all properties required for the \td\ in \cref{thm:critVtxTechnical:copy} except that it might not display the critical vertex sets.

\begin{lemma} \label{lem:CritDecompWithoutDisplayingCrit}
    Every graph of finite tree-width admits a tight, componental, rooted \td~$(T, \cV)$ whose adhesion sets are critical vertex sets and whose torsos are tough. Moreover, 
    \begin{enumerate}
        \item \label{itm:CritDecompWithoutDisplayingCrit:progress} $V_t \setminus V_e$ is non-empty for every node $t \in T$ and the unique edge $e =st \in T$ with $s <_T t$.
    \end{enumerate}
\end{lemma}

\begin{proof}
    We may assume that $G$ is connected.
    Indeed, if $G$ is not connected, we may apply \cref{lem:CritDecompWithoutDisplayingCrit} to every component $C$ of $G$ to obtain a \td\ $(T^C, \cV^C)$.
    Let $T$ be the disjoint union of all $T^C$ and a new node $r$ to which we join the root of every $T^C$.
    We assign~$V_r := \emptyset$ and $V_t := V^C_t$ for the respective component $C$ of $G$.
    Now $(T, \cV)$ is as desired.
    
    Since the connected graph $G$ has finite tree-width, we may fix a normal spanning tree $T$ of $G$ by \cref{thm:FiniteTWyieldsNST}. Let $(T,\cV)$ ne the tree-decomposition into finite parts introduced unter the name \td~$(T,\cV'_\textnormal{NT})$ in the introduction; i.e.\ the tree-decomposition with decomposition tree $T$ whose bags are given  by $V_t := \{t\} \cup N_G(\lfloor t \rfloor) \subseteq \lceil t \rceil$. 
    By construction, for an edge $e = st \in T$ with $s <_T t$, we have $G \strictup e = \lfloor t \rfloor$ and $V_e = \lceil s \rceil \cap N_G(\lfloor t \rfloor)$; 
    in particular, the rooted \td\ $(T, (V_t)_{t \in T})$ is componental and tight. Moreover, its bags are all finite and hence tough, so by \cref{lem:TightCompToughTDNearlyDisplaysCrit} there exists for every $X \in \crit(G)$ a node $t_X \in T$ such that cofinitely many tight components of $G-X$ are of the form $G\strictup e$ for an edge $e = t_Xt \in T$ with $t_X <_T t$. 

    Let $(T', \cV')$ be the \td\ induced by $(T, \cV)$ given by contracting all edges $e \in T$ with $V_e \notin \crit(G)$. We claim that $(T', \cV')$ is as desired. It is immediate from the construction that $(T', \cV')$ is still tight and componental and that all its adhesion sets are critical vertex sets of~$G$. So we are left to show that all the torsos of $(T', \cV')$ are tough. 
    To this end, for $X \in \crit(G)$, let $t'_X$ be the node of $T'$ whose branch set in $T$ contains $t_X$. Since we only contracted edges of~$T$ whose adhesion set is not a critical vertex set of $G$, we still have that cofinitely many tight components of $G-X$ are of the form $G\strictup e$ for an edge $e = t'_Xs \in T'$ with $t'_X <_{T'} s$. Hence, we may apply \cref{lem:DisplayingCritYieldsToughTorsos}, which concludes the proof.

    To verify that all bags are distinct, let nodes $t' \neq s' \in T'$ be given, and pick nodes $t \in t'$ and~$s \in s'$. In particular, $t \in V_t \subseteq V'_{t'}$ and $s \in V_s \subseteq V'_{s'}$ by the definition of $(T, \cV)$ and $(T', \cV')$. If either $t \notin V'_{s'}$ or $s \notin V'_{t'}$, then $V'_{t'} \neq V'_{s'}$ and we are done; so suppose otherwise. Then, again by construction, there are nodes $y \in t'$ and $x \in s'$ such that $y \geq_{T} s$ and $x \geq_T t$. But since the branch sets $t'$ and $s'$ are connected in $T$ and disjoint, this contradicts that $T$ is a tree. 

    Finally, $(T, \cV)$ satisfies~\cref{itm:CritDecompWithoutDisplayingCrit:progress} by construction, as we have $t \in V_t \setminus V_e$ for every node $t \in T$ and its unique edge $e=st \in T$ with $s <_T t$. As $(T', \cV')$ was induced by edge-contractions on $T$, also~$(T', \cV')$ satisfies \cref{itm:CritDecompWithoutDisplayingCrit:progress}.
\end{proof}

Next, we describe how one can turn the \td\ from \cref{lem:CritDecompWithoutDisplayingCrit} into one that additionally displays all critical vertex sets and their tight components cofinitely.

\begin{construction} \label{constr:TreeDecompDisplayingCrit}
    Let $(T, \cV)$ be a rooted \td\ of a graph $G$ of finite adhesion.
    For every $X \in \crit(G)$, let $t_X$ be the $\leq_T$-minimal node $t$ with $X \subseteq V_t$ (which exists by \cref{lem:TightCompToughTDNearlyDisplaysCrit}).
    We obtain the tree $T'$ from $T$ by simultaneously adding for each $X \in \crit(G)$ the node $t'_X$, the edge~$t_Xt'_X$ and rerouting each edge $t_Xt$ of  $T$ with $V_{t_Xt} = X$ to $t'_Xt$ in $T'$.
    The \td~$(T', \cV')$ is given by $V'_t := V_t$ for all $t \in T$ and $V'_{t'_X} := X$ for every $X \in \crit(G)$.
\end{construction}

To show that the \td\ $(T', \cV')$ from \cref{constr:TreeDecompDisplayingCrit} displays the critical vertex sets of $G$ we need the following auxiliary lemma.

\begin{lemma}\label{lem:CofinManyCompsYieldsDisplayingCrit}
    Let $(T, \cV)$ be a tight, rooted \td\ of a graph $G$ such that for every $X \in \crit(G)$ there is a unique node $t_X \in T$ with $V_{t_X} = X$ and cofinitely many tight components of~$G-X$ are some $G\strictup e$ for an edge $e = t_Xt \in T$ with $t_X <_T t$. 
    Then $(T, \cV)$ displays all critical vertex sets of $G$.
\end{lemma}
\begin{proof}
    We need to show for every infinite-degree node $t \in T$ with finite $V_t$ that $V_t \in \crit(G)$. 
    So let $t \in T$ be a node with finite $V_t$ and $V_t \notin \crit(G)$. 
    Since $V_t$ is finite, some subset $X \subseteq V_t$ must be the adhesion set corresponding to infinitely many edges $e$ incident with $t$ in $T$. 
    Since $(T, \cV)$ is tight, $X$ is a critical vertex set of $G$.
    By assumption on $(T, \cV)$ there is a unique node $t_X \in T$ with $V_{t_X} = X$. 
    If $t_X = t$ we are done so suppose otherwise. 
    Since by assumption only finitely many tight components of $G-X$ can meet $G\strictdown f$ for the unique edge $f = st_X$ with $s <_T t_X$, we have that $t_X <_T t$. 
    But then, again by the assumption on $(T, \cV)$, at most finitely many edges $e = ts$ with $t <_T s$ can contain a tight component of $G-X$ contradicting that $(T, \cV)$ is tight.
\end{proof}

We can now show that the \td\ $(T', \cV')$ from \cref{constr:TreeDecompDisplayingCrit} is as desired for \cref{thm:critVtxTechnical:copy} if we start with a \td\ $(T, \cV)$ that is tight and componental.

\begin{lemma}\label{lem:TightCompYieldsDisplayingCrit}
    Let $(T, \cV)$ be a tight, componental, rooted \td\ of a graph $G$ of finite adhesion whose torsos are all tough.  
    Then the \td\ $(T', \cV')$ from \cref{constr:TreeDecompDisplayingCrit} displays all critical vertex sets and their tight components cofinitely. 
    Additionally, 
    \begin{enumerate}
        \item\label{itm:TCTtoFullyTight} $(T', \cV')$ is fully tight,
        \item\label{itm:TCTtoCofinallyComponental} $(T', \cV')$ is cofinally componental; moreover, if $G \strictup e$ is disconnected for $e = s t \in E(T')$ with $s <_{T'} t$ then~$V_{s} \supseteq V_{t} \in \crit(G)$ and $\deg(t) = \infty$,
        \item\label{itm:TCTtoTough} the torsos of $(T', \cV')$ are tough.
    \end{enumerate}
\end{lemma}

\begin{proof}
    Let $(T', \cV')$ be the \td\ from \cref{constr:TreeDecompDisplayingCrit}. 
    Then \cref{itm:TCTtoFullyTight}, \cref{itm:TCTtoCofinallyComponental} and \cref{itm:TCTtoTough} hold for $(T', \cV')$.
    \cref{lem:TightCompToughTDNearlyDisplaysCrit} applied to $(T, \cV)$ 
    ensures that for $X \in \crit(G)$ cofinitely many tight components of $G-X$ are of the form $G \strictup e$ for an edge $e=t'_Xt \in T'$ with $t'_X <_{T'} t$. Thus, by \cref{lem:CofinManyCompsYieldsDisplayingCrit}, $(T', \cV')$ displays all critical vertex sets of $G$. 
    Moreover, by construction, $(T', \cV')$ displays the tight components of all critical vertex sets cofinitely.
\end{proof}

\begin{proof}[Proof of \cref{thm:critVtxTechnical:copy}]
    Apply \cref{constr:TreeDecompDisplayingCrit} to the \td\ $(T, \cV)$ from \cref{lem:CritDecompWithoutDisplayingCrit}. By \cref{lem:TightCompYieldsDisplayingCrit} this \td\ $(T', \cV')$ is as desired;
    in particular, it satisfies~\cref{itm:critVtxTechnical:copy:CofinComp}. 
    Moreover, since $(T, \cV)$ satisfies \cref{itm:CritDecompWithoutDisplayingCrit:progress} from \cref{lem:CritDecompWithoutDisplayingCrit}, $(T, \cV)$ still satisfies it at all $t \in T'$ which have not been not added in \cref{constr:TreeDecompDisplayingCrit}, i.e.\ $(T',\cV')$ satisfies \cref{itm:critVtxTechnical:copy:DistinctBags}.
\end{proof}

We remark that a Theorem similar to \cref{thm:critVtxTechnical:copy} was proven by Elm and Kurkofka \cite{infinitetangles}*{Theorem~2}. However, their result cannot be used in place our  \cref{thm:critVtxTechnical:copy}, as their method in general produces only a nested set of separations, and not a tree-decomposition\arXivOrNot{}{ (see the appendix of the arXiv version of this paper for a detailed discussion)}. \arXivOrNot{We refer the reader to the \cref{subsec:Counterexamples} for a detailed discussion of the differences between the two results.}{}

\section{Lifting paths and rays from torso}\label{sec:liftingspathsandraysfromtorso}

From \cref{thm:critVtxTechnical}  we obtain a \td\ $(T, \cV)$ whose torsos are tough.
Following our overall proof strategy for \cref{main:LinkedTightCompTreeDecompnew} we want to further decompose each of its torsos $\torsostar(\sigma_t)$ with a \td\ $(T^t, \cV^t)$ which is linked and rayless.
We have already indicated in the beginning of \cref{subsec:ATwoStepApproach} that by using the infinite connectivity of critical vertex sets the arising \td\ $(T', \cV')$ is linked, if $(T^t, \cV^t)$ is linked.
For this, we need to extend the path families in $\torsostar(\sigma_t)$ which witness the linkedness of $(T^t, \cV^t)$ to path families in $G$ joining up the same sets of vertices.
In this section we show that such an extension of (finite) path families is possible, as long as the separations in the star $\sigma_t$ at $t$ are `left-well-linked' (\cref{lem:DisjointPathsExtendedByWellLinked}).
Moreover, we prove similar results, \cref{prop:OnePathOrRayExtended} and \cref{lem:InfManyDisjPathExtended}, for extending rays and infinite path families  in torsos at stars whose separations are just left-tight.

We recall that for two sets $X,Y$ of vertices of a graph $G$, an \defn{$X$--$Y$ path} meets $X$ precisely in its first vertex and $Y$ precisely in its last vertex.
A path \defn{through} a set $C$ of vertices in~$G$ is a path in $G$ whose internal vertices are contained in $C$.

\begin{proposition} \label{prop:OnePathOrRayExtended}
    Let $\sigma$ be a star of left-tight, finite-order separations of a graph $G$.
    Let~$P$ be a path or ray in $\torsostar(\sigma)$.
    Then there exists a path or ray $P'$ in $G$, respectively, with $P' \cap G[\interior(\sigma)] \subseteq P$ which starts in the same vertex as $P$ and, if $P$ is a path, also ends in the same vertex as $P$, and such that $P'$ meets $V(P)$ infinitely often if $P$ is ray.
\end{proposition}

\begin{proof}
    Fix for every torso edge $e = uv$ of $P$ a separation $(A_e, B_e) \in \sigma$ such that both $u, v \in A_e \cap B_e$. Since each $(A_e, B_e)$ is left-tight, we may further fix for every torso edge $e = uv$ of $P$ a $u$--$v$ path~$P_e$ through $A_e\setminus B_e$. 
    Note that since all $(A,B) \in \sigma$ have finite order, we have $(A_e, B_e) = (A_f, B_f)$ for at most finitely many torso edges $f$ of $P$.
    As the strict left sides $A \setminus B$ of the separations~$(A,B)$ in the star $\sigma$ are pairwise disjoint, every $P_e$ meets at most finitely many $P_f$ for the torso edges~$f$ of $P$.
    
    Then the graph $H$ obtained from $P \cap G$ by adding all the $P_e$ for torso edges of $P$ is a connected subgraph of $G$ which contains the startvertex and, if $P$ is a path, the endvertex of $P$. 
    Moreover,~$H$ is locally finite: by construction, all vertices in $P \cap G$ have degree $1$ or $2$ in $H$ since $P$ is a path or ray, and all vertices in $H - V(P)$ are contained in at most finitely many $P_e$ by the argument above, and hence also have finite degree in $H$. 
    Now if $P$ is a path, then $H \subseteq G$ contains a path~$P'$ whose endvertices are the same as $P$, and if $P$ is a ray, then $H \subseteq G$ contains a ray $P'$ that starts in the same vertex as $P$ (cf.\ \cite{bibel}*{Proposition~8.2.1}). In particular, if $P$ is a ray, then~$P'$ meets~$V(P)$ infinitely often since each component of $H - V(P)$ is finite.
    
    Since $H \cap G[\interior(\sigma)] = P \cap G[\interior(\sigma)]$ by construction, and because $P' \subseteq H$, it follows that $P' \cap G[\interior(\sigma)] \subseteq P$, so $P'$ is as desired.
\end{proof}

We remark that later in this paper we sometimes want to apply \cref{prop:OnePathOrRayExtended} to a ray $P$ in the torso at a star $\sigma$ of finite-order separations of a graph $G$ in which all but one separation $(A,B)$ are left-tight but we are not interested in keeping the startvertex of $P$. Hence, we may apply \cref{prop:OnePathOrRayExtended} instead to the tail of $P$ which avoids the finite set $A \cap B$ and the star $\{(C \cap B, D \cap B) \mid (C,D) \in \sigma \setminus \{(A,B) \}\}$ induced by $\sigma \setminus \{(A,B) \}$ on $G[B]$.
This still yields a ray $P'$ in $G$ with $P' \cap G[\interior(\sigma)] \subseteq P$ which meets $V(P)$ infinitely often.

\begin{lemma} \label{lem:InfManyDisjPathExtended}
    Let $\sigma$ be a star of left-tight, finite-order separations of a graph $G$. Let $X, Y \subseteq \interior(\sigma)$ such that there are infinitely many disjoint $X$--$Y$ paths in $\torsostar(\sigma)$. Then there are infinitely many $X$--$Y$ paths in $G$.
\end{lemma}

\begin{proof}
    Let $\cP$ be an infinite family of disjoint $X$--$Y$ paths in $\torsostar(\sigma)$. We define an infinite family~$\cP'$ of disjoint $X$--$Y$ paths in $G$ recursively. 
    For this, set $P'_0 := \emptyset$, let $n \in \N$, and assume that we have already constructed a family $\cP'_n$ of $n$ pairwise disjoint $X$--$Y$ paths in $G$.

    Then $V(\cP'_n)$ is finite, and hence meets the separators of at most finitely many separations in~$\sigma$. Since all separations in $\sigma$ have finite order, and because $\cP$ is an infinite family of disjoint paths, there exists a path $P \in \cP$ that avoids both the finite set $V(\cP'_n)$ and the finitely many finite separators of separations in $\sigma$ that meet $V(\cP'_n)$.
    By \cref{prop:OnePathOrRayExtended}, there exists a path~$P'$ in $G$ with the same endvertices as $P$ and with $P' \cap G[\interior(\sigma)] \subseteq P$. In particular, $P'$ is an $X$--$Y$ path in $G$. Moreover, $P'$ is disjoint from the paths in $\cP'_n$ by the choice of $P$ and because $P' \cap G[\interior(\sigma)] \subseteq P$.
    Hence, we may define $\cP'_{n+1} := \cP'_n \cup \{P'\}$.
    
    Then $\cP' := \bigcup_{n \in \N} \cP'_n$ is as desired.
\end{proof}

Let us call a finite-order separation~$(A,B)$ of a graph~$G$ \defn{left-well-linked} if, for every two disjoint sets~$X, Y \subseteq A \cap B$, there is a family of~$\min\{|X|, |Y|\}$ disjoint $X$--$Y$~paths in $G$ through $A \setminus B$.

\begin{lemma} \label{lem:DisjointPathsExtendedByWellLinked}
    Let~$\sigma$ be a star of left-well-linked, finite-order separations of a graph~$G$. Then the following assertions hold:
    \begin{enumerate}[label=\rm{(\roman*)}]
        \item \label{itm:DisjRaysExtended} For every countable family $\cP$ of disjoint rays in $\torsostar(\sigma)$ there exists a family $\cP'$ of disjoint rays in $G$ with the same set of startvertices, and such that each ray in $\cP'$ meets~$V(\cP)$ infinitely often.
        \item \label{itm:DisjPathsExtended} For every $k \in \N$ and every family $\cP$ of $k$ disjoint paths in $\torsostar(\sigma)$ there exists a family $\cP'$ of $k$ disjoint paths in $G$ with the same set of startvertices and the same set of endvertices.
    \end{enumerate}
\end{lemma}

\begin{proof}
    We prove \cref{itm:DisjRaysExtended} and \cref{itm:DisjPathsExtended} simultaneously; so let $\cP$ be as in \cref{itm:DisjRaysExtended} or as in \cref{itm:DisjPathsExtended}.
    We define the family~$\cP'$ recursively, going through the countably many torso edges on~$\cP$.
    For this, fix an enumeration~$e_1, e_2, \dots$ of the torso edges in~$\cP$ such that if~$e_i \in P \in \cP$, then all torso edges occurring on~$P$ before~$e_i$ are enumerated as~$e_j$ for some~$j < i$. Further, for every $e_i$, we fix some $(A_i, B_i) \in \sigma$ such that both endvertices of $e_i$ are contained in $A_i \cap B_i$ where we choose $(A_i, B_i) = (A_j, B_j)$ for some $j < i$ if possible. 
    
    We start the recursion with~$\cP_0 := \cP$.
    At step~$i \in \N$, we assume that we have already constructed a family $\cP_{i-1}$ of disjoint paths/rays in $\torsostar(\sigma) \cup \bigcup_{j<i} G[A_j]$ with the same startvertices, and, if~$\cP$ is as in \cref{itm:DisjPathsExtended}, with the same endvertices, as~$\cP$, and without all torso edges whose endvertices are both contained in $A_j \cap B_j$ for some $j<i$.
    We now consider the torso edge~$e_i$.
    If~$e_i$ is not contained in any path/ray in~$\cP_{i-1}$, then we set $\cP_i := \cP_{i-1}$.
    Otherwise, for each~$P \in \cP_{i-1}$, let~$x_P$ be the first vertex on~$P$ such that the subsequent edge on~$P$ is a torso edge with both endvertices in~$A_i \cap B_i$, and let~$y_P$ be the respective last vertex on~$P$ such that its previous edge on~$P$ is a torso edge with both endvertices in~$A_i \cap B_i$.
    Now set~$X_i := \{x_P \mid P \in \cP_i\}$ and~$Y_i := \{y_P \mid P \in \cP_i\}$.
    Since the paths/rays in~$\cP_{i-1}$ are disjoint, $X_i$ and~$Y_i$ are disjoint subsets of~$A_i \cap B_i$ with~$k_i := |X_i| = |Y_i|$.
    We can thus use that~$(A_i, B_{i})$ is left-well-linked to find a family~$\cQ$ of~$k_{i}$ disjoint $X_{i}$--$Y_{i}$~paths in~$G[(A_{i}\setminus B_{i}) \cup X_{i} \cup Y_{i}]$; for~$Q \in \cQ$, write~$x_Q$ for its first and~$y_Q$ for its last vertex.
    
    For each~$P \in \cP_i$, we now define~$P^* := P x_P x_Q Q y_Q y_{P'} P'$ where~$Q$ is the unique path in~$\cQ$ with~$x_Q = x_P$ and~$P'$ is the unique path/ray in~$\cP'$ with~$y_{P'} = y_Q$.
    By construction, the set $\cP_{i} := \{P^* \mid P \in \cP_{i-1}\}$ is a family of disjoint paths/rays in~$\torsostar(\sigma) \cup \bigcup_{j \leq i} G[A_j]$ with the same startvertices, and, if $\cP$ is as in \cref{itm:DisjPathsExtended}, with the same endvertices, as~$\cP_{i-1}$ and thus as~$\cP$.
    Moreover,~$\cP_i$ not only avoids all torso edges $e_j$ with $j <i$ but also all torso edges whose endvertices are both contained in $A_j \cap B_j$ for some $j <i$.
    This completes step~$i$.

    To define~$\cP'$ from the~$\cP_i$, let~$X$ be the set of startvertices of paths/rays in~$\cP$ and let~$Y$ be the set of endvertices of paths in~$\cP$.
    By construction, for each vertex~$x \in X$ there is  a (unique) path/ray in~$\cP_i$ starting in~$x$ and we denote it by~$P^x_i$.
    We let~$P^x$ be the limit $\liminf_{i \in \N} P^x_i$ of the~$P^x_i$, and define~$\cP' := \{ P^x \mid x \in X\}$.
    By construction, all the~$P^x$ are disjoint.
    Note that if $\cP$ is as in \cref{itm:DisjPathsExtended}, then $\cP' = \cP_n$ for some $n \in \N$, since the construction yields $P_i^x = P_j^x$ for all $i, j \geq |E(\cP)|$; thus $\cP'$ is as desired for \cref{itm:DisjPathsExtended}.
    We now prove that if~$\cP$ is as in \cref{itm:DisjRaysExtended}, then all $P^x$ are indeed rays which meet $V(\cP)$ infinitely often.

    If an initial segment of some~$P^x_i$ is contained in~$G$, then it contains no torso edge and it hence remains untouched by the above construction in all steps~$j \ge i$; in other words, this initial segment of~$P^x_i$ in~$G$ is also an initial segment of all~$P^x_j$ with~$j \ge i$.
    Moreover, if~$P^x_i$ still contains some torso edge and we let~$e_j$ be the first such one occurring along~$P^x_i$, then the construction at step~$j$ implies that the maximal initial segment of~$P^x_{j}$ in~$G$ is strictly longer than the one in~$P^x_i$.
    Moreover, it contains a vertex in~$V(\cP)$, the respective~$y_{(P^x_j)'}$ in the above construction step, that has not been contained in the maximal initial segment of~$P^x_i$ in~$G$.

    Now if~$P^x_i$ contains no torso edge at some step~$i$, then it is a ray in~$G$ starting in~$x$, and thus some tail of $P^x_i$ equals the tail of some ray in~$\cP$ by construction; in particular, $P^x_i$ meets $V(\cP)$ infinitely often.
    Otherwise, the length of the initial segment of~$P^x_i$ in~$G$ strictly increases infinitely often, and hence the limit~$P^x$ of the~$P^x_i$ is a ray in~$G$ starting in~$x$ that meets infinitely many vertices of~$V(\cP)$, which is both witnessed by the respective $y_{(P^x_j)'}$ described above.
    Thus, $\cP'$ is as desired if $\cP$ is as in \cref{itm:DisjRaysExtended}. This concludes the proof.
\end{proof}

\section{Linked tree-decompositions into rayless parts} \label{sec:DecomposeIntoRaylessParts}

In this section we prove 
\cref{thm:RaylessThmTechnical}, which we restate here for convenience.

\begin{customthm}{\cref{thm:RaylessThmTechnical}}[Detailed version of~\cref{thm:RaylessThmIntro}] \label{thm:RaylessThmTechnical:copy}
    Let~$G$ be a graph, and let~$\sigma$ be a star of left-well-linked, left-fully-tight, finite-order separations of~$G$ such that $\torsostar(\sigma)$ has finite tree-width. 
    Further, let~$X\subseteq \interior(\sigma)$ be a prescribed finite set of vertices of~$G$. 
    
    Then~$G$ admits a linked $X$-linked, fully tight, rooted \td\ $(T, \cV)$ of finite adhesion 
    such that
    \begin{enumerate}[label=\rm{(R\arabic*)}]
        \item \label{itm:RaylessThmTechnical:rayless} its torsos at non-leaves are rayless and its leaf separations are precisely~$\{(B, A) \mid (A, B) \in \sigma\}$, 
        \item \label{itm:RaylessThmTechnical:EndAdhesionLinked} for all edges $e$ of $T$, the adhesion set $V_e$ is either linked to an end living in $G \up e$ or linked to a set $A \cap B \subseteq G \up e$ with $(A,B) \in \sigma$,
        \item \label{itm:RaylessThmTechnical:IncDis} for every $e <_T e' \in E(T)$ with $|V_e| \leq |V_{e'}|$, each vertex of $V_e \cap V_{e'}$ either dominates some end of~$G$ that lives in~$G \up e'$ or is contained in $A \cap B \subseteq V(G \up e')$ for some $(A,B) \in \sigma$, 
        \item \label{itm:RaylessThmTechnical:DistinctBags}
        for all edges $e = st \in T$ with $s <_T t$ and $s \neq r := \rt(T)$ we have $V_s \supsetneq V_e \subsetneq V_t$. Moreover, if $X \subsetneq\interior(\sigma)$, $G-X$ is connected and $N_G(G-X) = X$, then $X \subsetneq V_r$ and also $V_r \supsetneq V_e \subsetneq V_t$ for all edges $e = rt \in T$, and
        \item \label{itm:RaylessThmTechnical:componental} if $G \strictup e$ is disconnected for some edge $e \in T$, then $e$ is incident with a leaf of $T$.
    \end{enumerate}
\end{customthm}

\noindent We remark that in \cref{thm:RaylessThmTechnical:copy} we explicitly allow the case $\sigma = \emptyset$, where we go with the convention that the interior of the empty star is $V(G)$.

First we reduce \cref{thm:RaylessThmTechnical:copy} to the following statement which differs slightly from \cref{thm:RaylessThmTechnical:copy}.
A separation $(A,B)$ of a graph $G$ is \defn{left-connected} if $G[A \setminus B]$ is connected, and it is \defn{left-end-linked} if  $A \cap B$ is linked to some end $\eps$ of~$G$ with $C(A\cap B, \eps) \subset A$.

\begin{theorem}\label{thm:RaylessThmTechnicalRegions}
    Let~$G$ be a graph of finite tree-width, let~$\sigma$ be a star of left-end-linked left-connected finite-order separations of~$G$, and let~$X\subseteq \interior(\sigma)$ be a finite set of vertices of~$G$. 
    Assume further that for every~$(A, B) \in \sigma$, the graph~$G[A \cap B]$ is complete and~$A \cap B \subseteq \Dom(G[A\setminus B])$.

    Then~$G$ admits a linked, $X$-linked, tight, componental, rooted \td\ $(T, \cV)$ of finite adhesion such that
    \begin{enumerate}[label=\rm{(R\arabic*')}]
        \item \label{itm:RaylessThmTechnicalRegions:rayless} its torsos at non-leaves are rayless, its leaf separations are precisely~$\{(B, A) \mid (A, B) \in \sigma\}$ and no other separation induced by an edge of $T$ is $(B,A)$ for $(A,B) \in \sigma$,
        \item \label{itm:RaylessThmTechnicalRegions:EndLinked} $(T, \cV)$ is end-linked,
        \item \label{itm:RaylessThmTechnicalRegions:IncDis} for every $e <_T e' \in E(T)$ with $|V_e| \leq |V_{e'}|$ each vertex of $V_e \cap V_{e'}$ dominates some end of~$G$ that lives in~$G \up e'$, and
        \item \label{itm:RaylessThmTechnicalRegions:DistinctBags}
        $(T, \cV)$ satisfies \cref{itm:RaylessThmTechnical:DistinctBags} from \cref{thm:RaylessThmTechnical}.
    \end{enumerate}
\end{theorem}

\begin{proof}[Proof of \cref{thm:RaylessThmTechnical:copy} given \cref{thm:RaylessThmTechnicalRegions}]
    Let $H$ be the graph obtained from $\torsostar(\sigma)$ by adding, for every separation $(A,B) \in \sigma$ a disjoint ray $R_{A,B}$ as well as an edge from every vertex $v \in A \cap B$ to every vertex of $R_{A,B}$. We aim to apply \cref{thm:RaylessThmTechnicalRegions} to $H$ with $X$ and the star 
    $$\sigma' := \{\big(V(R_{A,B}) \cup (A \cap B), V(H-R_{A,B}) \big) \mid (A,B) \in \sigma\};$$ but to be able to do so we first have to show that $H$ has finite tree-width. 

    By assumption in \cref{thm:RaylessThmTechnical:copy}, $\torsostar(\sigma)$ admits a \td\ $(T^\sigma, \cV^\sigma)$ into finite parts. Any subgraph of $H$ of the form $H[(A \cap B) \cup V(R_{A,B})]$ clearly also admits such a \td\ $(T^{A,B}, \cV^{A,B})$.
    Since $H[A \cap B]$ is complete for all $(A,B) \in \sigma$, it is contained in some part of $(T^\sigma, \cV^\sigma)$ and also it is contained in some part of $(T^{A,B}, \cV^{A,B})$. 
    We then obtain the desired decomposition tree from the disjoint union of the decomposition trees by adding for each $(A,B) \in \sigma$ an edge between the nodes corresponding to the respective parts containing $H[A \cap B]$.
    Keeping the parts yields the desired \td\ of $H$ into finite parts.

    So by construction and the previous argument, $H$, $X$ and $\sigma'$ are as required for \cref{thm:RaylessThmTechnicalRegions}, which then yields a rooted \td\ $(T, \cV')$ of $H$. In particular, this $(T, \cV')$ has precisely $((B,A) \mid (A,B) \in \sigma')$ as its leaf separations. Thus, we obtain a \td\ $(T, \cV)$ of $G$ by letting $V_t := V'_t$ for all non-leaves of $T$ and $V_t := B$ for all leafs of $T$ whose bag is of the form $(A \cap B) \cup V(R_{A,B})$. In particular, its leaf separations are precisely $\{(B,A) \mid (A,B) \in \sigma\}$.

    We claim that the \td\ $(T, \cV)$ of $G$ is as desired.
    We remark that the adhesion sets corresponding to an edge of the decomposition tree are unchanged.
    So, $(T, \cV)$ still has finite adhesion, as $(T, \cV')$ has finite adhesion.
    Also, $(T, \cV)$ is linked and $X$-linked: as the $H[A \cap B]$ are complete, this follows from the ($X$-)linkedness of $(T, \cV')$ and \cref{lem:DisjointPathsExtendedByWellLinked}~\ref{itm:DisjPathsExtended}.
    Further, $(T, \cV)$ is fully tight, since $(T, \cV')$ is fully tight and all separations in $\sigma$ are left-fully-tight.
    Also $(T, \cV)$ satisfies \cref{itm:RaylessThmTechnical:componental}, since $(T, \cV')$ is componental and the separations in $\sigma$ are fully tight and no separation of $H$ induced by an edge of $T$ which is not incident with a leaf is $(B,A)$ for some $(A,B) \in \sigma'$ by \cref{itm:RaylessThmTechnicalRegions:rayless}. 
    By construction of $(T, \cV)$, property \cref{itm:RaylessThmTechnicalRegions:rayless} of $(T, \cV')$ immediately implies that $(T, \cV)$ satisfies \cref{itm:RaylessThmTechnical:rayless}.
    Additionally, for all edges~$e$ of~$T$, the adhesion set $V'_e$ is linked to an end of $H$ by \cref{itm:RaylessThmTechnicalRegions:EndLinked}. If that end corresponds to some ray $R_{A,B}$, then $V_e = V'_e \subseteq \interior(\sigma)$ is linked to $A \cap B \subseteq V(G \up e)$ since $A \cap B$ separates $\interior(\sigma)$ from~$R_{A,B}$ and because of \cref{lem:DisjointPathsExtendedByWellLinked}~\ref{itm:DisjPathsExtended} as the $H[A \cap B]$ are complete. 
    Otherwise, \cref{lem:DisjointPathsExtendedByWellLinked}~\cref{itm:DisjRaysExtended} and~\cref{lem:InfManyDisjPathExtended} ensure that~$V_e$ is linked to an end of $G$ that lives in $G \up e$.  Indeed, let $\cR$ be a family of $|V_e|$ equivalent, disjoint rays in $\torsostar(\sigma)$ that start in $V_e$ and that witness that $V_e$ is end-linked. Then \cref{lem:DisjointPathsExtendedByWellLinked}~\ref{itm:DisjRaysExtended} yields a family $\cR'$ of $|V_e|$ disjoint rays in $G\up e$ that start in $V_e$ and that each meet~$V(\cR)$ infinitely often. In particular, since $\cR$ is finite, for every $R' \in \cR'$ there is $R \in \cR$ such that~$R'$ meets $V(R)$ infinitely often. To see that the rays in $\cR'$ are equivalent, let $R'_0, R'_1 \in \cR'$ be given. Since $R_0$ and $R_1$ are equivalent in $\torsostar(\sigma)$, the infinite sets $V(R'_0) \cap V(R_0)$ and $V(R'_1) \cap V(R_1)$ cannot be separated by finitely many vertices. Hence, we may greedily pick infinitely many disjoint $V(R'_0) \cap V(R_0)$--$V(R'_1) \cap V(R_1)$ paths in $\torsostar(\sigma)$. Since these paths are in fact $R'_0$--$R'_1$ paths in $\torsostar(\sigma)$, \cref{lem:InfManyDisjPathExtended} yields that there are infinitely many disjoint $R'_0$--$R'_1$ paths in $G$, which concludes the proof that $V_e$ is end-linked in $G$, and that $(T, \cV)$ satisfies \cref{itm:RaylessThmTechnical:EndAdhesionLinked}.
    
    Also $(T, \cV)$ satisfies \cref{itm:RaylessThmTechnical:IncDis} because $(T, \cV')$ satisfies \cref{itm:RaylessThmTechnicalRegions:IncDis}.
    Finally, $(T, \cV)$ satisfies \cref{itm:RaylessThmTechnical:DistinctBags} since~$(T, \cV')$ satisfies \cref{itm:RaylessThmTechnicalRegions:DistinctBags} too.
\end{proof}

Let us briefly sketch the proof of \cref{thm:RaylessThmTechnicalRegions}.
We will construct the \td~inductively. We start with the trivial \td~whose decomposition tree is a single vertex whose bag is the whole vertex set of $G$. 
In the induction step, we assume that we have already constructed a linked, $X$-linked, tight, componental, rooted \td~$(T^n, \cV^n)$ of $G$ of finite adhesion which satisfies \cref{itm:RaylessThmTechnicalRegions:EndLinked}, \cref{itm:RaylessThmTechnicalRegions:IncDis} but it may not satisfy \cref{itm:RaylessThmTechnicalRegions:rayless}.
Instead its torsos at non-leaves are rayless and $A$ is contained in a bag at a leaf of $T^n$ for every $(A,B) \in \sigma$. 
Then, for every leaf $\ell$ of $T^n$ whose (unique) incident edge $e$ does not induce a separation in $\sigma$, we define a set $Y \subseteq V(G \up e)$ such that the adhesion set $V^n_e$ is contained in $Y$.
We then replace the bag $V^n_\ell$ with $Y$ and add for each component $C$ of $(G \up e) - Y$ a new leaf to $\ell$ and associate the bag $V(C) \cup N(C)$ to it. 
By carefully choosing these sets $Y$ for each such leaf we will ensure that the arising \td~$(T^{n+1}, \cV^{n+1})$ again satisfies all the properties as assumed for $(T^n, \cV^n)$. 
We then show that the pair $(T, \cV)$ which arises as the limit for $n \rightarrow \infty$ is indeed a \td~of~$G$ and that it satisfies all the desired properties.

This section is organised as follows. First, we describe in \cref{subsec:BuildingParts} the \cref{recursion:buildingpackage} that will construct the bags of the desired \td\ for \cref{thm:RaylessThmTechnicalRegions}, that is, the sets $Y$ mentioned above.  This algorithm is simple to state, but we require some tools to properly analyse it.
In \cref{subsec:regions} we build a set of tools centred around `regions', which we then use in \cref{subsec:AnalysisOfRecursion} to prove some properties of \cref{recursion:buildingpackage} that will later on ensure that the resulting \td~is as desired. 
Finally, in \cref{subsec:ProofOfRaylessThm}, we follow the above described approach to construct a \td\ by inductively applying \cref{recursion:buildingpackage} and then prove with the help of the main result from \cref{subsec:AnalysisOfRecursion} that this \td\ is as desired.

\subsection{Building the bags of the \td} \label{subsec:BuildingParts}
In this section we describe a transfinite recursion, \cref{recursion:buildingpackage}, that will construct the bags of the \td~for \cref{thm:RaylessThmTechnicalRegions}. For this, we need the following definitions.

Let $G$ be a graph.
A \defn{region} $C$ of $G$ is a connected subgraph of $G$. 
By $\Bar{C}$ we denote the \defn{closure} $G[V(C) \cup N_G(C)]$ of $C$.
A \defn{$k$-region} for $k \in \N$ is a region $C$ whose neighbourhood $N(C)$ has size~$k$.
A \defn{$\lk$-} or \defn{$\lek$-region} for $k \in \N$ is a $k'$-region $C$ for some $k' < k$ or $k' \leq k$, respectively.
Similarly, an \defn{$\lA$-region} is a $k$-region for some $k \in \N$.
Two regions $C$ and $D$ of $G$ \defn{touch}, if they have a vertex in common or $G$ contains an edge between them.
Two regions $C$ and $D$ of a graph~$G$ are \defn{nested} if they do not touch, or $C \subseteq D$, or $D \subseteq C$.
Note that the set of all regions~$G \strictup e$ given by a rooted componental \td\ of a graph~$G$ is nested.

A region $C$ is \defn{$\eps$-linked} for an end $\eps$ of $G$ if $\eps$ lives in $C$ and the neighbourhood of~$C$ is linked to~$\eps$.
A region~$C$ is \defn{end-linked} if it is $\eps$-linked for some end $\eps$ of $G$.
We emphasise that a region $C$ is $\eps$-linked if $N(C)$ is linked to the end $\eps$ (and not $V(C)$). 
To distinguish both cases, we say a region is $\eps$-linked and a set of vertices is linked to $\eps$.

\begin{algo}\label{recursion:buildingpackage}
    (Construction of a bag)
    
    \textbf{Input:} a connected graph $H$; a finite set~$X$ of $k \in \N$ vertices of $H$; a set $\cD$ of pairwise non-touching end-linked $\lA$-regions $D$ that are disjoint from~$X$ and satisfy $X \cap N_H(D) \subseteq \Dom(D)$; a vertex $x \in V(G)\setminus X$ that lies in no $D \in \cD$. 

    \textbf{Output:} a transfinite sequence $C_0, C_1, \dots, C_i, \dots$, indexed by some ordinal $< |H|^+$, of distinct end-linked $\lA$-regions which are disjoint from $X$ and pairwise nested and which are also nested with all $D \in \cD$; a set $Y := V(H) \setminus (\bigcup_{i} C_i)$
    
    \textbf{Recursion:} Iterate the following step:
    \begin{enumerate}[labelindent=3.9em, label=\textbf{Case \Alph*:}, ref=Case~\Alph*, leftmargin=!] 
        \item \label{RegionsCaseA} If there is a $\lk$-region $C_i$ of $H$ that is disjoint from $X$, is $\eps$-linked for some end $\eps$ which lives in no~$C_j$ for $j<i$, nested with the $C_j$ for $j < i$ and with all $D \in \cD$, then choose a \defn{nicest} such region $C_i$.
        Here, nicest\footnote{Note that there might be several nicest regions.} means that
        \begin{enumerate}[labelindent=1.5em, label=(N\arabic*), leftmargin=!]
            \item \label{nicest:ell_iminimum} $C_i$ is such an $\ell_i$-region where $\ell_i \in \N$ is minimum among such regions, and
            \item \label{nicest:inclusion-wisemaximal} $C_i$ is an inclusion-wise maximal such region subject to \cref{nicest:ell_iminimum}.
        \end{enumerate}
    
    \item \label{RegionsCaseB} If there is no region as in \ref{RegionsCaseA}, but there is a $\lA$-region $C_i$ that is disjoint from $X':=X \cup \{x\}$, is $\eps$-linked for some end $\eps$ which lives in no $C_j$ for $j < i$, nested with the $C_j$ for $j < i$ and with all $D \in \cD$ such that $X' \cap N_H(C_i) \subseteq \Dom(C_i)$, then choose a nicest such region $C_i$.

    \item \label{RegionsCaseC} If there is no region as in \ref{RegionsCaseA} or \ref{RegionsCaseB}, then terminate the recursion.
    \end{enumerate}
\end{algo}

Recall that we will construct the \td~for \cref{thm:RaylessThmTechnicalRegions} recursively, where in each step we replace the leaves of the previous \td~with stars whose torsos at the centre vertices are rayless. So we can think of the graph~$H$ in \cref{recursion:buildingpackage} as the part $G \up \ell$ above a leaf $\ell$ of the decomposition tree $T^n$ constructed so far and the set~$X$ as the adhesion set $V_e$ that corresponds to the unique edge $e$ incident with that leaf in $T^n$. The set~$Y$ which \cref{recursion:buildingpackage} outputs will then be the bag at the centre of the newly added star, and the parts at the new leaves will be the components of $H - Y$ (which will be the $\supseteq$-maximal elements of the~$C_i$, see \cref{theorem:propertiesofrecursion:buildingpackage} below) together with their boundaries. 

Let us briefly give some intuition on why the set $Y$ is a good candidate for the new bag.
\cref{thm:RaylessThmTechnicalRegions} requires the torsos of the \td\ to be rayless, so $Y$ should not contain any end of $H$: for this, \cref{recursion:buildingpackage} iteratively cuts off all ends of $H$ by choosing $\lA$-regions~$C_i$ around them.
Additionally, we have to make sure that in the limit of our construction we do  end up with a \td\, which in particular must  satisfy \cref{prop:TD1}.
The specified vertex~$x$ will ensure that we make the appropriate progress when defining the \td s.
In fact, in the final construction of the \td~we will carefully specify the vertex~$x$ in order to ensure that in the end every vertex will lie in some bag of the pair $(T, \cV)$ arising as the limit for $n \rightarrow \infty$.
But this is not the only point where we need to be careful. Since the \td\ should be linked, we cannot just choose the regions $C_i$ arbitrarily; instead, we need to choose those first whose neighbourhood has size less than $|X|$. This is encoded in \cref{RegionsCaseA}, while \cref{RegionsCaseB} then will cut off all the remaining ends of $G$.
We will justify these intuitions in \cref{theorem:propertiesofrecursion:buildingpackage}.

We begin by argueing that \cref{recursion:buildingpackage} is well-defined: Whenever there exists a region in \cref{RegionsCaseA} or \cref{RegionsCaseB}, then there obviously exists such a region as in \cref{RegionsCaseA} or \cref{RegionsCaseB} satisfying \cref{nicest:ell_iminimum}, and then Zorn's Lemma and the following lemma (specifically: item \ref{prop:limitoflekregionsislek}) ensure that there is also such a region satisfying \cref{nicest:inclusion-wisemaximal}.

\begin{lemma}
\label{lem:limitofregionissuchregion}
    Let $H$ be a graph. Let $\cC$ be a chain of regions of $H$.
    Then $C = \bigcup \cC$ is a region, and $\cC' := \{V(H-C) \cap N_H(C') \mid C' \in \cC\}$ is a chain with $N_H(C) = \bigcup \cC'$.

    Moreover, let $X \subseteq V(H)$.
    Then the following statements hold:
    \begin{enumerate}
        \item \label{prop:disjointfromX} If all $C' \in \cC$ are disjoint from $X$, then $C$ is disjoint from $X$.
        \item \label{prop:limitoflekregionsislek} If, for some $k \in \N$, every $C' \in \cC$ is a $\lek$-region, then $C$ is a $\lek$-region. 
        \item \label{prop:boundaryandXshareonlydominating} If, for every $C' \in \cC$, $X \cap N_H(C') \subseteq \Dom (C')$, then $X \cap N_H(C) \subseteq \Dom (C)$.
        \item \label{prop:regionendlinked} If $N_H(C)$ is finite and every $C' \in \cC$ is end-linked, then $C$ is end-linked.
        \item \label{prop:limitisstronglynested} If every $C' \in \cC$ is nested with a region $D$, then also $C$ is nested with every $D$. 
    \end{enumerate}
\end{lemma}

\begin{proof}
    Since $\cC$ is a chain and all the $C' \in \cC$ are connected, $C$ is connected.
    By definition of neighbourhood, $N_H(C) = V(H - C) \cap \bigcup_{C' \in \cC} N_H(C')$ and, since $\cC$ is a chain, the set $\cC' = \{V(H-C) \cap N_H(C') \mid C' \in \cC\}$ forms a chain.
    This ensures that \cref{prop:limitoflekregionsislek} and \cref{prop:boundaryandXshareonlydominating} hold. Next, \cref{prop:disjointfromX} follows immediately from the definition of $C$ as union of the $C' \in \cC$.
    
    \cref{prop:regionendlinked}:
    If the neighbourhood of $C$ is finite, the fact that $\cC'$ is a chain with $\bigcup \cC' = N_H(C)$ yields that there is a~$C' \in \cC$ such that $N_H(C') \supseteq N_H(C)$.
    Now, since $C' \subseteq C$, the end-linkedness of $C'$ yields the end-linkedness of~$C$.

    \cref{prop:limitisstronglynested}:
    Since every $C' \in \cC$ is nested with $D$, the regions $C'$ and $D$ either do not touch or one is contained in the other.
    If $D$ is contained in some $C' \in \cC$, then $D$ is also contained in $C \supseteq C'$, as desired.
    So we may assume that $D$ is contained in no $C' \in \cC$.
    Note that whenever $C'_0 \in \cC$ is contained in $D$ or does not touch~$D$, every $C' \in \cC$ with $C' \subseteq C'_0$ is contained in $D$ or does not touch $D$, respectively.
    Thus, either all~$C' \in \cC$ are contained in $D$ or they all do not touch $D$.
    This yields that their union $C$ either is contained in~$D$ or does not touch $D$, respectively.
\end{proof}

\subsection{Regions} \label{subsec:regions}

In this section we collect some statements about regions which we then use in \cref{subsec:AnalysisOfRecursion} to analyse the regions $C_i$ and the set $Y$ which \cref{recursion:buildingpackage} outputs. 
Recall that we want \cref{recursion:buildingpackage} to cut off all ends of $H$ in that every end of $H$ lives in some $C_i$. In order to prove that \cref{recursion:buildingpackage} actually achieves this (at least in the case where every end of $H$ has countable combined degree), we will work in this section towards \cref{lem:theminiuncrossinglemma,lem:theultimateuncrossing}, which will later on ensure that if at step $i$ of the recursion in \cref{recursion:buildingpackage} there is still an end of $H$ that does not live in any $C_j$ for $j < i$, then there is a region that is a `candidate for $C_i$' in \cref{RegionsCaseA} or \cref{RegionsCaseB}. Then \cref{recursion:buildingpackage} will not terminate as long as there is still some such uncovered end of $G$ that lives in no $C_j$.

We say that a region $C$ is a \defn{candidate for $C_i$} if $C_i$ was chosen in \cref{RegionsCaseA} or \cref{RegionsCaseB} and $C$ satisfied all properties of the regions in \cref{RegionsCaseA} or \cref{RegionsCaseB} at step $i$, respectively, except that $C$ may have not been a nicest such region.
We start by showing for that every end of $H$ there is a region $C_i$ as in in \cref{RegionsCaseA} (\cref{lem:existenceofregion}~\cref{existenceoflessthank}) or \cref{RegionsCaseB} (\cref{lem:existenceofregion}~\cref{existenceofstrictlyincreasing}) -- except that it might not be nested with $\cD$ and with the previously chosen $C_j$ for $j < i$.

\begin{lemma}
\label{lem:existenceofregion}
    Let~$H$ be a graph, let~$X \subseteq V(H)$ be finite, and let $\epsilon$ be an end of~$H$ with countable combined degree.
    Then the following statements hold:
    \begin{enumerate}
        \item \label{existenceoflessthank} For every $\size$-minimal $X$--$\eps$ separator $S$ the component of $H-S$ in which $\eps$ lives is $\eps$-linked and disjoint from $X$.
        \item \label{existenceofstrictlyincreasing} There is an $\eps$-linked $\lA$-region $C$ disjoint from $X$ which satisfies $X \cap N_H(C) \subseteq \Dom(\eps)$.
    \end{enumerate}
\end{lemma}

\begin{proof}
    \cref{existenceoflessthank}:
    Since $S$ separates $X$ and $\eps$, the component $C(S, \eps)$ is disjoint from $X$.
    It remains to show that~$C$ is $\eps$-linked. 
    Because $\eps$ has countable combined degree, \cite{enddefiningsequences}*{Lemma~5.1} yields an \defn{$\eps$-defining} sequence $(S_n)_{n \in \N}$, that is a sequence of finite sets $S_n \subseteq V(G)$ such that $C(S_n, \eps) \supseteq C(S_{n+1}, \eps)$, $S_n \cap S_{n+1} \subseteq \Dom(\eps)$ and $\bigcap_{n \in \N} C(S_n, \eps) = \emptyset$.
    It suffices to find an $\eps$-defining sequence $(S'_n)_{n \in \N}$ with $S'_0 = S$ such that there are~$|S'_n|$ pairwise disjoint $S'_n$--$S'_{n+1}$ paths for every $n \in \N$.
    
    So set $S'_0:= S$.
    Assume that we have constructed $S'_0, \dots, S'_n$ for some $n \in \N$.
    Since the $C(S_n, \eps)$ are $\subseteq$-decreasing and their intersection is empty, we may choose $N \in \N$ sufficiently large such that $S'_n \cap C(S_N, \eps) = \emptyset$.
    Then we choose $S'_{n+1}$ as a $\size$-minimal $S_{N+1}$--$\eps$ separator;
    in particular, $C(S'_n, \eps) \supseteq C(S_N, \eps) \supseteq C(S_{N+1}, \eps) \supseteq C(S'_{n+1}, \eps)$, and thus $S'_n \cap S'_{n+1} \subseteq S_N \cap S_{N+1} \subseteq \Dom(\eps)$.
    Note that every $S'_n$--$S'_{n+1}$ separator would have been a suitable choice for $S'_n$ and thus has size at least $|S'_n|$.
    By Menger's Theorem (see for example \cite{bibel}*{Proposition~8.4.1}), the desired paths exist.
    
    \cref{existenceofstrictlyincreasing}:
    \cite{enddefiningsequences}*{Lemma~5.1} yields a region $C'$ disjoint from $X$ in which~$\eps$ lives and whose finite neighbourhood~$N(C')$ shares with $X$ only vertices in $\Dom(\eps)$.
    Now applying \cref{existenceoflessthank} to a $\size$-minimal $N(C')$--$\eps$ separator yields the desired $\eps$-linked region $C:=C(S,\eps) \subseteq C'$.
\end{proof}

In order to get a candidate for $C_i$ we will use the region from \cref{lem:existenceofregion} to obtain one which is additionally nested with~$\cD$ and with the $C_j$ for $j < i$.
For this, we first need one auxiliary lemma.

A region~$C$ of a graph $G$ is \defn{well linked} if, for every two disjoint finite $X,Y \subseteq N_G(C)$, there is a family of $\min\{|X|,|Y|\}$ pairwise disjoint $X$--$Y$ paths in $G$ through $C$ (i.e.~all internal vertices contained in $C$).

\begin{lemma}\label{lem:endlinkedimplieswelllinked}
    Let $H$ be a graph and let $C$ be an end-linked region of $H$. 
    Then $C$ is well linked. 
\end{lemma}

\begin{proof}
    Let~$X,Y \subseteq N(C)$ be disjoint and finite, and suppose for a contradiction that there is no family of $\min\{|X|,|Y|\}$ pairwise disjoint $X$--$Y$ paths through $C$.
    By Menger's theorem (see for example \cite{bibel}*{Proposition~8.4.1}), there is an $X$--$Y$ separator~$S$ of size less than $\min\{|X|,|Y|\}$ in $H[V(C) \cup X \cup Y]$.
    Since~$C$ is end-linked, there is some end $\eps$ of $H$ which lives in~$C$ and a family $\{R_v \mid v \in N(C)\}$ of pairwise disjoint $N(C)$--$\eps$ paths and rays with $v \in R_v$.
    Since $|S| < \min\{|X|,|Y|\}$ and the~$R_v$ are pairwise disjoint, there is $x \in X$ and $y \in Y$ such that~$R_x$ and~$R_y$ both avoid~$S$.
    Since $R_x$ and $R_y$ are $N(C)$--$\eps$ paths or rays, there is some $R_x$--$R_y$ path~$P$ in~$C$ which avoids the finite set~$S$.
    Hence, $R_x + P + R_y$ is a connected subgraph of $H[V(C) \cup X \cup Y]$ which meets $X$ and $Y$ but avoids~$S$.
    This contradicts the fact that $S$ is an $X$--$Y$ separator in $H[V(C) \cup X \cup Y]$.
\end{proof}

Let $\eps$ be an end of $H$ that does not live in any $C_j$ for $j < i$ or in any $D \in \cD$ and which can be separated from $X$ by fewer than $|X|$ vertices. Then our next lemma yields an $\eps$-linked region which is a candidate for the region $C_i$ in \cref{RegionsCaseA}.

\begin{lemma}
\label{lem:theminiuncrossinglemma}
    Let $H$ be a graph and let $\cE$ be a set of pairwise non-touching well-linked $\lA$-regions.

    If $C$ is a $k$-region with $k \in \N$ which is $\eps$-linked for some end $\eps$ that only lives in $D \in \cE$ which are contained in~$C$, then there exists an $\eps$-linked $\lek$-region $C'$ that is nested with all $D \in \cE$.
    
    Moreover, the set $\cE'$ of regions $D \in \cE$ which are not nested with~$C$ is finite, the region $C'$ is not only nested with all $D \in \cE$ but contains $D$ or does not touch $D$, and $C'$ can be chosen such that $C' \subseteq C \cup \bigcup_{D \in \cE'} D$ and $N_H(C') \subseteq N_H(C) \cup \bigcup_{D \in \cE'} N_H(D)$. 
\end{lemma}

\begin{proof}
    We first show that $\cE'$ is finite.
    For this, it suffices to show that each of the pairwise disjoint regions in $\cE'$ meets the finite set $N(C)$.
    So let $D \in \cE'$ be given. Because $D$ and $C$ are not nested, they touch, that is, the closure $\bar{C}$ meets $D$. Moreover, $D$ is not contained in $C$, so $D - C$ is non-empty. 
    Thus, since $D$ is connected, there is a $(D - C)$--$\Bar{C}$ path in $D$.
    Its endvertex in $\Bar{C}$ is in~$N(C)$, as $D - C$ and $C$ are disjoint.
    Thus, $D \in \cD$ meets $N(C)$.

    Let $\cE'_{<}$ consist of all $D \in \cD'$ with $|V(D) \cap N(C)| < |V(C) \cap N(D)|$.
    Now $C^* := G[V(C) \cup \bigcup_{D \in \cE'_{<}} V(D)]$ has neighbourhood $(N(C) \setminus \bigcup_{D \in \cE'_{<}} V(D)) \cup \bigcup_{D \in \cE'_{<}} (N(D) \setminus V(\Bar{C}))$, since the $D \in \cE$ are pairwise non-touching. 
    We claim that $C^*$ is a $\eps$-linked $\lek$-region.

    The subgraph $C^*$ is connected and thus a region since every $D \in \cE' \supseteq \cE'_{<}$ touches $C$.
    For $D \in \cE'_{<}$, the set~$V(D) \cap N(C)$, which separates $N(D) \setminus V(\Bar{C})$ and~$V(C) \cap N(D)$ in~$\Bar{D}$, must have at least size~$|N(D) \setminus V(\Bar{C})|$, since $D$ is well linked by assumption and $|V(D) \cap N(C)| < |V(C) \cap N(D)|$ by definition of $\cE'_{<}$. 
    Hence, $|N(D) \setminus V(\Bar{C})| \leq |V(D) \cap N(C)| < |V(C) \cap N(D)|$ for  every $D \in \cE'_{<}$ yields that the size of the neigbourhood of $C^*$ is at most $k$ and a family $\cP_D$ of $|N(D) \setminus V(\Bar{C})|$ pairwise disjoint $(N(D) \setminus V(\Bar{C}))$--$(V(C) \cap N(D))$ paths through $D$.
    We claim that $C^*$ is $\eps$-linked.
    Indeed, since the $D \in \cE$ are pairwise non-touching, all these paths in $\cP_D$ with $D \in \cE$ and the trivial paths in $N(C) \setminus (\bigcup_{D \in \cE'_{<}} V(D))$ are pairwise disjoint. 
    Hence, we obtain the desired $N(C^*)$--$\eps$ paths and rays by extending the collection of all these paths and the trivial paths in $N(C) \setminus (\bigcup_{D \in \cE'_{<}} V(D))$ via the original $N(C)$--$\eps$ paths or rays witnessing the $\eps$-linkedness of~$C$.

    By definition of $C^*$ and since the $D  \in \cE$ are pairwise non-touching, we have
    $C^* \cap \Bar{D} = C \cap \Bar{D}$, $V(C^*) \cap N(D) = V(C) \cap N(D)$ and $N(C^*) \cap V(D) = N(C) \cap V(D)$ for every $D \in \cE \setminus \cE'_{<}$.
    Hence, the regions $D \in \cE$ which are not nested with $C^*$ are precisely those in $\cE^* := \cE' \setminus \cE'_{<}$ and every $D \in \cE^*$ satisfies 
    $|V(D) \cap N(C^*)| = |V(D) \cap N(C)| \geq |V(C) \cap N(D)| = |V(C^*) \cap N(D)|$.
    Hence, the size of 
    $$B := \Big(N(C^*) \setminus (\bigcup_{D \in \cE^*} V(D))\Big) \cup \bigcup_{D \in \cE^*} \big(V(C^*) \cap N(D)\big),$$
    which includes the neighbourhood of $C^* - (\bigcup_{D \in \cE^*}) \Bar{D}$, is at most $|N(C^*)| \leq k$.
    Since $\eps$ lives only in $D \in \cE$ which are contained in $C \subseteq C^*$ and thus such $D \notin \cE'\supseteq \cE^*$, the end $\eps$ lives in~$C^* - (\bigcup_{D \in \cE^*} \Bar{D})$.
    So the $\eps$-linkedness of $C^*$ and $|B| \leq |N(C^*)|$ yields that $B$ is linked to~$\eps$.
    All in all, the component $C' := C_H(B, \eps)$ is the desired $\eps$-linked $\lek$-region with $C' \subseteq (C \cup \bigcup_{D \in \cE'_{<}} D) - (\bigcup_{D \in \cE^*} \Bar{D}) \subseteq C \cup \bigcup_{D \in \cE'} D$ and $N(C') = B \subseteq N(C) \cup \bigcup_{D \in \cE'} N(D)$.
    Note that by definition $C'$ does not touch $D \in \cE^*$.
    Since every other $D \in \cE \setminus \cE^*$ either does not touch $C^*$ or is contained in $C^*$ and also does not touch any $D' \in \cE^*$, the region $D$ does not touch $C'$ or is contained in~$C'$.
\end{proof}

The next lemma yields for every end $\eps$ that does not already live in some $C_j$ for $j < i$ or in some $D \in \cD$ an $\eps$-linked region which is a candidate for the region $C_i$ in \cref{RegionsCaseB}.

\begin{lemma}
\label{lem:theultimateuncrossing}
    Let $H$ be a graph, let $Z \subseteq V(H)$ be finite, and let $\cE$ be a set of pairwise non-touching well-linked $\lA$-regions of $H$.
    
    If $\eps$ is an end of $H$ of countable combined degree that lives in no $D \in \cE$, then there exists an $\eps$-linked $\lA$-region~$C$ that is disjoint from $Z$, satisfies $Z \cap N_H(C) \subseteq \Dom(\eps)$ and is nested with all $D \in \cE$. Moreover, every region in $\cE$ that touches $C$ is contained in $C$. 
\end{lemma}

\begin{proof}
    By \cref{lem:existenceofregion}~\cref{existenceofstrictlyincreasing}, there exists an $\eps$-linked $\lA$-region $C_0$ of $H$ that is disjoint from~$Z$ and satisfies $Z \cap N(C_0) \subseteq \Dom(\eps)$.
    Let $\cE'$ be the set of all regions in $\cE$ that are not nested with~$C_0$. The `moreover'-part of \cref{lem:theminiuncrossinglemma} ensures that $\cE'$ is finite.
    Thus, \cref{lem:existenceofregion}~\cref{existenceofstrictlyincreasing} applied to the finite set $Z' := Z \cup N(C_0) \cup \bigcup_{D \in \cE'} N(D)$ and the end $\eps$ yields an $\eps$-linked region $C_1$ that is disjoint from $Z'$ and satisfies $Z \cap N(C_1) \subseteq \Dom(\eps)$. 
    Since $\eps$ lives in no $D \in \cE \supseteq \cE'$, the fact that~$C_1$ is disjoint from $Z'$ yields that $C_1 \subseteq C_0 - (\bigcup_{D \in \cE'} \Bar{D}$).
    
    Note that, since $C_1$ is nested with all regions in~$D$ that were not nested with $C_0$, every $D \in \cE$ which touches $C_1$ is contained in $C_0$. 
    In particular, $Z \cap N(D) \subseteq Z \cap N(C_0) \subseteq \Dom(\eps)$ for all such $D \in \cD$ as $C_0$ avoids $Z$.
    Now \cref{lem:theminiuncrossinglemma} yields a $\eps$-linked $\lA$-region $C_2$ that is disjoint from $Z$, is nested with $\cE$ and satisfies $Z \cap N(C_2) \subseteq Z \cap N(C_0) \subseteq \Dom(\eps)$. 
    Thus $C := C_2$ is the desired region.
    
    For the `moreover'-part we note that $\eps$ lives in $C$ but in no $D \in \cE$, and thus every $D \in \cE$ that touches $C$ and that is nested with $C$ is contained in $C$. 
    Hence, the `moreover'-part follows since~$C$ is nested with all $D \in \cE$.
\end{proof}

\subsection{Analysis of \texorpdfstring{\cref{recursion:buildingpackage}}{Algorithm 7.2}} \label{subsec:AnalysisOfRecursion}

In this section we analyse \cref{recursion:buildingpackage} and provide in \cref{theorem:propertiesofrecursion:buildingpackage} the properties of the regions $C_i$ and the set $Y$ obtained from \cref{recursion:buildingpackage} which we need for the proof of \cref{thm:RaylessThmTechnicalRegions}.
First let us take note of some basic properties of \cref{recursion:buildingpackage} that follow easily from its definition:

\begin{observation} \label{observation:basicpropertiesofrecursion}
    In the setting of \cref{recursion:buildingpackage}, we have that
    \begin{enumerate}[label=\rm{(O\arabic*)}]
        \item \label{property:BoundaryOfRegionsInCaseAandB}
        every $C_i$ with $|N_H(C_i)| < k$ was chosen in \cref{RegionsCaseA} and every $C_i$ with $|N_H(C_i)| \geq k$ was chosen in \cref{RegionsCaseB},
        \item \label{property:CaseBifell_iatleastk} if $\ell_i \geq k$, then $X \cap N_H(C_i) \subseteq \Dom(C_i)$,
        \item \label{property:ell_iareincreasing} the $\ell_i$ are increasing,
        \item \label{property:C_iareneverdecreasing} a $C_i$ is never contained in $C_j$ with $j < i$,
        \item \label{property:iftheC_iareincreasingthenell_iarestrictlyincreasing} if $C_j \subseteq C_i$ with $j < i$, then $\ell_j < \ell_i$, 
        \item \label{property:terminates} \cref{recursion:buildingpackage} terminates at some ordinal $< |H|^+$.
    \end{enumerate}
\end{observation}

\begin{proof}
    \cref{property:BoundaryOfRegionsInCaseAandB}:
    It is immediate from \cref{recursion:buildingpackage} that every $\gek$-region was chosen in \cref{RegionsCaseB}.
    Every $\lek$-region as in \cref{RegionsCaseB} is already a region as in \cref{RegionsCaseA}; thus, every $\lk$-region $C_i$ was chosen in \cref{RegionsCaseA}.
    
    \cref{property:CaseBifell_iatleastk}:
    This is immediate from \cref{recursion:buildingpackage} and \cref{property:BoundaryOfRegionsInCaseAandB}.

    \cref{property:ell_iareincreasing}:
    By \cref{property:BoundaryOfRegionsInCaseAandB}, every $C_i$ with $|N(C_i)| < k$ was chosen in \cref{RegionsCaseA}. 
    Moreover, every $C_i$ that was chosen in \cref{RegionsCaseA} was already a candidate for $C_j$ in earlier steps $j < i$ where $C_j$ was chosen due to \cref{RegionsCaseA}, but was not the chosen nicest such region (and similar for \cref{RegionsCaseB}). 
    Thus \cref{nicest:ell_iminimum} ensures that the $\ell_i$ are increasing.
    
    \cref{property:C_iareneverdecreasing}: 
    A $C_i$ will never be a subgraph of any $C_j$ with $j < i$ since the end to which $N(C_i)$ is linked lives in $C_i$ but not in $C_j$ by the choice of $C_i$ in \cref{recursion:buildingpackage}.

    \cref{property:iftheC_iareincreasingthenell_iarestrictlyincreasing}:
    By \cref{property:ell_iareincreasing} we have $\ell_j \leq \ell_i$.
    So suppose for a contradiction that $\ell_{j} = \ell_{i}$.
    By \cref{recursion:buildingpackage},~$C_i$ was already a candidate for $C_{j}$.
    Since $\ell_i = \ell_j$, the region $C_{i}$ satisfied \cref{nicest:ell_iminimum} in step $j$.
    But then $C_j = C_{i}$ by \cref{nicest:inclusion-wisemaximal} of $C_{j}$, which contradicts that the end to which $N(C_i)$ is linked does not live in~$C_j$ by \cref{recursion:buildingpackage}.

    \cref{property:terminates}: 
    It suffices to show that $C_i$ contains some vertex which is in no other $C_j$ for $j < i$.
    If all~$C_j$ for $j < i$ are disjoint from $C_i$, then we are done.
    So we may assume that there is some $C_{j} \subseteq C_i$ with ${j} < i$.
    Consider such a $\subseteq$-maximal $C_{j^*}$; 
    it exists, since chains of such regions $C_j \subseteq C_i$ are finite by \cref{property:iftheC_iareincreasingthenell_iarestrictlyincreasing}.
    If $N(C_{j^*}) \subseteq N(C_i)$, then the connectedness of $C_i$ and $C_{j^*}$ yields $C_i = C_{j^*}$ which contradicts their distinctness.
    Thus, there exists a vertex $v \in V(C_i) \cap N(C_{j^*})$.
    Since the $C_j$ for $j < i$ are nested and by the $\subseteq$-maximality of $C_{j^*}$, the region $C_{j^*}$ does not touch $C_j$ or contains~$C_j$ for all $j \neq j^*$ with $C_j \subseteq C_i$ and $j < i$; in particular, in both cases $V(C_j) \cap N(C_{j^*}) = \emptyset$.
    Thus, $v \in C_i \setminus \bigcup_{j < i} C_j$ as desired.
\end{proof}

Using the following properties of \cref{recursion:buildingpackage}, we show in \cref{subsec:ProofOfRaylessThm} that the iterative application of \cref{recursion:buildingpackage} yields the desired \td\ for \cref{thm:RaylessThmTechnicalRegions}:

\begin{theorem}\label{theorem:propertiesofrecursion:buildingpackage}
    In the setting of \cref{recursion:buildingpackage} the following statements hold:
    \begin{enumerate}[label=\rm{(A\arabic*)}]
        \item\label{property:progress} Either $x \in Y$ or $|X| > |N_H(C)|$ for the component $C$ of $H - Y$ containing $x$. 
        \item \label{property:DareC_i} Every $D \in \cD$ is contained in a component of $H-Y$.
       \end{enumerate}
      Moreover, if every end of $H$ has countable combined degree, then also  the following hold:
      \begin{enumerate}[label=\rm{(A\arabic*)},resume]
        \item \label{property:rayless} Every end of $H$ lives in some $C_i$.
        \item \label{property:componentsareC_i}Every component of $H-Y$ is a $C_i$.
        \item \label{property:increasinglydisjoint} If $|N_H(C)| \geq k$ for a component $C$ of $H -Y$, then $X \cap N_H(C) \subseteq \Dom(C)$.
        \item \label{property:linkediniterativeapplication}
        Set $H^0 := H$ and $Y^0 := Y$. Let $C'_1 \supseteq C'_2 \supseteq \cdots \supseteq C'_n$ be a sequence of regions of $H$ such that, for every $m \in \{0, \dots, n-1\}$, the region $C'_{m+1}$ is a component of $H^m - Y^m$ where, for $m \in \{1, \dots, n-1\}$,~$Y^m$ is given by \cref{recursion:buildingpackage} applied to $H^m := \bar{C}'_m$, the finite set $X^m := N_H(C'_m)$, the set $\cD^m := \{D \in \cD \mid D \subseteq C'_m\}$ and an arbitrary vertex $x_m \in C'_m - (\bigcup \cD)$. 
        Then there are $\min\{| N_H(C'_m)| \mid m \in \{1, \dots,n\} \}$ pairwise disjoint $X$--$N_H(C'_n)$ paths in~$H$.
        \item \label{property:DistinctBags} If $H-X$ is connected and $N_H(H-X) = X$, then $N_H(C) \not\subseteq X$ for all components $C$ of $H-Y$; in particular, $X \subsetneq Y$.
    \end{enumerate}
\end{theorem}

\begin{proof}
    \cref{property:progress}: 
    Assume that $x \notin Y$.
    Let $\cC_x$ be the collection of all those $C_i$ that contain $x$. Then~$\cC_x$ is non-empty since by the definition of $Y$ at least one $C_i$ contains $x$. We claim that $C := \bigcup \cC_x$ is a component of $H-Y$ with $|N(C)| < k = |X|$. 
    Indeed, it is immediate from \cref{recursion:buildingpackage} that every $C_i \in \cC_x$ is a region as in \ref{RegionsCaseA}, since the regions that where chosen in \ref{RegionsCaseB} avoid $x$; in particular, every $C_i \in \cC_x$ is a $\lk$-region.
    Since all the $C_i \in \cC_x$ meet in $x$ and are nested, $\cC_x$ is a chain (with respect to inclusion).
    Thus, \cref{lem:limitofregionissuchregion} ensures that $C$ is a region of $H$ with $|N(C)| < k = |X|$.
    
    To finish the proof of \cref{property:progress}, it thus remains to show that $C$ is a component of $H-Y$. For this, it suffices to prove that $N(C) \subseteq Y$ since $C$ is a region and hence connected.
    As $C = \bigcup \cC_x$, we have $N(C) \subseteq \bigcup \{N(C_i) \mid C_i \in \cC_x\}$.
    Since all the $C_i$ are nested, every $C_j$ which meets the neighbourhood of some $C_i \in \cC_x$ already contains $C_i$; in particular, such $C_j$ are also in $\cC_x$ as $x \in C_i \subseteq C_j$.
    Thus $N(C)$ meets no~$C_i$, and so the definition of $Y$ yields that $N(C) \subseteq Y$.

    \cref{property:DareC_i}:
    By the assumptions on $\cD$ and the choice of the $C_i$, in every step $i$ of \cref{recursion:buildingpackage} every $D \in \cD$ is a candidate for the region $C_i$ in \cref{RegionsCaseB} as long as $D$ is $\eps$-linked for some end $\eps$ of $H$ which lives in no $C_j$ for $j < i$. Again by the assumptions on $\cD$, the region $D \in \cD$ is $\eps$-linked for some end $\eps$ of $H$.
    So, since \cref{recursion:buildingpackage} terminates because of \cref{RegionsCaseC} by \cref{observation:basicpropertiesofrecursion} \cref{property:terminates}, the end $\eps$ lives is some $C_i$. As $\eps$ lives in both~$C_i$ and~$D$, they touch. Thus, $C_i \subseteq D$ or $D \subseteq C_i$ since they are nested by \cref{recursion:buildingpackage}.
    Hence, we finish the proof of \cref{property:DareC_i} by showing that if $C_i \subseteq D$, then $C_i = D$.
    For this, it suffices to show that in step~$i$ of \cref{recursion:buildingpackage}, the region $D$ was a candidate for $C_i$ and satisfies \cref{nicest:ell_iminimum}, because then $C_i \subseteq D$ together with \cref{nicest:inclusion-wisemaximal} yields $C_i = D$. 

    It is immediate from \cref{recursion:buildingpackage} and the assumptions on $\cD$ that in step $i$, the region $D \in \cD$ is nested with all $C_j$ for $j <i$, disjoint from $X$ and that $X \cap N(D) \subseteq \Dom(D)$.
    Since $D$ is $\eps$-linked and $\eps$ lives in $C_i \subseteq D$, there are $|N(D)|$ disjoint $N(D)$--$N(C_i)$ paths in $\bar{D} - C_i$. 
    These paths can be extended to disjoint $N(D)$--$\eps'$ paths and rays where $\eps'$ is the end so that $C_i$ is $\eps'$-linked. 
    In particular, $D$ is end-linked to the end $\eps'$ which lives in no $C_j$ for $j < i$, and $\ell_i = |N(C_i)| \geq |N(D)|$. Thus, $D$ was a candidate for $C_i$ and satisfies \cref{nicest:ell_iminimum}, as desired.

    \cref{property:rayless}:
    Suppose towards a contradiction that there is an end $\eps$ of $H$ that lives in no~$C_i$.
    For every $\ell \in \N$, let $\cC_{\max}^{<\ell}$ be the set of all $C_i$ with $\ell_i < \ell$ that are $\subseteq$-maximal among all $\lell$-regions~$C_j$. 
    By \cref{observation:basicpropertiesofrecursion} \cref{property:iftheC_iareincreasingthenell_iarestrictlyincreasing} every $C_i$ with $\ell_i < \ell$ is included in some $C_j \in \cC^{<\ell}_{\max}$.  
    Moreover, since all $C_i$ are nested, the regions in $\cC^{< \ell}_{\max}$ are pairwise non-touching. 
    Hence, by \cref{lem:theultimateuncrossing} applied to $Z := X \cup \{x\}$ and the set $\cE$ of all $\subseteq$-maximal elements in $\cD \cup \cC^{<k}_{\max}$, there exists a $\eps$-linked $\lA$-region $C$ that is disjoint from $X \cup \{x\}$, satisfies $(X \cup \{x\}) \cap N(C) \subseteq \Dom(\eps)$ and is nested with all $\subseteq$-maximal regions in $\cD \cup \cC^{<k}_{\max}$. Set $\ell := |N(C)|$. 
    Then \cref{lem:theminiuncrossinglemma} applied to $C$ and the set $\cE$ of all $\subseteq$-maximal regions in $\cD \cup \cC^{< \ell + 1}_{\max}$ yields a $\eps$-linked $\leell$-region~$C'$ which is nested with all regions in $\cD \cup \cC^{< \ell + 1}_{\max}$ and such that $(X \cup \{x\}) \cap N(C') \subseteq \Dom(C')$, where we have used that $(X \cup \{x\}) \cap N(D) \subseteq \Dom(D)$ for all $D \in \cC^{<\ell+1}_{\max} \setminus \cC^{<k}_{\max}$ since any such region $D$ was chosen in \cref{RegionsCaseB}. 
    Moreover, by \cref{lem:theminiuncrossinglemma}, $C'$ is not strictly contained in any region in $\cD \cup \cC^{< \ell + 1}_{\max}$. In particular, $C'$ is nested with all $D \in \cD$ and all $C_i$ with $\ell_i \leq \ell$.
    But this contradicts \cref{recursion:buildingpackage}: If there is no $C_i$ with $\ell_i > \ell$, then $C'$ contradicts that \cref{recursion:buildingpackage} terminated because of \cref{RegionsCaseC} by \cref{observation:basicpropertiesofrecursion} \cref{property:terminates}. Otherwise, if there exists some $C_i$ with $\ell_i > \ell$, then~$C'$ witnesses that the first such $C_i$ did not satisfy \cref{nicest:ell_iminimum}.

    \cref{property:componentsareC_i}:
    Let $C$ be a component of $H-Y$. 
    By Zorn's Lemma there exists a $\subseteq$-maximal index set $I$ such that $C_j \subseteq C_i \subseteq C$ for all $j \leq i \in I$. Then $\bigcup_{i \in I} C_i = C$. Indeed, if $C' := \bigcup_{i \in I} C_i \subsetneq C$, then $N(C') \cap C \neq \emptyset$ since $C$ is connected. So by the definition of $Y$, there is some $C_j$ such that $N(C') \cap C_j \neq \emptyset$. In particular, $C_j \subseteq C$ since $C_j$ is connected. But since $N(C') \subseteq \bigcup_{i \in I} N(C_i)$, the $(C_i \mid i \in I)$ are $\subseteq$-increasing and all the $C_i$ are nested, we have $C_i \subseteq C_j$ for all $i \in I$, which contradicts that $I$ is $\subseteq$-maximal.
    By \cref{observation:basicpropertiesofrecursion} \cref{property:ell_iareincreasing}, the sequence $(\ell_i)_{i \in I}$ is strictly increasing. In particular, $I$ is either finite or of the same order type as~$\N$. In the former case we are done as $C = C_i$ for $i = \max(I)$. So assuming the latter, we now aim towards a contradiction. Enumerate $I = \{i_n : n \in \N\}$ so that $i_n < i_m$ for all $n < m \in \N$.
    
    Consider the auxiliary graph $A$ that arises from $C$ by contracting $C_{i_1}$ and, for every $n \in \N$, all components of $C_{i_{n+1}} - C_{i_n}$.
    The graph $A$ is obviously infinite and connected.
    Since the $N(C_{i_n})$ are finite and the $C_{i_n}$ are connected, there are at most finitely many components of $C_{i_{n+1}} - C_{i_n}$ for every $n \in \N$.
    Together with the fact that the $N(C_{i_n})$ are finite this yields that $A$ is also locally finite.
    Hence, the minor $A$ of $C$ contains a ray by \cite{bibel}*{Proposition~8.2.1}.
    We lift this ray to a ray~$R$ in $C$ by choosing suitable paths in each branch set connecting the endvertices of the incident edges.
    Then by construction, the end of $H$ that contains $R$ lives in no $C_i$, contradicting~\cref{property:rayless}. 

    \cref{property:increasinglydisjoint}: This follows directly from \cref{property:componentsareC_i} and \cref{property:CaseBifell_iatleastk}. 

    \cref{property:linkediniterativeapplication}:
    Set $k' := \min\{ | N(C_m') | \mid m \in \{1, \dots,n \} \}$. 
    By Menger's Theorem (see for example \cite{bibel}*{Proposition~8.4.1}), we either find the desired $k'$ pairwise disjoint $X$--$N(C_n')$ paths or there is an $X$--$N(C_n')$ separator $S$ with $|S| < k'$.
    We may assume the latter and let $S$ be a $\size$-minimal such separator. Further, let $\eps$ be an end such that $C'_n$ is $\eps$-linked, which exists by \cref{recursion:buildingpackage} and \cref{property:componentsareC_i}.
    Since~$C_n'$ is $\eps$-linked, the separator $S$ is also a $\size$-minimal $X$--$\eps$ separator.
    Then the component $\tilde{C}$ of~$H-S$ in which~$\eps$ lives is $\eps$-linked and disjoint from $X$ by \cref{lem:existenceofregion}~\cref{existenceoflessthank}. Moreover, we have $C'_n \subseteq \tilde{C}$ since $S$ avoids $C'_n$ by its choice as a minimal $X$--$N(C'_n)$ separator.
    
    Now let $C$ be a $\eps$-linked $(\leq\!|S|)$-region that is disjoint from $X$, contains $C'_n$ and is contained in as many of the $C'_i$ as possible.
    Note that $\tilde{C}$ is a candidate for $C$.
    Let $N \geq 0$ be the largest index~$i$ such that $C \subseteq C'_N$, and let $C^N_i$ be the regions obtained from the application of \cref{recursion:buildingpackage} to~$H^N$ and $N(C'_N)$ (or $H$ and $X$ if $N = 0$). 
    Note that $N < n$ since otherwise we have $C = C'_n$ and hence $|N(C'_n)| = |N(C)| < k'$, a contradiction.
    Let $\cC^{\leq |S|}_{\max}$ be the $\subseteq$-maximal elements of the~$C^N_i$ with $\ell^N_i \leq |S|$. 
    Note that by \cref{observation:basicpropertiesofrecursion}~\cref{property:iftheC_iareincreasingthenell_iarestrictlyincreasing} every~$C^N_i$ with $\ell^N_i < |S|$ is contained in some $C^N_j \in \cC^{\leq |S|}_{\max}$.
    Since all regions in $\cC^{\leq |S|}_{\max} \cup \cD$ are end-linked, they are well linked by \cref{lem:endlinkedimplieswelllinked}.
    Moreover, if $\eps$ lives in $D \in \cD$, then $D \subseteq C'_n \subseteq C$ by \cref{property:DareC_i}.
    To be able to apply \cref{lem:theminiuncrossinglemma} to $C$ and the set $\cE$ of $\subseteq$-maximal regions in $\cC^{\leq |S|}_{\max} \cup \cD$ it suffices that $\eps$ lives in no $C_j^N$ with $\ell_j^N \leq |S|$.

    So suppose for a contradiction that $\eps$ lives in some $C^N_j$ with $\ell^N_j \leq |S|$. We claim that then $C'':= C^N_j \cup C'_n$ is an $(\leq |S|)$-region. Indeed, $C''$ is a region because $C^N_j$ and $C'_n$ are both connected and intersect as $\eps$ lives in both of them. Moreover, $N(C'') = (N(C^N_j) \cup N(C'_n)) \setminus V(C'') = (N( C^N_j) \setminus V(C'_n)) \dot\cup (N(C'_n) \setminus V(\bar{C}^N_j))$ has size at most $|N(C^N_j)| = \ell^N_j \leq |S|$:
    if not, one easily checks by doubling counting that $Z := ((N(C^N_j) \cup N(C'_n)) \cap V(C'')) \cup (N(C^N_j) \cap N(C'_n)) = (N(C'_n) \cap V(C^N_j)) \dot\cup (N(C^N_j) \cap V(C'_n)) \dot\cup (N(C^N_j) \cap N(C'_n))$ has size less than $|N(C'_n)|$. 
    But this contradicts that $C'_n$ is $\eps$-linked as~$\eps$ lives in a component of $C_j^N \cap C'_n$, whose neighbourhood then is contained in $Z$, and so all $N(C'_n)$--$\eps$ rays and paths have to meet $Z$. Hence, $C''$ is a $(\leq |S|)$-region with $C'_n \subseteq C''$. 
    Moreover, $C^N_j$ and $C'_{N+1}$ both contain $C'_n$ by assumption and thus touch. So since $C^N_j$ and $C'_{N+1}$ are nested by \cref{recursion:buildingpackage} (because $C'_{N+1}$ is some $C^N_i$ by \cref{property:componentsareC_i}), \cref{observation:basicpropertiesofrecursion}~\cref{property:C_iareneverdecreasing} implies that $C^N_j \subseteq C'_{N+1}$, and therefore $C'' \subseteq C'_{N+1}$. But this contradicts our choice of~$C$. 
    Thus, $\eps$ lives in no region in $\cC^{\leq |S|}_{\max} \cup \cD$.
    
    So we may apply \cref{lem:theminiuncrossinglemma} to $C$ and the set $\cE$ of $\subseteq$-maximal separations in $\cC^{\leq |S|}_{\max} \cup \cD$ to obtain an $\eps$-linked $(\leq\!|S|)$-region $C'$ which is disjoint from $X$ and which, for every region in $\cC^{\leq |S|}_{\max} \cup \cD$, either contains that region or does not touch it.
    Thus, $C'$ is nested with all $C^N_i$ with $\ell^N_i \leq |S|$ and with all $D \in \cD$.
    This concludes the proof since
    the existence of $C'$ then contradicts that $C^N_I$ satisfies \cref{nicest:ell_iminimum} where $I = \min\{i \mid \ell^n_i > |S| \}$. Here we used that $C'_{N+1}$ is some $C^N_i$ and that $|N(C'_{N+1})| \geq k' > |S|$, so the minimum is not taken over the empty set.

    \cref{property:DistinctBags}: Let $C$ be a component of $H-Y$. By \cref{property:componentsareC_i}, we have $C = C_i$ for some suitable $i$. Suppose first that $C_i$ that was chosen in \cref{RegionsCaseA}. Then $|N_H(C)| < |X|$; since $H-X$ is connected and $N_H(H-X) = X$, $N_H(C)$ contains a vertex in $V(H)\setminus X$, implying $N_H(C) \not\subseteq X$.
    Now suppose that $C_i$ that was chosen in \cref{RegionsCaseB}. Then $N_H(C)$ separates $x$ and $C$. Since $H-X$ is connected and $x \notin X$, we have $N_H(C) \not\subseteq X$. The in-particular part is now also clear.
\end{proof}

\subsection{Proof of \texorpdfstring{\cref{thm:RaylessThmTechnicalRegions}}{Theorem 7.1}} \label{subsec:ProofOfRaylessThm}

Using \cref{theorem:propertiesofrecursion:buildingpackage} we now show that the \td\ obtained from iteratively applying \cref{recursion:buildingpackage} is as desired for \cref{thm:RaylessThmTechnicalRegions}.

\begin{proof}[Proof of \cref{thm:RaylessThmTechnicalRegions}]
    First, we define the desired \td\ $(T, \cV)$ for the case $X = \interior(\sigma)$.
    The decomposition tree $T$ is a star whose edges are in a bijective correspondence to $\sigma$.
    We assign to the centre the bag $X$, and to each leaf $\ell$ the bag $A$ where $(A,B) \in \sigma$ corresponds to the edge incident with $\ell$.
    It is immediate to see that this \td\ is as desired.
    Let us assume from now on that $X \subsetneq \interior(\sigma)$.

    Suppose first  that $G$ is connected.
    By \cref{thm:FiniteTWyieldsNST}, we may fix a normal spanning tree~$T_{NST}$ of~$G$ whose root is in~$X$.
    We denote by $\cD$ the set of subgraphs $G[A \setminus B]$ of $G$ with $(A,B) \in \sigma$.
    The assumptions on $\sigma$ ensure that~$\cD$ satisfies the assumptions in \cref{recursion:buildingpackage}. 
    Since $G$ admits a normal spanning tree, all its ends have countable combined degree. 
    Hence, $G$ satisfies the assumptions of \cref{theorem:propertiesofrecursion:buildingpackage}.
    We now define the desired \td~recursively as follows.

    Let $(T^0, \cV^0)$ be the trivial \td\ of $G$ where $T^0$ is the tree on a single vertex~$r$ which is also its root and $V^0_r := V(G)$.
    Now let $n \geq 0$ and suppose that we have already constructed linked, $X$-linked, tight, componental, rooted \td s $(T^m, \cV^m)$ of finite adhesion such that for all $m \leq n$ 
    \begin{enumerate}
        \item \label{item:proofofraylees:limit}  $T^m \subseteq T^n$ and $V_t^m = V_t^n$ for all $t \in T^m$ that are not at height $m$, the decomposition tree $T^m$ has height $m$, all their leaves are on height $m$ except possibly those whose corresponding leaf separation is a $(B,A)$ with $(A,B) \in \sigma$,
        \item \label{item:proofofrayless:raylessendlinkedandIncDis} the torsos of $(T^m, \cV^m)$ at non-leaves are rayless, and
        $(T^m, \cV^m)$ satisfies \cref{itm:RaylessThmTechnicalRegions:EndLinked} and \cref{itm:RaylessThmTechnicalRegions:IncDis}.
    \end{enumerate}
    
    Note that $(T^0, \cV^0)$ satisfies all these properties immediately, where we remark that we treat the unique node of $T^0$ as a leaf.
    Let $L$ be the set of leaves of $T^n$ whose corresponding leaf separation is not some $(B,A)$ with $(A,B) \in \sigma$. 
    Note that, if $n = 0$, then $L = \{r\}$.
    For every leaf $\ell \in L$, we construct a \td\ $(T^\ell, \cV^\ell)$ of $G[V_\ell^{n}]$ as follows.
    Set $X' := X$ if $n = 0$ and $\ell$ is the unique node of $T^0$, and $X' := V_f^n$ otherwise where $f$ is the unique edge of $T^n$ that is incident with~$\ell$. Further, let $x$ be a $(\leq_{T_{NST}})$-minimal vertex in $V_\ell^n \setminus (X' \cup \bigcup \cD)$. Note that, if $n= 0$, $x$ exists (but might not be unique) as $X \subsetneq \interior(\sigma)$.
    If $n >0$, the vertex $x$ exists and is unique, since~$T_{NST}$ is normal and because $(T^n, \cV^n)$ is componental by construction and so $G[V_\ell^n] - (X' \cup \bigcup \cD)$ is connected by the assumption on $\sigma$.
    
    Now apply \cref{recursion:buildingpackage} to $H := G[V_\ell^n]$ with the finite set $X'$, the vertex $x$ and $\{ D \in \cD \mid D \subseteq H \}$ to obtain $Y \subseteq V(H)$.
    Then let $T^\ell$ be the star with centre $\ell$ and whose set of leaves is the set $\cC$ of all components of $H-Y$. Further, set $V^\ell_\ell := Y$ and $V^\ell_C := \Bar{C}$ for every~$C \in \cC$.
    Then \cref{theorem:propertiesofrecursion:buildingpackage}~\cref{property:componentsareC_i} yields that for every edge $e = \ell C$ of $T^\ell$, the graph $H \strictup e$ is a $C_i$ in \cref{recursion:buildingpackage} applied to $H$.
    Thus, $(T^\ell, \cV^\ell)$ is componental and end-linked; thus, $(T^\ell, \cV^\ell)$ satisfies \cref{itm:RaylessThmTechnicalRegions:EndLinked}.
    Moreover, the torso at $\ell$ is rayless: since $(T^\ell, \cV^\ell)$ is tight, we can obtain from any ray in the torso at $\ell$ a ray in $G$ that meets $V^\ell_\ell$ infinitely often by (the comment after) \cref{prop:OnePathOrRayExtended}. 
    But this contradicts \cref{theorem:propertiesofrecursion:buildingpackage}~\cref{property:rayless}. 
    Furthermore, it satisfies the following (which will ensure \cref{itm:RaylessThmTechnicalRegions:IncDis} of $(T^{n+1},\cV^{n+1})$) by \cref{theorem:propertiesofrecursion:buildingpackage}~\cref{property:increasinglydisjoint}: if $|V^\ell_e| \geq |V^n_f|$ for an edge $e \in T^\ell$, then $V^\ell_e \cap V^n_f \subseteq \Dom(H \up e)$.
    
    By construction, the $T^\ell$ are pairwise disjoint and only share their centre $\ell$ with $T^n$.
    We set $T^{n+1} := T^n \cup \bigcup_{\ell \in L} T^{n+1}_\ell$ with $\rt(T^{n+1}) := \rt(T^n) (= r)$, and bags $V^{n+1}_t := V^n_t$ for every node $t \in T^n - \ell$ and $V^{n+1}_t := V^\ell_t$ for every node $t \in T^{n+1}\setminus T^n$ where $\ell$ is the unique leaf of $T^n$ such that $t \in T^\ell$.
    In particular, the decomposition tree $T^{n+1}$ has height $n+1$, all its leaves are on height $n+1$, except possibly those whose corresponding leaf separation is $(B,A)$ for some $(A,B) \in \sigma$, and $(T^{n+1}, \cV^{n+1})$ satisfies all properties that we demanded from the $(T^m, \cV^m)$:
    All properties but the ($X$-)linkedness of $(T^{n+1}, \cV^{n+1})$ follows immediately from the construction and the discussion of the properties of the $(T^\ell, \cV^\ell)$.
    Its ($X$-)linkedness is ensured by \cref{theorem:propertiesofrecursion:buildingpackage}~\cref{property:linkediniterativeapplication}.
    
    Let $(T,\cV)$ be the limit of $(T^n, \cV^n)$ for $n \to \infty$, that is $T = \bigcup_{n \in \N} T^n$, $\rt(T) := r = \rt(T^n)$ for all $n \in \N$ and $V_t := V^N_t = V^n_t$ for all $n \geq N$ where $N$ is the minimal number such that $t \in T^N$ and either $t$ is not a leaf of $T^N$ or its leaf separation is a $(B,A)$ with $(A,B) \in \sigma$. 
    
    We first show that $(T, \cV)$ is a \td~of $G$. For this, by \cref{item:proofofraylees:limit}, the definition of $(T, \cV)$ and since all $(T^n, \cV^n)$ are \td s of $G$, it suffices to show that every vertex of $G$ is contained in some bag $V_t$. In other words, we need to show that each $v \in V(G)$ satisfies~\labelText{($\ast$)}{tag:ast}: 
    there are $n \in N$ and $t \in T^n$ such that $v \in V^n_t$ and $t$ is either a non-leaf of $T^n$ are its leaf separation is a $(B,A)$ with $(A,B) \in \sigma$.
    Let us first assume that $v \notin D$ for all $D \in \cD$.
    Then \nameref{tag:ast} is ensured by considering \cref{theorem:propertiesofrecursion:buildingpackage}~\cref{property:progress} along $\rt(T_{NST})T_{NST}v = v_0 v_1 \dots v_m$:
    Suppose for a contradiction that \nameref{tag:ast} does not hold for $v$. 
    Then let $i$ be the minimal index such that $v_i$ does not satisfy \nameref{tag:ast}; in particular, $v_i \in G \strictup e^n$ for a unique edge $e^n$ incident with a leaf of $T^n$ for every~$n \in \N$.
    We have $i \geq 1$, since $\rt(T_{NST}) \in X \subseteq V^1_r$ by assumption on $T_{NST}$ and construction of $V_r^1$.
    Thus,~$v_{i-1}$ satisfies \nameref{tag:ast};
    let $N$ be a sufficiently large integer given by \nameref{tag:ast} of~$v_{i-1}$.
    Then $v_i = \min_{\leq_S} (V(G \strictup e^n) \setminus (\bigcup V(\cD)))$ for every $n \geq N$.
    Hence, \cref{theorem:propertiesofrecursion:buildingpackage}~\cref{property:progress} ensures that the finite sets $V^n_{e^n}$ strictly decrease in their size for $n \geq N$, which is a contradiction.

    Second, assume that $v \in D$ for some $D \in \cD$. 
    By \cref{theorem:propertiesofrecursion:buildingpackage}~\cref{property:DareC_i}, we find a path $P$ ending in a leaf $\ell$ of $T$ or a ray $P$ in $T$ starting in the root $r$ of $T$ such that $D \subseteq G \strictup e$ for every $e \in P$. 
    If~$P$ is a path, then we find $v \in V_\ell$ as desired. 
    In particular, by construction of $(T, \cV)$ and since the regions in $\cD$ are nested, the leaf separation at $\ell$ is $(V(G) \setminus V(D) , V(D) \cup N(D))$.
    
    But $P$ cannot be a ray. Indeed, applying \nameref{tag:ast} to the finite set $N(D)$ yields $N \in \N$ such that all $u \in N(D)$ are contained in bags $V^N_t$ at nodes $t \in T^N$ which are either not leaves or their corresponding leaf separations are some $(B,A)$ with $(A,B) \in \sigma$. 
    Then $D \subseteq G \strictup e_N$ for the $N$-th edge of $P$ by the definition of $P$, and $N(D) \subseteq G \down e_N$ by the choice of $N$. 
    Since $(T^N, \cV^N)$ is componental, this implies that $V(D) \cup N(D) = V^N_{\bar{D}}$ is a bag at a leaf of $(T^N, \cV^N)$, and thus, by the construction of $(T, \cV)$ also appears as a bag at a leaf of $(T, \cV)$.
    In particular, the argument above implies that all $(B,A)$ with $(A,B) \in \sigma$ are leaf separations of $(T, \cV)$. Since $(T, \cV)$ has no other leaf separations by construction and \cref{item:proofofraylees:limit}, it follows that the leaf separations of $(T, \cV)$ are precisely $\{(B,A) \mid (A,B) \in \sigma\}$.

    Now it is immediate from the construction that $(T,\cV)$ is linked, $X$-linked, tight, componental, has finite adhesion and satisfies \cref{itm:RaylessThmTechnicalRegions:rayless}, \cref{itm:RaylessThmTechnicalRegions:EndLinked} and \cref{itm:RaylessThmTechnicalRegions:IncDis} because the $(T^n, \cV^n)$ are linked, $X$-linked, tight, componental, have finite adhesion and satisfy \cref{item:proofofrayless:raylessendlinkedandIncDis}.
    
    We are thus left to show the `moreover'-part. For this, recall that $(T, \cV)$ is tight and componental, so every $G[V^n_\ell]$ considered in the construction of $(T, \cV)$, except possibly $G[V^0_r]$, satisfies the premise of \cref{property:DistinctBags}. It follows for all edges $e = ts \in T$ with $t \leq_T s$ and $t \neq r$ that $V_t \supsetneq V_e \subsetneq V_s$. 
    In particular, if $G-X$ is connected and $N_G(G-X) = X$, then also $G[V^0_r]$ satisfies the premise of \cref{property:DistinctBags}, so $V_r \supsetneq V_e \subsetneq t$ for all edges $e = rt \in T$ and $X \subsetneq V_r$. 

    The proof of the case that $G$ is disconnected is analogous by choosing for each component $C$ of~$G$ a normal spanning tree $T^{C}_{NST}$ and considering the partial order $\leq_{NST}$ as the disjoint union of their tree orders.\footnote{The only adaption one has to do is to prove \nameref{tag:ast} for the roots of every $T^C_{NST}$.}
\end{proof}

\section{Lean tree-decompositions} \label{sec:LeanTreeDecs}

In this section we prove \cref{main:LeanTD}, the strengthening of~\cref{main:LinkedTightCompTreeDecompnew} for graphs without half-grid minor in which we replace the linkedness of the \td\ with leanness.
We restate it here in its more detailed version:
 
\begin{customthm}{Theorem~3'}[Detailed version of \cref{main:LeanTD}] \label{thm:LeanTDTechnical}
Every graph~$G$ without half-grid minor admits a lean, cofinally componental, rooted \td\ into finite parts which displays the infinities.

Moreover, if the tree-width of $G$ is finitely bounded, then the \td\ can be chosen to have width $\tw(G)$.
\end{customthm}

The proof of~\cref{thm:LeanTDTechnical} is structured as follows.
We start our construction with the \td~$(T, \cV)$ of the graph~$G$ from~\cref{maincor:TreeDecompDisplayingInfsTechnical}, and we then aim to refine its (finite) parts $G[V_t]$ via lean \td s $(T^t, \cV^t)$ given by the finite version of~\cref{thm:LeanTDTechnical}, that is Thomas's \cref{thm_intro_krizthomas2}.
If $V_t$ is a critical vertex set, then we may choose $(T^t, \cV^t)$ as the trivial \td\ into one bag.
In order to combine such refinements the \td s~$(T^t, \cV^t)$ along~$(T, \cV)$, we apply the finite result not to the parts themselves, but to the torsos.
Torsos, however, need not have the same tree-width as~$G$.
So our second ingredient to the proof of~\cref{thm:LeanTDTechnical} provides a sufficient condition on the separations induced by the edges incident with a node $t \in T$ which allows us to transfer the tree-width bound from~$G$ to the torso of $(T, \cV)$ at $t$ via \cite{SATangleTreeDualityInfGraphs}*{Corollary 6.3} (see \cref{lem:TorsosHaveSameTreeWidth} below).
Moreover, it also yields the well-linkedness of the separations on their left side (see \cref{lem:RobustImpliesWellLinked} below). This property is crucial to the proof of \cref{thm:LeanTDTechnical} as it ensures that the \td\ $(T', \cV')$ arising from~$(T, \cV)$ and the $(T^t, \cV^t)$ by refinement is lean:
First, if all separations induced by edges at some node~$t$ are left-well-linked, then we can transfer families of disjoint paths from the torso at~$t$ to $G$ via \cref{lem:DisjointPathsExtendedByWellLinked}~\cref{itm:DisjPathsExtended}, which yields the paths families required for lean between bags that belong to the same $(T^t, \cV^t)$.
Second, whenever the separation induced by an edge $e = t_0t_1$ of $T$ is well linked on both sides, this ensures that the leanness of the two $(T^{t_i}, \cV^{t_i})$ combines to the leanness of the \td\ resulting from gluing them together along $V_{e}$.
This will ensure that we obtain disjoint paths families, as required for lean, also between bags of $(T', \cV')$ that belong to distinct \td s $(T^t, \cV^t)$ and $(T^s, \cV^s)$.
Our third, and last, ingredient to the proof of \cref{thm:LeanTDTechnical}, then, is a pre-processing step: We first contract all edges which neither satisfy the well-linked condition on both sides nor are incident with a node whose bag is a critical vertex set.
This will ensure that the sufficient condition mentioned above is met by all separations induced by edges of $T$ that are incident with a node of $T$ whose torso we need to refine (recall that we do not need to refine those torsos whose bags are critical vertex sets).
Combining these three ingredients then yields a \td\ of~$G$, which we prove to be as desired.
\medskip

Following \cite{SATangleTreeDualityInfGraphs}, we call a finite-order separation $(A, B)$ of a graph $G$ \defn{left-$\ell$-robust} for~$\ell \in \N$ if there exist a set $U \subseteq A$ of size $\ell$ and a family~$\{P_x \mid x \in A \cap B\}$ of pairwise disjoint paths in $G[A]$ such that~$P_x$ ends in~$x$ and for each $x \in A \cap B$ there are~$\ell$ many $U$--$P_x$ paths in~$G[(A \setminus B) \cup \{x\}]$ that do not meet outside~$P_x$.
Analogously, $(A,B)$ is \defn{right-$\ell$-robust} for $\ell \in \N$ if $(B,A)$ is left-$\ell$-robust.
We call $(A,B)$ \defn{$\ell$-robust} if it is both left- and right-$\ell$-robust.

\begin{lemma}{\cite{SATangleTreeDualityInfGraphs}*{Corollary~6.3}} \label{lem:TorsosHaveSameTreeWidth}
    Let $G$ be a graph of tree-width at most $w \in \N$, and let~$\sigma$ be a finite star of separations of~$G$ of order at most~$w+1$ whose interior is finite.
    Suppose that all separations in~$\sigma$ are~left-$\ell$-robust for~$\ell = (w+1)^2(w+2)+w+1$.
    Then $\torsostar(\sigma)$ has tree-width at most $w$.
\end{lemma}

\begin{lemma} \label{lem:RobustImpliesWellLinked}
    If a separation of a graph~$G$ has order~$k$ and is~left-$(2k+1)$-robust, then it is left-well-linked.    
\end{lemma}
\begin{proof}
    Consider a~left-$(2k+1)$-robust separation~$(A, B)$ of~$G$, and let~$X, Y \subseteq A \cap B$ be disjoint.
    Suppose for a contradiction that there is no family of~$\min\{|X|, |Y|\}$ disjoint $X$--$Y$~paths through~$A \setminus B$ in $G$.
    By~Menger's theorem (see for example \cite{bibel}*{Proposition~8.4.1}), there then is an $X$--$Y$~separator $S$ of size less than~$\min\{|X|, |Y|\}$ in~$G[(A\setminus B) \cup X \cup Y]$.
    
    Now fix a set~$U$ and a path family~$\{P_x \mid x \in A \cap B\}$ which witness that~$(A, B)$ is~left-$(2k+1)$-robust, where~$k$ is the order of~$(A, B)$. 
    Since~$|S| < \min\{|X|, |Y|\}$ and the~$P_x$ are pairwise disjoint, there are~$x \in X$ and~$y \in Y$ such that~$P_x$ and~$P_y$ avoid~$S$.
    For~$z \in \{x,y\}$, at least~$k+1$ of the~$2k+1$ $P_z$--$U$~paths in~$G[(A \setminus B) \cup \{z\}]$ given by the left-$(2k+1)$-robustness avoid~$S$.
    So since $U$ has size~$2k+1$, there are such a~$P_x$--$U$~path $Q_x$ and such a~$P_y$--$U$~path $Q_y$ that both avoid~$S$ and end in the same vertex in~$U$.
    Hence, $P_x + Q_x + Q_y + P_y$ is a connected subgraph of~$G[(A\setminus B) \cup X \cup Y]$ which meets $X$ and~$Y$ but avoids~$S$. 
    This contradicts that $S$ is an $X$--$Y$~separator in~$G[(A\setminus B) \cup X \cup Y]$, which completes the proof.
\end{proof}

Let us now turn to the third ingredient for our proof of \cref{thm:LeanTDTechnical}: the pre-processing step in which we contract certain edges of the \td\ from \cref{maincor:TreeDecompDisplayingInfs}.
This ensures that the assumptions of both~\cref{lem:RobustImpliesWellLinked,lem:TorsosHaveSameTreeWidth} are met at all edges incident with nodes whose (finite) torso we later aim to refine using~Thomas's \cref{thm_intro_krizthomas2}. 

To simplify the wording, let us thus make the following definitions.
Given some $m \in \N_0$, we call a separation of~$G$ of order~$k$ \defn{left-$m$-good} if it is~$\ell$-left-robust for~$\ell := \max\{m, 2k+1\}$.
Analogously, $(A,B)$ is \defn{right-$m$-good} for $m \in \N$ if $(B,A)$ is left-$m$-good.
We call~$(A, B)$ \defn{$m$-good} if it is left- and right-$m$-good.
An left-$m$-good separation of~$G$ is \defn{left-good} if $m = (w+1)^2(w+2)+w+1$ for $G$ with $w = \tw(G) \in \N$ or $m = 0$ for graphs $G$ whose tree-width is not finitely bounded.
Analogously, we define \defn{right-good}.
We call~$(A, B)$ \defn{good} if it is both left- and right-good.
The bounds in the definition of (left-)good are exactly the ones sufficient to apply~\cref{lem:TorsosHaveSameTreeWidth,lem:RobustImpliesWellLinked} in the respective contexts.

Setting out from \cref{maincor:TreeDecompDisplayingInfsTechnical}, our pre-processing step yields the following result: 
\begin{lemma} \label{lem:MakeTDrobust}
    Let $G$ be a graph without half-grid minor, let $m \in \N$ and let $(T, \cV)$ be a fully tight, rooted \td\ into finite parts which displays the infinities and satisfies \cref{itm:TreeDecompDisplayingInfsTechnical:CofinComp} from \cref{maincor:TreeDecompDisplayingInfsTechnical}.
    Then the \td\ $(T', \cV')$ of $G$ induced by contracting every edge $st$ of~$T$ with $\deg(s), \deg(t) < \infty$ whose induced separation is not $m$-good has finite parts, is cofinally componental, and displays the infinities. 
    Moreover,
    \begin{enumerate}
        \item \label{itm:MakeTDrobust:goodatnoncritical}
        if $\deg(t) < \infty$ for $t \in T'$, then all separations in $\sigma'_t$ at $t$ are left-$m$-good. 
    \end{enumerate}
\end{lemma}

We first show three auxiliary lemmas.
 \begin{lemma} \label{claim:EndsYieldRobustness}
        Let $\eps$ be a finitely dominated end of a graph $G$, and $\ell \in \N$ arbitrary. For every collection~$R_1, \dots, R_d$ of disjoint rays in~$\eps$ that avoid~$\Dom(\eps)$, there exists a set~$U \subseteq V(R_1)$ of size~$\ell$ such that
        \begin{itemize}
            \item for every $i \in \{2, \ldots d\}$, there are $\ell$ pairwise disjoint $U$--$R_i$ paths~$P^i_1, \dots, P^i_\ell$ in $G$ avoiding~$\Dom(\eps)$, and
            \item for every $v \in \Dom(\eps)$, there are $\ell$ independent $U$--$v$ paths~$P^v_1, \dots, P^v_\ell$ in $G$ avoiding~$\Dom(\eps) \setminus \{v\}$ that only meet in $v$.
        \end{itemize}
    \end{lemma}
    \begin{proof}
        Since $R_1, \ldots, R_d$ belong to the same end and~$\Dom(\eps)$ is finite, there are, for every~$i \in \{2, \dots, d\}$, infinitely many pairwise disjoint $R_1$--$R_i$ paths~$Q^{i}_1, Q^i_2, \dots$ avoiding~$\Dom(\eps)$, and we write $q^{i}_k$ for the endvertex of~$Q^{i}_k$ in~$R_1$.
        Similarly, for every~$v \in \Dom(\eps)$ there are infinitely many $v$--$R_1$~paths $Q^v_1, Q^v_2, \dots$ in $G$ avoiding~$\Dom(\eps) \setminus \{v\}$ that only meet in~$v$, and we write $q^v_k$ for the endvertex of~$Q^v_k$ in~$R_1$.

        We now find the elements~$u_1, \dots, u_\ell$ of~$U$ one by one: 
        Let $u_1$ be the first vertex of $R_1$.
        Given~$u_1, \dots, u_{j-1}$ for some~$2 \leq j \le \ell$, we then choose~$u_j$ as the first vertex on~$R_1$ such that $u_{j-1}R_1u_j$ contains as inner vertices some~$q_i^k$ for every~$i \in \{2, \ldots, d\}$ and some~$q^k_v$ for every~$v \in \Dom(\eps)$.
        Then~$U := \{u_1, \ldots, u_\ell\}$ is as desired, as witnessed by the~$P^i_j := u_j R_1 q^i_k  Q^i_k$ and $P^v_j := u_j R_1  q^v_k Q^v_k$ for the respective~$k$ as given by the choice of~$u_j$.
    \end{proof}

\begin{lemma} \label{lem:ContractingEdgesEllRobust}
    Let $m \in \N$, let $(T, \cV)$ be a rooted \td\ of finite adhesion of a graph~$G$ without half-grid minor, and let $\eps$ be an end of~$G$. Assume that $\eps$ gives rise to a ray $R = r_0r_1\dots$ in~$T$ such that $\liminf_{e \in R} |V_e| = \Delta(\eps)$. Then cofinitely many edges $e$ of $R$ with $|V_e| = \Delta(\eps)$ induce $m$-good separations. 
\end{lemma}

\begin{proof}
    Since $G$ has no half-grid minor, the combined degree $\Delta(\eps)$ of $\eps$ is finite. 
    Thus, as $(T, \cV)$ has finite adhesion and $\eps$ gives rise to $R$, we have $\liminf_{e \in R} V_e \supseteq \Dom(\eps)$. In fact, $\liminf_{e \in R} V_e = \Dom(\eps)$ since $\liminf_{e \in R} |V_e| = \Delta(\eps)$, so the $V_e$ eventually have to meet each ray in a family of $\deg(\eps)$ disjoint $\eps$-rays avoiding $\Dom(\eps)$ precisely once.
    Hence, for some $N_0 \in \N$ all indices $i \geq N_0$ satisfy $\Dom(\eps) \subseteq V_{e_i}$ where $e_i := \{r_i, r_{i+1}\}$. Let us denote the set of all indices $i \geq N_0$ with $|V_{e_i}| = \Delta_G(\eps) =: k$ by~$I$. Note that $I$ is infinite since $\liminf_{e \in R} |V_{e}| = \Delta(\eps)$. 

    Write~$(A_i, B_i)$ for the separation induced by the edge~$\ve_i = (r_i, r_{i+1})$ for all~$i \in \N$.
    We now show that cofinitely many of the separations~$(A_i, B_i)$ with~$i \in I$ are $m$-good, which clearly implies the assertion.
    So let~$\ell := \max\{2k+1, (w+1)^2(w+2)+w+1\}$ if~$w := \tw(G) \in \N$, and let~$\ell := 2k+1$ if the tree-width of $G$ is not finitely bounded.

    Fix a collection~$R_1, \dots, R_d$ of~$d := \deg(\eps)$ disjoint rays in~$\eps$ that avoid~$\Dom(\eps)$, and consider a set~$U \subseteq V(R_1)$ and corresponding paths~$P^i_1, \dots, P^i_\ell$ and~$P^v_1, \dots, P^v_\ell$ as given by~\cref{claim:EndsYieldRobustness}.
    Let~$Z$ be the union of~$U$ and all the~$V(P^i_j)$ and~$V(P^v_j)$. Then $Z$ is finite. 
    So since~$\bigcap_{i \in \N} (B_i \setminus A_i) = \bigcap_{i \in \N} G \strictup e_i = \emptyset$ and~$\bigcap_{i \in \N} (A_i \cap B_i) = \liminf_{i \in \N} V_{e_i} = \Dom(\eps)$, there exists~$N_1 \in \N$ such that, for all~$i \ge N_1$, we have~$Z \subseteq (A_i \setminus B_i) \cup \Dom(\eps)$.
    Consider~$i \in I$ with~$i \ge N_1$. We claim that $(A_i, B_i)$ is left-$\ell$-robust. Indeed, we have~$|V_{e_i}| = |A_i \cap B_i| = \Delta_G(\eps)$ and~$\Dom(\eps) \subseteq A_i \cap B_i$.
    So the set $U$ and the initial segments $R_j \cap G[A_i]$ together with the trivial paths in $\Dom(\eps)$ are as required in the definition of left-$\ell$-robust, as witnessed by the paths $P^i_\ell, \dots, P^i_\ell$ and $P^v_1, \dots, P^v_\ell$ (and because $Z \subseteq (A_i\setminus B_i) \cup \Dom(\eps)$).
    Hence, $(A_i,B_i)$ is left-$\ell$-robust.
    To show that~$(A_i, B_i)$ is also~right-$\ell$-robust for~$i \in I$, we apply~\cref{claim:EndsYieldRobustness} in~$G[B_i]$ to the rays~$R_j \cap G[B_i]$ and take the~$P_x$ as suitable finite initial segments of those rays and the trivial paths on $\Dom(\eps)$.
    Altogether, we obtain that all~$(A_i, B_i)$ with~$i \in I$ and~$i \ge N_1$ are~$\ell$-robust and hence $m$-good.
\end{proof}

\begin{lemma} \label{lem:ContractingEdgesFiniteParts}
    Let $(T, \cV)$ be a tight, rooted \td\ of a graph $G$ into finite parts which displays the infinities of $G$. Assume that $F$ is some set of edges of $T$ which are not incident with any infinite-degree node of $T$ and such that, for every end $\eps$ of $G$, the set $F$ avoids cofinitely many edges $e$ of the arising ray $R_\eps$ in $T$ with $|V_e| = \Delta(\eps)$. Then the \td\ $(T', \cV')$ obtained from $(T, \cV)$ by contracting all edges in $F$ has finite parts and displays the infinities.
\end{lemma}

\begin{proof}
    We first note that $(T', \cV')$ displays the infinities:
    Since $(T, \cV)$ displays the ends homeomorphically, their combined degrees and their dominating vertices, so does $(T', \cV')$, since by the assumptions on $F$ every end $\eps$ of $G$ still gives rise to a ray in $T'$, and this ray still contains infinitely many edges $e$ with $|V_e| = \Delta(\eps)$.
    By assumption, $F$ contains no edges incident with nodes of infinite degree; thus, $(T', \cV')$ displays the critical vertex sets and their tight components cofinitely, as $(T, \cV)$ does so.

    To prove that all bags of $(T',\cV')$ are finite, it suffices to show by \cref{prop:RaylessToughGraphsAreFinite} that its torsos are tough and rayless.
    We first show that the torsos are tough.
    By construction of $(T', \cV')$, we did not contract edges of $T$ which are incident with some node $t \in T$  whose corresponding bag~$V_t$ is critical.
    In particular, for every $X \in \crit(G)$, the unique node $t_X \in T$ with $V_t = X$ is also in~$T'$ with $V'_{t_X} = X$, and cofinitely many tight components of $G-X$ are some $G \strictup e$ for an edge $e = t_Xt \in T'$ with $t_X <_{T'} t$, since $(T, \cV)$ displays all critical vertex sets and their tight components cofinitely.
    Thus, \cref{lem:DisplayingCritYieldsToughTorsos} ensures that all torsos of the tight rooted \td\ $(T', \cV')$ are tough.

    It remains to show that the torsos are rayless.
    Note that the torso of every node $t \in T'$ with $V'_t \in \crit(G)$ is rayless, as it is finite.
    Thus, it remains to consider $t \in T'$ with $V'_t \notin \crit(G)$.
    So suppose that there is such a node $t$ whose torso contains a ray $R'$.
    Since $(T', \cV')$ is tight, we may apply (the comment after) \cref{prop:OnePathOrRayExtended} to $R'$ and the star $\sigma'_t$ at $t$ to obtain a ray of $G$ that meets the part~$V'_t$ infinitely often, which contradicts the fact that $(T', \cV')$ displays the ends of~$G$.
\end{proof}

\begin{proof}[Proof of \cref{lem:MakeTDrobust}]
    Let $F$ be the set of edges in $T$ that we contracted. By \cref{lem:ContractingEdgesEllRobust}, the set $F$ satisfies the premise of \cref{lem:ContractingEdgesFiniteParts}, and hence $(T', \cV')$ has finite parts and displays the infinities.

    Further, $(T', \cV')$ is cofinally componental:
    Let $R'$ be a rooted ray in $T'$.
    It suffices to show that $G \strictup e$ is disconnected for at most finitely many consecutive edges $e$ of $T'$.
    Let $e =rs, f = st \in R'$ with $r <_{T'} s <_{T'} t$ be two successive edges such that $G \strictup e$ and $G\strictup f$ are disconnected.
    Since $(T', \cV')$ is obtained from $(T, \cV)$ by edge-contractions, $e$ and $f$ are also edges of $T$.
    It follows from \cref{itm:TreeDecompDisplayingInfsTechnical:CofinComp} from \cref{maincor:TreeDecompDisplayingInfsTechnical} of $(T, \cV)$ together with the construction of $(T', \cV')$ from $(T, \cV)$ that $s$ and $t$ were already nodes of $T$ and $V_t, V_s \in \crit(G)$ as well as $\deg(t), \deg(s) = \infty$ and $V_s \supseteq V_t$. Since $(T, \cV)$ displays the critical vertex sets of $G$, this implies that $V_s \neq V_t$, and thus $V_s \supsetneq V_t$.
    Hence, this can only happen finitely many times consecutively, as critical vertex sets are finite.

    It remains to show \cref{itm:MakeTDrobust:goodatnoncritical}. For this, let $t \in T'$ be a node with finite degree, and denote with $T_t$ the subtree $T_t$ of $T$ whose contraction yields~$t$.
    Suppose for a contradiction that some separation $(A,B) \in \sigma'_t$ is not left-$m$-good.
    We remark that as $(T', \cV')$ is obtained from $(T, \cV)$ by edge-contractions, the separations induced by the edges of $T$ with precisely one endvertex in~$T_t$ are the same as the separations induced by the edges incident with~$t$ in~$T'$, i.e.\, the ones in $\sigma'_t$; so let $e = t's$ be such an edge with $t' \in T_t$ which induces $(A,B)$.
    The construction of $T'$ yields that either $(A,B)$ is $m$-good or one of $t'$, $s$ has infinite degree in~$T$.
    If~$T_t$ is a singleton, then $\deg(t') < \infty$ by assumption on $t$.
    If $T_t$ contains at least one edge, then all nodes of $T_t$, in particular~$t'$, have finite degree, as the edges in $T_t$ have been contracted.
    In both cases, $\deg(s) = \infty$ and thus $V_{s} =: X \in \crit(G)$.
    Since $(T, \cV)$ displays the critical vertex sets and their tight components cofinitely and because $\deg(s) = \infty$, cofinitely, and thus infinitely, many of the tight components of $G-X$ are contained in $G[A]$. Since $X = V_s$ and thus $X \supseteq A \cap B$, this shows that $(A,B)$ is left-$m$-good
    as witnessed by the trivial paths in $A \cap B$ and a set $U$ consisting of $m$ vertices that lie in pairwise distinct tight components of $G-X$ contained in $G[A]$.
\end{proof}

With the three ingredients at hand, we are ready to prove the main result of this section. 

\begin{proof}[Proof of~\cref{thm:LeanTDTechnical}]
    Let $G$ be a graph without half-grid minor; in particular, $G$ has no $K^{\aleph_0}$ minor. 
    So $G$ has finite tree-width, as it has a normal spanning tree by \cite{halin78}.
    Let~$(T, \cV)$ be the (rooted) \td~of $G$ from \cref{maincor:TreeDecompDisplayingInfsTechnical}, and let~$(T', \cV')$ be the rooted \td\ obtained from~$(T, \cV)$ by applying~\cref{lem:MakeTDrobust} with $m = (w+1)^2(w+2)+w+1$ if $w := \tw(G) \in \N$ for finitely bounded tree-width $G$ or with $m = 0$ for other $G$.
    Then \cref{lem:MakeTDrobust} yields that all the bags of~$(T', \cV')$ are finite and that,
    moreover, for every node~$t \in T'$ either~$V'_{t} \in \crit(G)$ or each separation in the star~$\sigma'_{t}$ at $t$ is left-good.
    If~$V'_t$ is critical in~$G$, then~$|V'_t| \leq \tw(G)+1$ since critical vertex sets are infinitely connected and thus every~\td\ of~$G$, in particular those witnessing~$\tw(G)$, contains $V'_t$ in one of its bags.
    If~$V'_t$ is not critical in~$G$, then all separations in~$\sigma'_{t}$ are left-good, and we can apply~\cref{lem:TorsosHaveSameTreeWidth} to find that the torso of~$(T', \cV')$ at~$t$ has again at most the tree-width of $G$.
    Altogether, every torso of~$(T', \cV')$ is finite and has tree-width at most~$\tw(G)$.

    We may thus apply 
    Thomas's result \cite{LeanTreeDecompThomas}*{Theorem 5} (cf. \cref{thm_intro_krizthomas} for finite graphs) to the torsos of~$(T', \cV')$ and obtain, for every node~$t \in T'$, an (unrooted) lean \td~$(T^t, \cV^t)$ of~$\torsostar(\sigma_t)$ where~$T^t$ is a finite tree\footnote{We remark that this is not explicitly stated in~\cite{LeanTreeDecompThomas}*{Theorem~5} but follows directly from its proof.} and whose parts have size at most the tree-width of $G$; in particular, all bags in~$\cV^t$ are finite, as the torso at $t$ is finite.
    Furthermore, we may assume that for~$t \in T'$ with~$V'_t \in \crit(G)$, the \td~$(T^t, \cV^t)$ is the trivial~\td\ consisting of a single node-tree, since the torso of~$(T', \cV')$ at~$t$ is complete.
    
    Now for every edge~$e = st \in T'$, the adhesion set~$V'_e$ induces a complete subgraph in both~$\torsostar(\sigma'_s)$ and~$\torsostar(\sigma'_t)$.
    Thus, we may fix~$u \in T^s$ and~$w \in T^t$ with~$V'_e \subseteq V^s_u, V^t_w$ for every edge $e = st \in T'$.
    We now build a tree~$\tilde{T}$ from the disjoint union of the~$T^t$ by joining the corresponding~$u$ and~$w$ for every~$e \in E(T')$; we say that these new edges~$uv$ of~$\tilde{T}$ \defn{belong to}~$T'$ and \defn{correspond to} the respective $st \in T'$.
    Keeping the respective parts from the~$(T^t, \cV^t)$, we obtain a \td~$(\tilde{T}, \tilde{\cV})$ of~$G$ which has finite parts of size at most the tree-width of $G$.
    We remark that the separation induced by an edge $uv \in \tilde{T}$ which belongs to $T'$ is the same as the separation induced by the corresponding edge in $T'$.

    We claim that $(\tilde{T}, \tilde{\cV})$ is as desired. For this, let us first note that $(\tilde{T}, \tilde{\cV})$ is cofinally componental since $(T', \cV')$ is cofinally componental by \cref{lem:MakeTDrobust} and because the $T^t$ are finite.
    Moreover, since $(T', \cV')$ displays the infinities by~\cref{lem:MakeTDrobust}, it follows that $(\tilde{T}, \tilde{\cV})$ also does so, as the $T^t$ are finite and consist of a single node, if $V'_t \in \crit(G)$ and $\deg_{T'}(t) = \infty$.
    
    It remains to show that~$(\tilde{T}, \tilde{\cV})$ is lean.
    For this, fix nodes $t_1, t_2 \in \tilde{T}$ and sets $Z_1 \subseteq \tilde{V}_{t_1}$ and $Z_2 \subseteq \tilde{V}_{t_2}$ with $|Z_1| = |Z_2| =: k$.
    We prove the claim by induction on the number of edges on~$t_1 \tilde{T} t_2$ that belong to~$T'$.

    If there is no such edge, then there exists a node~$t \in T'$ with~$t_1, t_2 \in T^t$. 
    Since~$(T^t, \cV^t)$ is a lean \td\ of~$\torsostar(\sigma'_t)$, either there exists an edge~$e \in t_1 T^t t_2$ with~$|V^t_e| < k$ or there are~$k$ disjoint~$Z_1$--$Z_2$ paths in~$\torsostar(\sigma'_t)$.
    In the first case, the construction of~$(\tilde{T}, \tilde{V})$ yields~$t_1 \tilde{T} t_2 = t_1 T^t t_2$ and~$\tilde{V}_e = V^t_e$, so~$e$ is as desired.
    In the second case, we distinguish between the case whether $V'_t$ is critical or not.
    If $V'_t$ is critical, then $t_1 = t_2$ and $\tilde{V}_{t_1} = V'_t$ as $T^t$ has a single node by construction. We then use the infinitely many tight components of $G-V'_t$ to find the desired disjoint $Z_1$--$Z_2$ paths.
    And if $V'_t$ is not critical,
    then all separations in~$\sigma'_t$ are left-good by \cref{itm:MakeTDrobust:goodatnoncritical} from \cref{lem:MakeTDrobust} and because $(T', \cV')$ displays the critical vertex sets of $G$.
    Hence, they are left-well-linked by~\cref{lem:RobustImpliesWellLinked}. This allows us to apply~\cref{lem:DisjointPathsExtendedByWellLinked}~\ref{itm:DisjPathsExtended} to lift the~$k$ pairwise disjoint~$Z_1$--$Z_2$ paths in~$\torsostar(\sigma'_t)$ to~$G$. 

    Now suppose that there is an edge~$f = s_1 s_2$ on~$t_1 \tilde{T} t_2$ which belongs to~$T'$.
    We may assume by renaming that $s_1$ appears before $s_2$ on $t_1 \tilde{T} t_2$.
    If~$|\tilde{V}_f| < k$, then~$f$ is the desired edge $e$; so suppose otherwise.
    For~$i \in \{1, 2\}$, the induction hypothesis then yields either an edge~$e \in t_i \tilde{T} s_i$ with~$|\tilde{V}_e| < k$ or a family~$\cP_i$ of~$k$ pairwise disjoint~$Z_i$--$\tilde{V}_f$ paths in~$G$.
    The first case already yields the desired edge~$e$; so we may assume that the second case holds for both~$i \in \{1, 2\}$.

    If~$\tilde{V}_{s_1} \in \crit(G)$, \cref{lemma:LinkingPathsAlongCritVertexSet} yields the desired family of $k$ disjoint $Z_1$--$Z_2$ paths, as $\tilde{V}_f \subseteq \tilde{V}_{s_1}$.
    The same argument applies if $\tilde{V}_{s_2} \in \crit(G)$.

    Suppose now that~$\tilde{V}_{s_1}, \tilde{V}_{s_2} \notin \crit(G)$. 
    Since there are no~$k$ disjoint~$Z_1$--$Z_2$ paths in~$G$, Menger's theorem (see for example \cite{bibel}*{Proposition~8.4.1}) yields a separation~$(C, D)$ of~$G$ of order less than~$k$ with~$Z_1 \subseteq C$ and~$Z_2 \subseteq D$.
    We now use~$(C, D)$ to find a $Z_1$--$\tilde{V}_f$ or a~$Z_2$--$\tilde{V}_f$-separator of size less than~$k$, which gives the desired contradiction.

    Let~$(A, B)$ be the separation of~$G$ induced by~$\vf = (s_1, s_2)$, and let~$X := (A \cap B) \cap (C \setminus D)$ and~$Y := (A \cap B) \cap (D \setminus C)$.
    By symmetry on the assumptions up to this point, we may assume~$|Y| \le |X|$.
    Since~$\tilde{V}_{s_1} \notin \crit(G)$, $(A, B)$ is left-good by the construction of~$(\tilde{T}, \tilde{\cV})$ and \cref{itm:MakeTDrobust:goodatnoncritical} from \cref{lem:MakeTDrobust}.
    Hence, $(A,B)$ is left-well-linked by~\cref{lem:RobustImpliesWellLinked}.
    Thus, there exist~$|Y|$ pairwise disjoint~$X$--$Y$~paths through $A \setminus B$ in $G$.
    All these paths meet~$(C \cap D) \cap (A \setminus B)$, so~$|(C \cap D) \cap (A \setminus B)| \ge |Y|$.
    But then
    \begin{align*}
        |(A \cap C) \cap (B \cup D)|
            &= |(A \cap B) \cap (C \setminus D)| + |(A \cap B) \cap (C \cap D)| + |(C \cap D) \cap (A \setminus B)| \\
            &\ge |X| + |(A \cap B) \cap (C \cap D)| + |Y| = |A \cap B|,
    \end{align*}
    which in turn yields by double counting that
    \begin{equation*}
        |(A \cup C) \cap (B \cap D)| = |A \cap B| + |C \cap D| - |(A \cap C) \cap (B \cup D)| \le |C \cap D| < k.            
    \end{equation*}
    But~$(A \cup C) \cap (B \cap D)$ is a~$Z_2$--$\tilde{V}_f$-separator (or a $Z_1$--$\tilde{V}_f$ separator in the symmetric case), a contradiction.
\end{proof}

\section{Tree-decompositions distinguishing infinite tangles} \label{subsec:ToT}

An \defn{infinite tangle} of a graph $G$ is a set $\tau$ of (oriented) finite-order separations of $G$ such that
\begin{itemize}
    \item $\tau$ contains precisely one orientation of every finite-order separation of $G$, and
    \item there are no three (not necessarily distinct) separations $(A_1, B_1), (A_2, B_2), (A_3, B_3) \in \tau$ with $G[A_1] \cup G[A_2] \cup G[A_3] = G$.
\end{itemize}

A separation of $G$ \defn{distinguishes} two infinite tangles if they orient it differently. It distinguishes them \defn{efficiently} if they are not distinguished by any separation of smaller order. A \td\ $(T, \cV)$ of $G$ \defn{distinguishes} some set of infinite tangles of $G$ \defn{efficiently} if every two tangles in this set are distinguished efficiently by a separation that is induced by an edge of $(T, \cV)$.

By \cite{EndsAndTangles}*{Theorem 3}, for every infinite tangle $\tau$ of a graph $G$ there is either an end $\eps$ of $G$ such that a finite-order separation $(A,B)$ lies in $\tau$ if and only if $B$ contains a tail of every, or equivalently some, ray in $\eps$, or there is a critical vertex set $X$ of $G$ such that $(V(C) \cup X, V(G-C)) \in \tau$ for all $C \in \cC(G-X)$. If the first case holds for $\tau$, then we say that $\tau$ is \defn{induced} by $\eps$.

Following \cite{infinitetangles}, we say that two infinite tangles $\tau, \tau'$ of $G$ are \defn{combinatorially distinguishable} if at least one of them is induced by an end or there exists a finite set $X \subseteq V(G)$ such that $(V(C) \cup X, V(G-C)) \in \tau$ for all $C \in \cC_X$ and such that $(V(G-C), V(C) \cup X) \in \tau'$ for a component $C \in \cC_X$.

With the above characterization of infinite tangles the following observation is immediate:
\begin{observation} \label{obs:LeanDecompDistinguishesInfTangles}
    Every \td\ of a graph $G$ that displays its infinities distinguishes all combinatorially distinguishable infinite tangles of $G$. \qed
\end{observation}

We emphasise though that such a \td\ need not distinguish those infinite tangles \emph{efficiently}. In fact, we remark that the graph in \cite{ExamplesLinkedTDInfGraphs}*{\constructionExample} has finite tree-width and thus by \cref{maincor:TreeDecompDisplayingInfs} a \td\ that displays its infinities, but no \td\ of that graph efficiently distinguishes all its infinite tangles that are induced by ends \cite{ExamplesLinkedTDInfGraphs}*{\exampleNoTDEffDistAllEnds}.

However, Elm and Kurkofka \cite{infinitetangles}*{Theorem~1} showed that every graph $G$ has a nested set of separations that efficiently distinguishes all the combinatorially distinguishable infinite tangles of~$G$. 
They also showed that this is best possible in the following sense: every graph $G$ that has at least two combinatorially indistinguishable infinite tangles has no nested set of separations that efficiently distinguishes all infinite tangles of $G$~\cite{infinitetangles}*{Corollary~3.4}.
In the special case where~$G$ has no half-grid minor we obtain the following strengthening of their result:

\begin{customcor}{Corollary 5'}[Tangle version of \cref{main:ToTIntroVersion}]\label{thm:ToTtechnical}
    Every graph $G$ without half-grid minor has a \td\ that efficiently distinguishes all the combinatorially distinguishable infinite tangles of $G$.
\end{customcor}

\begin{proof} 
    We show that the \td\ $(T, \cV)$ from \cref{thm:LeanTDTechnical} is as desired. For this, let any pair $\tau_1, \tau_2$ of combinatorially distinguishable infinite tangle be given. Then by \cref{obs:LeanDecompDistinguishesInfTangles}, there is an edge $e \in E(T)$ such that the separation $\{A_e, B_e\}$ induced by $e$ distinguishes $\tau_1$ and~$\tau_2$ (though not necessarily efficiently). For $i = 1,2$, if there is a critical vertex set $X_i$ of $G$ such that $(V(C) \cup X_i, V(G-C)) \in \tau_i$ for all $C \in \cC(G-X_i)$, then there is unique infinite-degree node $t_i \in T$ with $V_{t_i} = X_i$, and we then let $f_i$ be the unique edge of the $t_iTe$ path incident with $t_i$. Otherwise, $\tau_i$ is induced by an end $\eps_i$ of $G$, and we then let $f_i$ be any edge on the unique ray $R$ of $T$ starting in $e$ and arising from $\eps_i$ such that $|V_f| \geq |V_{f_i}|$ for all edges $f >_T f_i$ on $R$. 

    Now using the fact that $(T, \cV)$ is lean, we find an edge $e'$ and a family $\{P_x \mid x \in V_{e'}\}$ of pairwise disjoint $V_{f_1}$--$V_{f_2}$ paths in $G$. Since for $i = 1,2$ either $V_{f_i}$ is linked to $\eps_i$ or $V_{f_i} = X_i$, these paths~$P_x$ can be extended to paths or rays witnessing that $\tau_1$ and $\tau_2$ cannot be distinguished by a separation of order less than $V_{e'}$, which shows that $(T, \cV)$ distinguishes $\tau_1$ and $\tau_2$ efficiently.
\end{proof}

\begin{proof}[Proof of \cref{main:ToTIntroVersion}]
    Apply \cref{thm:ToTtechnical}.
\end{proof}

\bibliographystyle{amsplain}
\arXivOrNot{\bibliography{collectivearXiv.bib}}{\bibliography{collective.bib}}

\arXivOrNot{
\appendix

\section{A closer look at \texorpdfstring{\cref{thm:critVtxTechnical}}{Theorem 6'}} \label{subsec:Counterexamples}

A theorem which is similar to \cref{thm:critVtxTechnical:copy} was proven by Elm and Kurkofka \cite{infinitetangles}*{Theorem~2}:
\begin{theorem}\label{thm:AKJan}
    Every connected graph $G$ has a nested\footnote{A set $N$ of separations is \defn{nested} if for every two separations $\{A, B\}, \{C,D\}$ one of $A \subseteq C$ \& $B \supseteq D$, $A \subseteq D$ \& $B \supseteq C$, $B \subseteq C$ \& $A \supseteq D$ and $B \subseteq D$ \& $A \supseteq C$ holds.} set of separations whose separators are precisely the critical vertex sets of $G$ and all whose torsos\footnote{We refer the reader to \cite{infinitetangles}*{Section~2.6} for a definition of \defn{torsos} in this context. We remark that if the nested set is the set of induced separations of some \td, then these torso are precisely the torsos at nodes of the decomposition tree.} are tough.
\end{theorem}

In the remainder of this section we compare \cref{thm:critVtxTechnical:copy} and \cref{thm:AKJan}. 
For this, let us first note that the set of separations induced by a \td\ is always nested, so for graphs of finite tree-width, our \cref{thm:critVtxTechnical:copy} immediately implies \cref{thm:AKJan} from Elm and Kurkofka.
However, not every nested set of separations is induced by a \td. 
In fact, in \cite{infinitetangles}*{Example 5.12} Elm and Kurkofka describe a graph which does not admit a \td\ whose induced nested of separation is of their form\footnote{Moreover, one can show that their example does not even admit a \td\ whose adhesion sets are the critical vertex sets of $G$ and whose torsos are tough.}. We remark that this graph does not have finite tree-width.

But also for graphs of finite tree-width it is not only not obvious that their result does not yield a \td\ but also not true: We present in the following \cref{example:AKandJandonotinduceTD} a graph $G$ of finite tree-width such that no \td\ of $G$ induces their constructed nested set of separations.
For this let us recall the nested set that Elm and Kurkofka constructed in more detail. Their nested set \cite{infinitetangles}*{Proof of Theorem~2, i.e. Theorem~5.10 \& Theorem~5.11} is of the form 
\begin{equation} \label{eq:AKandJansNestedSet}
\{\{V(G) \setminus \bigcup \cK(X), X \cup \bigcup \cK(X) \} \mid X \in \crit(G) \} \cup \{\{V(G) \setminus C, X \cup C \} \mid X \in \crit(G), C \in \cK(X) \}
\end{equation}
for some choice of $\cK(X) \subseteq \breve{\cC}(G-X)$ containing all but at most one element of $\breve{\cC}(G-X)$.

The nested set induced by the \td\ $(T', \cV')$ from \cref{thm:critVtxTechnical:copy} is also of the form as in \cref{eq:AKandJansNestedSet} but for some choice of $\cK(X) \subseteq \cC(G-X)$ consisting of cofinitely many (instead of one) elements of $\breve{\cC}(G-X)$. 
The following example shows that there are graphs of finite tree-width that have no \td\ inducing a nested set as in \cref{eq:AKandJansNestedSet} for some choice of $\cK(X)$ containing all but at most one element of $\cC(G-X)$. In particular, we cannot strengthen the definition of `displaying the tight components of all critical vertex sets cofinitely' by dropping the `cofinitely'. 

\begin{example}\label{example:AKandJandonotinduceTD}
    There is a graph $G$ of finite tree-width such that no \td\ of $G$ induces a set of separations of the form as in \eqref{eq:AKandJansNestedSet} for a choice of $\cK(X) \subseteq \breve{\cC}(G-X)$ containing all but at most one element of $\breve{\cC}(G-X)$.
\end{example}

\begin{proof}
    Let $R = dr_0r_1, \dots$ be a ray, and let $H$ be the graph obtained from $R$ by first adding the edges $dr_i$ for all $i \in 2\N$ and the edges $r_ir_{i+2}$ and $r_ir_{i+3}$ for all $i \in 2\N+1$, and then adding vertices~$u_{ij}$ for $i \in 2\N \setminus \{0\}$ and $j \in \N$ and joining them to $R$ with edges $r_iu_{ij}$. To obtain $G$, we further add the edges
    \[
    \{u_{ij}r_{i-1}, u_{ij}r_k \mid i \in 2\N\setminus \{0\}, j \in \N, k \in 2\N \cap \{0, 1, \dots, i-2\}\}
    \]
    (see \cref{fig:ExampleForCritDecomp1}).
    Note that $G$ has a normal spanning tree as depicted in blue in \cref{fig:ExampleForCritDecomp1}.

    \begin{figure}[ht]
        \centering
        \includegraphics[width=1\columnwidth]{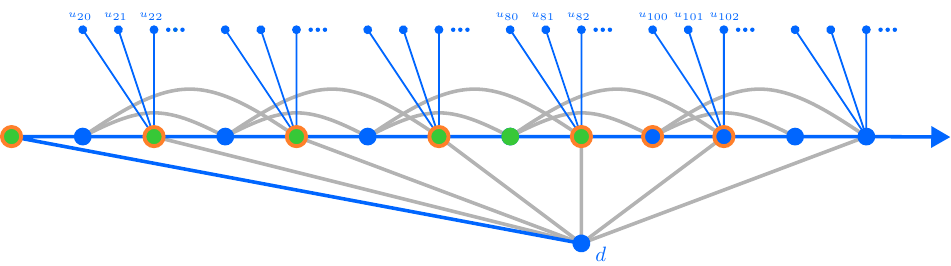}
        \caption{Depicted is the graph $H$. The green (orange) vertices are precisely the neighbours in $G$ of $u_{8j}$ ($u_{10j}$) for all $j \in \N$. The blue edges induce a normal spanning tree of $G$ with root $d$.}
        \label{fig:ExampleForCritDecomp1}
    \end{figure}
    
    We claim that $G$ is as desired, i.e.\ that no \td\ of $G$ induces a set of separations of the form as in \cref{eq:AKandJansNestedSet} for a choice of $\cK(X) \subseteq \breve{\cC}(G-X)$ containing all but at most one element of~$\breve{\cC}(G-X)$. For this, let us first observe that all critical vertex sets of $G$ are of the form 
    \[
    X_i := \{r_j \mid j \in 2\N, j \leq i\} \cup \{r_{i-1}\},
    \]
    and the components of $G - X_i$ either consist of a single vertex $u_{ij}$ or are of the form 
    \begin{align*}
        C_i &:= G[\{r_j \mid j \in 2\N+1, j \leq i-2\} \cup \{u_{jk} \mid j < i, k \in \N\}\}]  \text{ or}\\
        C'_i &:= G[\{r_j \mid j \geq i\} \cup \{u_{jk} \mid j > i, k \in \N\} \cup \{d\}]. 
    \end{align*}
    In particular, for all $X_i$ all components of $G-X_i$ are tight (cf.\ \cref{fig:ExampleForCritDecomp2}). 
    \begin{figure}[ht]
        \centering
        \includegraphics[width=1\columnwidth]{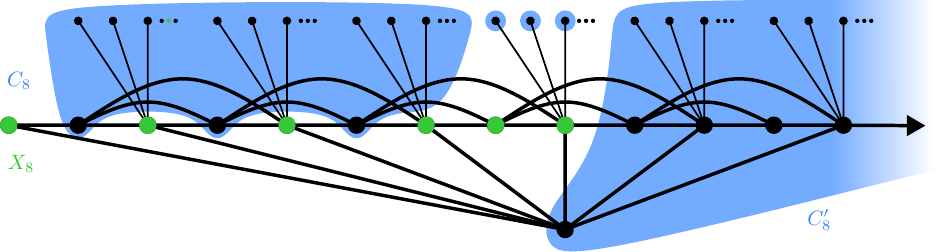}
        \caption{Indicated in green is the critical vertex set $X_{8}$ of $G$. The components of $G-X_{8}$ are indicated in blue.}
        \label{fig:ExampleForCritDecomp2}
    \end{figure}
    Hence, the unique nested set that is of the form as in \cref{eq:AKandJansNestedSet} for a choice of $\cK(X_i) \subseteq \breve{\cC}(G-X_i)$ containing all but at most one element of $\breve{\cC}(G-X_i)$ is the set
    \[
    N := \{\{V(C) \cup X_i, V(G-C)\} \mid X_i \in \crit(G), C \in \cC(G-X_i)\}.
    \]
    But $N$ is not induced by any \td\ of $G$. Indeed, suppose for a contradiction that $N$ is induced by a \td\ $(T, \cV)$ of $G$, and consider the separations $\{V(G - C'_i), V(C'_i) \cup X_i\}$. 
    Since $C'_i \supseteq C'_j$ for all $i \leq j$, it is easy to see that these separations have to be induced by edges $e_i$ of $T$ that all lie on a common ray in $T$. In particular, there is some $j \in 2\N$ such that the $e_i$ all lie on a common rooted ray in $T$ and that $G\strictup e_i = V(C'_i)$. But since $d \in C'_i$ for all $i \in 2\N\setminus \{0\}$, one can conclude that $d$ does not lie in any bag of $(T, \cV)$, which contradicts property \cref{prop:TD1} of \td s. 
\end{proof}
}
{}

\end{document}

%% file: Preamble.tex
\makeatletter
\let\origsection=\section \def\section{\@ifstar{\origsection*}{\mysection}} 
\def\mysection{\@startsection{section}{1}\z@{.7\linespacing\@plus\linespacing}{.5\linespacing}{\normalfont\scshape\centering\S}}
\makeatother  

\usepackage{amsmath}
\usepackage{amssymb}
\usepackage{amsthm}
\usepackage{mathtools}
\usepackage{letterswitharrows}
\usepackage{tikz-cd}
\usepackage[only,llbracket,rrbracket]{stmaryrd}
\usepackage[linesnumbered,ruled,vlined]{algorithm2e}
\usepackage{enumitem}
\setenumerate{label={\normalfont (\roman*)}}

\usepackage[utf8]{inputenc}
\usepackage[T1]{fontenc}
\usepackage{lmodern}
\usepackage[babel]{microtype}
\usepackage[english]{babel}
\usepackage{relsize}

\usepackage{graphicx}
\usepackage{subcaption}

\usepackage{svg}

\usepackage{geometry}
\usepackage{xcolor} 
\colorlet{darkishRed}{red!80!black}
\colorlet{darkishBlue}{blue!60!black}
\colorlet{darkishGreen}{green!60!black}
\usepackage{hyperref}
\hypersetup{
	colorlinks,
	linkcolor={red!60!black},
	citecolor={green!60!black},
	urlcolor={blue!60!black},
}
\usepackage[abbrev, msc-links]{amsrefs}
\usepackage[nameinlink, capitalise, noabbrev]{cleveref}
\crefformat{enumi}{#2#1#3}
\crefformat{equation}{#2(#1)#3}
\crefname{mainresult}{Theorem}{Theorems}
\usepackage{doi}

\renewcommand{\PrintDOI}[1]{\doi{#1}}

\let\setminus=\smallsetminus

\renewcommand{\subset}{\subseteq}

\renewcommand{\leq}{\leqslant}
\renewcommand{\geq}{\geqslant}
\renewcommand{\ge}{\geq}
\renewcommand{\le}{\leq}

\newtheorem{theorem}{Theorem}[section] 
\newtheorem{proposition}[theorem]{Proposition}
\newtheorem{corollary}[theorem]{Corollary}
\newtheorem{lemma}[theorem]{Lemma}

\newtheorem{observation}[theorem]{Observation}
\newtheorem{conjecture}[theorem]{Conjecture}

\newtheorem{mainresult}{Theorem} 
\newtheorem{maincorollary}[mainresult]{Corollary}
\newtheorem{innercustomthm}{}

\newtheorem{innercustomcor}{}

\newenvironment{customthm}[1]
  {\renewcommand\theinnercustomthm{#1}\innercustomthm}
  {\endinnercustomthm}

\newenvironment{customcor}[1]
  {\renewcommand\theinnercustomcor{#1}\innercustomcor}
  {\endinnercustomcor}

\newcounter{claimcounter}[theorem]
\newtheorem*{claim*}{Claim}

\newcounter{subclaimcounter}[claimcounter]
\newtheorem*{subclaim*}{Subclaim}

\theoremstyle{definition}

\crefname{mainexample}{Example}{Examples}
\newtheorem{example}[theorem]{Example}
\crefname{example}{Example}{Examples}

\newtheorem{construction}[theorem]{Construction}

\newtheorem{algo}[theorem]{Algorithm}

\crefname{routine}{Routine}{Routines}

\crefname{subroutine}{Subroutine}{Subroutines}

\crefname{subsubroutine}{Subsubroutine}{Subsubroutines}

\crefname{step}{Step}{Steps}
\theoremstyle{remark}

\newcommand{\COMMENT}[1]{{}}

\let\eps=\varepsilon
\let\epsilon=\varepsilon
\let\theta=\vartheta
\let\rho=\varrho
\let\phi=\varphi
\def\N{\mathbb N}

\makeatletter

\def\calCommandfactory#1{%
  \expandafter\def\csname c#1\endcsname{\mathcal{#1}}}
\def\frakCommandfactory#1{%
  \expandafter\def\csname frak#1\endcsname{\mathfrak{#1}}}
   
\newcounter{ctr}
\loop
  \stepcounter{ctr}
  \edef\X{\@Alph\c@ctr}
  \expandafter\calCommandfactory\X
  \expandafter\frakCommandfactory\X
\ifnum\thectr<26
\repeat

\usepackage{etoolbox}

\newbool{arXiv}
\booltrue{arXiv} 

\newcommand{\arXivOrNot}[2]{\ifbool{arXiv}{{#1}}{{#2}}}

%% file: specificPreamble.tex
\newcommand{\lk}{({<}\,k)}
\newcommand{\lell}{({<}\,\ell)}
\newcommand{\lek}{({\le}\,k)}
\newcommand{\leell}{({\le}\,\ell)}
\newcommand{\gek}{({\ge}\,k)}

\newcommand{\lA}{({<}\,\aleph_0)}
\newcommand{\td}{tree-decom\-posi\-tion}

\newcommand{\down}{{\downarrow}}
\newcommand{\up}{{\uparrow}}
\newcommand{\strictdown}{ \mathring{\downarrow} }
\newcommand{\strictup}{ \mathring{\uparrow} }

\DeclareMathOperator{\Dom}{Dom}

\DeclareMathOperator{\rt}{root}
\DeclareMathOperator{\dom}{dom}
\DeclareMathOperator{\interior}{int}
\newcommand{\size}{|\!\cdot\!|}
\DeclareMathOperator{\torsostar}{torso}

\DeclareMathOperator{\crit}{crit}
\DeclareMathOperator{\tw}{tw}

\newcounter{mylabelcounter}

\makeatletter
\newcommand{\labelText}[2]{%
#1\refstepcounter{mylabelcounter}%
\immediate\write\@auxout{%
  \string\newlabel{#2}{{1}{\thepage}{{\unexpanded{#1}}}{mylabelcounter.\number\value{mylabelcounter}}{}}%
}%
}

\makeatletter
\newcommand\footnoteref[1]{\protected@xdef\@thefnmark{\ref{#1}}\@footnotemark}
\makeatother

%% file: main.bbl
\begin{bibdiv}
\begin{biblist}

\bib{SATangleTreeDualityInfGraphs}{article}{
      author={Albrechtsen, S.},
       title={Tangle-tree duality in infinite graphs},
        date={2024},
     journal={arXiv preprint},
}

\bib{ExamplesLinkedTDInfGraphs}{article}{
      author={Albrechtsen, S.},
      author={Jacobs, R.~W.},
      author={Knappe, P.},
      author={Pitz, M.},
       title={Counterexamples regarding linked and lean tree-decompositions of
  infinite graphs},
        date={2024},
     journal={arXiv preprint},
}

\bib{bellenbaum2002two}{article}{
      author={Bellenbaum, Patrick},
      author={Diestel, Reinhard},
       title={Two short proofs concerning tree-decompositions},
        date={2002},
     journal={Combinatorics, Probability and Computing},
      volume={11},
      number={6},
       pages={541\ndash 547},
}

\bib{carmesin2019displayingtopends}{article}{
      author={Carmesin, Johannes},
       title={All graphs have tree-decompositions displaying their topological
  ends},
        date={2019},
     journal={Combinatorica},
      volume={39},
      number={3},
       pages={545\ndash 596},
      eprint={1409.6640},
}

\bib{carmesin2022canonical}{article}{
      author={Carmesin, Johannes},
      author={Hamann, Matthias},
      author={Miraftab, Babak},
       title={Canonical trees of tree-decompositions},
        date={2022},
     journal={Journal of Combinatorial Theory, Series B},
      volume={152},
       pages={1\ndash 26},
      eprint={2002.12030},
}

\bib{diestel1994depth}{article}{
      author={Diestel, Reinhard},
       title={The depth-first search tree structure of ${TK}^{\aleph_0}$-free
  graphs},
        date={1994},
     journal={Journal of Combinatorial Theory, Series B},
      volume={61},
      number={2},
       pages={260\ndash 262},
}

\bib{EndsAndTangles}{article}{
      author={Diestel, Reinhard},
       title={Ends and tangles},
        date={2017},
     journal={Abhandlungen aus dem Mathematischen Seminar der Universität
  Hamburg},
      volume={87{\rm , Special issue in memory of Rudolf Halin}},
       pages={~223\ndash 244},
      eprint={1510.04050},
}

\bib{bibel}{book}{
      author={Diestel, Reinhard},
       title={{Graph Theory}},
     edition={5},
   publisher={Springer},
        date={2017},
}

\bib{diestel2001normal}{article}{
      author={Diestel, Reinhard},
      author={Leader, Imre},
       title={Normal spanning trees, {A}ronszajn trees and excluded minors},
        date={2001},
     journal={Journal of the London Mathematical Society},
      volume={63},
      number={1},
       pages={16\ndash 32},
}

\bib{infinitesplinter}{article}{
      author={Elbracht, Christian},
      author={Kneip, Jakob},
      author={Teegen, Maximilian},
       title={Trees of tangles in infinite separation systems},
organization={Cambridge University Press},
        date={2022},
     journal={Mathematical Proceedings of the Cambridge Philosophical Society},
      volume={173},
      number={2},
       pages={297\ndash 327},
      eprint={2005.12122},
}

\bib{infinitetangles}{article}{
      author={Elm, Ann-Kathrin},
      author={Kurkofka, Jan},
       title={A tree-of-tangles theorem for infinite tangles},
organization={Springer},
        date={2022},
     journal={Abhandlungen aus dem Mathematischen Seminar der Universit{\"a}t
  Hamburg},
      volume={92},
      number={2},
       pages={139\ndash 178},
      eprint={2003.02535},
}

\bib{erde2018unified}{article}{
      author={Erde, Joshua},
       title={A unified treatment of linked and lean tree-decompositions},
        date={2018},
     journal={Journal of Combinatorial Theory, Series B},
      volume={130},
       pages={114\ndash 143},
      eprint={1703.03756},
}

\bib{enddefiningsequences}{article}{
      author={Gollin, J~Pascal},
      author={Heuer, Karl},
       title={Characterising k-connected sets in infinite graphs},
        date={2022},
     journal={Journal of Combinatorial Theory, Series B},
      volume={157},
       pages={451\ndash 499},
      eprint={1811.06411},
}

\bib{halin65}{article}{
      author={Halin, Rudolf},
       title={{\"U}ber die {M}aximalzahl fremder unendlicher {W}ege in
  {G}raphen},
        date={1965},
     journal={Mathematische Nachrichten},
      volume={30},
       pages={63\ndash 85},
}

\bib{halin1975chainlike}{article}{
      author={Halin, Rudolf},
       title={Some path problems in graph theory},
organization={Springer},
        date={1975},
     journal={Abhandlungen aus dem {M}athematischen {S}eminar der
  {U}niversit{\"a}t {H}amburg},
      volume={44},
      number={1},
       pages={175\ndash 186},
}

\bib{halin1977systeme}{article}{
      author={Halin, Rudolf},
       title={Systeme disjunkter unendlicher {W}ege in {G}raphen},
        date={1977},
     journal={Numerische Methoden bei Optimierungsaufgaben Band 3: Optimierung
  bei graphentheoretischen und ganzzahligen Problemen},
       pages={55\ndash 67},
}

\bib{halin78}{article}{
      author={Halin, Rudolf},
       title={Simplicial decompositions of infinite graphs},
        date={1978},
     journal={Annals of Discrete Mathematics},
      volume={3},
       pages={93–109},
}

\bib{jacobs2023efficiently}{article}{
      author={Jacobs, Raphael~W},
      author={Knappe, Paul},
       title={Efficiently distinguishing all tangles in locally finite graphs},
        date={2024},
     journal={Journal of Combinatorial Theory, Series B},
      volume={167},
       pages={189\ndash 214},
      eprint={2303.09332},
}

\bib{jung1969wurzelbaume}{article}{
      author={Jung, H.A.},
       title={{Wurzelb{\"a}ume und unendliche {W}ege in {G}raphen}},
        date={1969},
     journal={Mathematische Nachrichten},
      volume={41},
       pages={1\ndash 22},
}

\bib{koloschin2023end}{article}{
      author={Koloschin, Marcel},
      author={Krill, Thilo},
      author={Pitz, Max},
       title={End spaces and tree-decompositions},
        date={2023},
     journal={Journal of Combinatorial Theory, Series B},
      volume={161},
       pages={147\ndash 179},
      eprint={2205.09865},
}

\bib{komjath1995note}{article}{
      author={Komj{\'a}th, P{\'e}ter},
       title={A note on minors of uncountable graphs},
        date={1995},
     journal={Mathematical Proceedings of the Cambridge Philosophical Society},
      volume={117},
      number={1},
       pages={7\ndash 9},
}

\bib{kriz1991mengerlikepropertytreewidth}{article}{
      author={K{\v r}{\'i}{\v z}, Igor},
      author={Thomas, Robin},
       title={The menger-like property of the tree-width of infinite graphs},
        date={1991},
        ISSN={00958956},
     journal={Journal of Combinatorial Theory, Series B},
      volume={52},
      number={1},
       pages={86\ndash 91},
}

\bib{kurkofka2022strengthening}{article}{
      author={Kurkofka, Jan},
      author={Melcher, Ruben},
      author={Pitz, Max},
       title={A strengthening of {H}alin's grid theorem},
        date={2022},
     journal={Mathematika},
      volume={68},
      number={4},
       pages={1009\ndash 1013},
      eprint={2104.10672},
}

\bib{pitz2020unified}{article}{
      author={Pitz, Max},
       title={A unified existence theorem for normal spanning trees},
        date={2020},
     journal={Journal of Combinatorial Theory, Series B},
      volume={145},
       pages={466\ndash 469},
      eprint={2003.11575},
}

\bib{pitz2021proof}{article}{
      author={Pitz, Max},
       title={Proof of {H}alin’s normal spanning tree conjecture},
        date={2021},
     journal={Israel Journal of Mathematics},
      volume={246},
      number={1},
       pages={353\ndash 370},
      eprint={2005.02833},
}

\bib{pitz2023note}{article}{
      author={Pitz, Max},
       title={A note on minor antichains of uncountable graphs},
        date={2023},
     journal={Journal of Graph Theory},
      volume={102},
      number={3},
       pages={552\ndash 555},
      eprint={2005.05816},
}

\bib{PolatEME1}{article}{
      author={Polat, N.},
       title={Ends and multi-endings {I}},
        date={1996},
     journal={Journal of Combinatorial Theory, Series B},
      volume={67},
       pages={86\ndash 110},
}

\bib{GMIV}{article}{
      author={Robertson, N.},
      author={Seymour, P.D.},
       title={Graph minors. {IV}. {T}ree-width and well-quasi-ordering},
        date={1990},
     journal={Journal of Combinatorial Theory, Series~B},
      volume={48},
       pages={227\ndash 254},
}

\bib{GMX}{article}{
      author={Robertson, N.},
      author={Seymour, P.D.},
       title={Graph minors. {X}. {O}bstructions to tree-decomposition},
        date={1991},
     journal={Journal of Combinatorial Theory, Series~B},
      volume={52},
       pages={153\ndash 190},
}

\bib{robertson1995excluding}{article}{
      author={Robertson, Neil},
      author={Seymour, Paul~D},
      author={Thomas, Robin},
       title={Excluding infinite clique minors},
        date={1995},
     journal={Memoirs of the American Mathematical Society},
      volume={118},
      number={566},
       pages={vi+103},
}

\bib{seymour1993binarytree}{article}{
      author={Seymour, Paul~D},
      author={Thomas, Robin},
       title={Excluding infinite trees},
        date={1993},
     journal={Transactions of the American Mathematical Society},
      volume={335},
      number={2},
       pages={597\ndash 630},
}

\bib{thomas1988counter}{article}{
      author={Thomas, Robin},
       title={A counter-example to `{W}agner's conjecture' for infinite
  graphs},
        date={1988},
     journal={Mathematical Proceedings of the Cambridge Philosophical Society},
      volume={103},
      number={1},
       pages={55\ndash 57},
}

\bib{thomas1989wqo}{article}{
      author={Thomas, Robin},
       title={Well-quasi-ordering infinite graphs with forbidden finite planar
  minor},
        date={1989},
     journal={Transactions of the American Mathematical Society},
      volume={312},
      number={1},
       pages={279\ndash 313},
}

\bib{LeanTreeDecompThomas}{article}{
      author={Thomas, Robin},
       title={A menger-like property of tree-width: The finite case},
        date={1990},
     journal={Journal of Combinatorial Theory, Series B},
      volume={48},
       pages={67\ndash 76},
}

\end{biblist}
\end{bibdiv}
